\documentclass[12pt,twoside,a4paper]{article}
\usepackage{amssymb,graphicx,amsmath,amsthm, bm,mathtools,longtable,array,booktabs}
\usepackage[usenames,dvipsnames]{xcolor}
\usepackage[colorlinks = true,
            linkcolor = blue,
            urlcolor  = blue,
            citecolor = blue,
            anchorcolor = blue]{hyperref}
 
\usepackage{enumitem}  
\usepackage{wrapfig}
\usepackage{cite} 
\usepackage[T1]{fontenc}  
 \usepackage[normalem]{ulem}
 
\usepackage{diagbox}

\numberwithin{equation}{section}

\newcommand{\changeFirstReferee}[1]{{\color{black}#1}}
\newcommand{\changeSecondReferee}[1]{{\color{black}#1}}
\newcommand{\changeFirstRefereeG}[1]{{\color{black}#1}}
\newcommand{\changeSecondRefereeG}[1]{{\color{black}#1}}
\newcommand{\changeV}[1]{{\color{Maroon}#1}}

\makeatletter
\def\twodigits#1{\expandafter\@twodigits\csname c@#1\endcsname}
\def\@twodigits#1{%
  \ifnum#1<10 0\fi
  \number#1}
\makeatother
\AddEnumerateCounter{\twodigits}{\@twodigits}{10}

\usepackage{amsmath}

\usepackage[a4paper]{geometry}
\newcommand{\KSigmaD}{\mathrm{K}_{D \Sigma } }

\newcommand{\KSigmaOmegaTwo}{\mathrm{K}_{\Omega_{2} \Sigma} }
\newcommand{\KSigmaGamma}{\mathrm{K}_{\Gamma\Sigma} }
\newcommand{\KGammaSigma}{\mathrm{K}_{\Sigma\Gamma} }
\newcommand{\KGammaOmegaOneC}{\mathrm{K}_{\Omega_1^{\rm c}\Gamma} }

\newcommand{\RGammaSigma}{\mathrm{R}_{\Sigma\Gamma}}

\newcommand{\x}{\boldsymbol{x}}
\newcommand{\dd}{\boldsymbol{d}}
\newcommand{\z}{\boldsymbol{z}}
\newcommand{\sphere}{\mathbb{S}^{1}}
\newcommand{\ddh}{\widehat{\dd}}
\newcommand{\xh}{\widehat{\x}}
\newcommand{\zh}{\widehat{\z}}
 
\theoremstyle{remark}
\newtheorem{remark}{Remark}[section]
\theoremstyle{plain}
\newtheorem{lemma}[remark]{Lemma}
\newtheorem{theorem}[remark]{Theorem}
\newtheorem{proposition}[remark]{Proposition}
\newtheorem{corollary}[remark]{Corollary}

\title{Analysis and application of an overlapped FEM-BEM  for wave propagation  in unbounded and heterogeneous media}
\author{V. Dom\'{\i}nguez \thanks{Department of  Estad\'{\i}stica, Inform\'atica y Matem\'aticas, Universidad P\'{u}blica de Navarra, Tudela, Spain/ Institute for Advanced Materials (INAMAT), Pamplona, Spain. Email: victor.dominguez@unavarra.es}
\and M. Ganesh \thanks{Department Applied Mathematics and Statistics Department, Colorado School of Mines, Golden, CO, USA. Email: mganesh@mines.edu}
}
\date{\today}

\DeclareMathOperator{\spann}{span}

\begin{document}
\maketitle 
\begin{center}
\sl Dedicated to the memory of Francisco-Javier ``Pancho'' Sayas (1968-2019).
\end{center}

\begin{abstract}
An overlapped continuous model framework, for the Helmholtz wave propagation problem in unbounded regions  comprising  bounded heterogeneous
media, was recently introduced and analyzed by the authors  ({\tt J. Comput. Phys., {\bf 403},  109052, 2020}). The continuous Helmholtz system incorporates a radiation condition (RC) and our equivalent  hybrid framework facilitates  application of widely used 
finite element methods (FEM)  and boundary element methods (BEM), and the resulting discrete systems retain the RC exactly. The  FEM and BEM discretizations, respectively,  applied to  the designed interior heterogeneous and exterior homogeneous media
Helmholtz systems include the FEM and BEM solutions matching in artificial interface domains, and allow for computations of the exact ansatz based far-fields. In this article  we present rigorous numerical analysis of  a discrete two-dimensional FEM-BEM overlapped coupling implementation of the algorithm. We also demonstrate the efficiency of our discrete FEM-BEM framework  and  analysis using numerical experiments, including applications  to non-convex heterogeneous multiple particle Janus  configurations.  
Simulations of  the far-field induced differential scattering cross sections (DSCS) of heterogeneous  configurations and 
orientation-averaged (OA) counterparts  are important for several applications, including inverse wave problems. Our robust FEM-BEM framework \changeSecondReferee{facilitates} computations of such quantities of interest, without boundedness or  homogeneity or shape restrictions on the wave propagation model.
\end{abstract}

\vspace{0.4in}

\vspace{0.4in}

\newpage
\section{Introduction}\label{sec:intro}
Simulation of  scattered acoustic and electromagnetic fields, and hence understanding the impact of refractive indices of  wave propagation media,
are crucial for a large class of applications~\cite{colton:inverse,Ihlenburg:1998,Ned:2001,Jan_sph_20_book}. The term 
{\em Janus particles}  was mentioned in a Nobel Prize lecture~\cite{jan_lect} about three decades ago, and since then additional  interests include understanding the effect of a class of
piecewise-continuous heterogeneous  refractive indices induced Janus configurations.  A simple Janus particle 
is designed  by combining two distinct homogeneous refractive indices, and  Janus configurations in general  may comprise multiple particles or complex structures with 
heterogeneous material properties.  
The example configuration $\Omega_0$, of the type  illustrated in Figure~\ref{fig:01},  has been investigated  for synthesis  and applications, see
for example~\cite{jan_appl_1, jan_appl_2, jan_appl_3, jan_appl_4}.
\begin{figure}[!ht]
 \centerline{\includegraphics[width=0.56\textwidth]{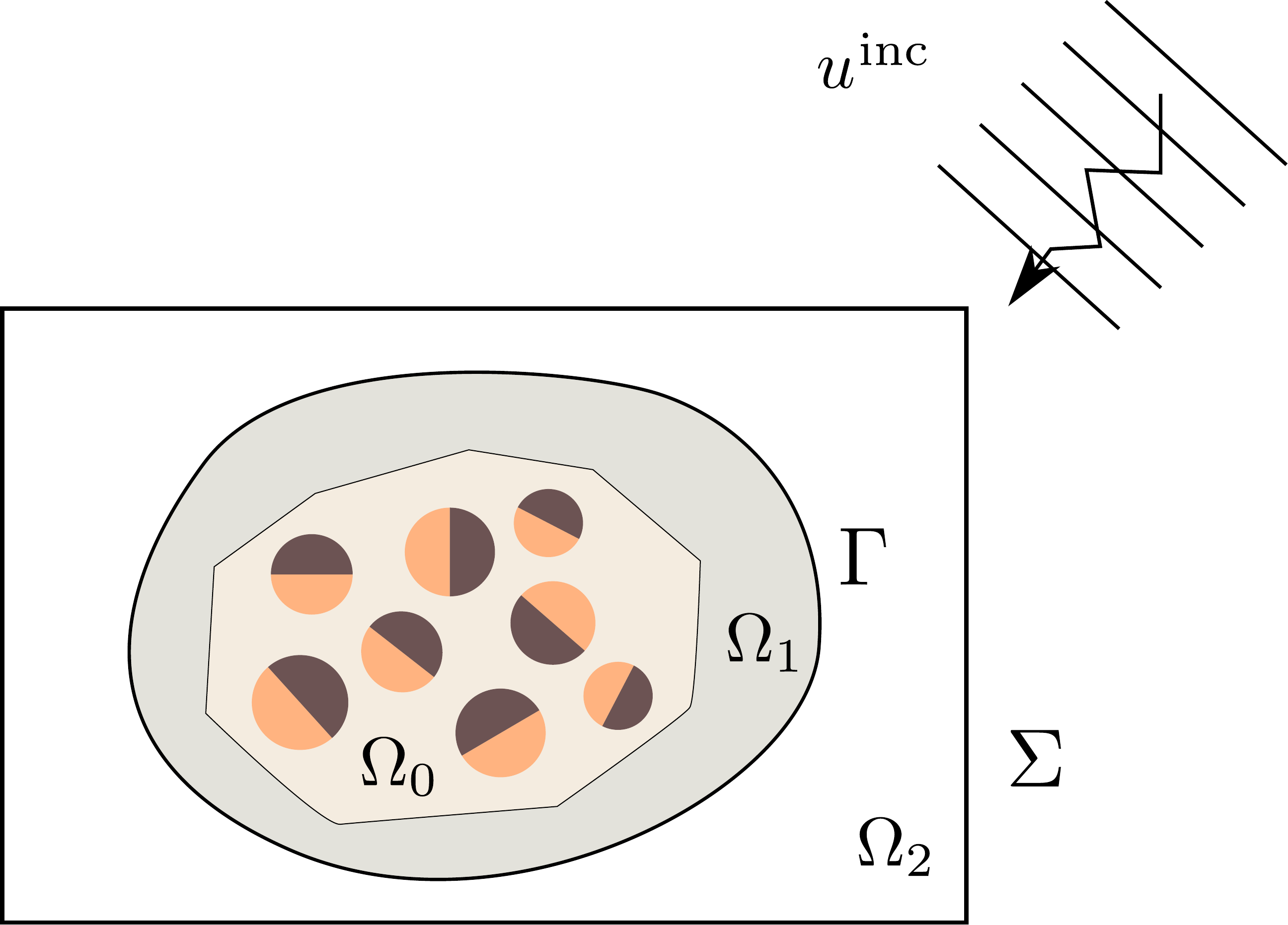}}
 \caption{\label{fig:01} {\em An incident wave $u^{\rm inc}$ impinging on a given Janus-type  heterogeneous multiple-particle configuration ($\Omega_0$)  in $\mathbb{R}^2$.
 Our framework introduces  two artificial boundaries $\Sigma$ and $\Gamma$, and hence two overlapped computational regions: (i) a bounded FEM-domain $\Omega_2$ (with  boundary $\Sigma$)
 and (ii) an unbounded   BEM-region $\mathbb{R}^2 \setminus \overline{\Omega}_1$ (exterior to the smooth  interface $\Gamma$).
The FEM and BEM solutions are constrained  to match in the
overlapped region ${\Omega_{12}:=}(\mathbb{R}^2 \setminus \overline{\Omega}_1)\cap \Omega_2$.}}
\end{figure}

Developing numerical methods to understand scattering effects by Janus configurations (even with simple structures) has been 
an active area of research, see for example the  Janus spherical acoustic 
configuration  effect research~\cite{Jan_sph_20_book, Jan_sph_20_pap} published in 2020 and references therein. 
Motivated by details in the book~\cite{Jan_sph_20_book}, authors of~\cite{Jan_sph_20_pap} numerically investigated scattering effects of homogeneous spheres 
with two  standard (sound-soft and transmission) boundary properties, leading to solving the Helmholtz equation with two distinct boundary conditions and
the Sommerfeld  radiation condition (SRC), for the unknown scattered- and  far-fields.

As described in~\cite{Jan_sph_20_book, Jan_sph_20_pap}, to understand the impact of Janus-type wave configurations,
accurate simulations of certain quantities of interests (QoI) are important. The two key  QoI are 
the  far-field intensity based differential scattering cross sections (DSCS)  and 
also the configuration orientation effect QoI, the orientation-averaged (OA) DSCS. Simulation of the
OA-DSCS is equivalent to simulating far-fields generated by the configuration, such as $\Omega_0$ in Figure~\ref{fig:01}, with the incident wave
$u^{\rm inc}$ impinging $\Omega_0$ from hundreds of  directions surrounding the configuration. 
The key numerical tool in~\cite{Jan_sph_20_pap} is the stable  far-fields based T-matrix  framework that was  analyzed (with {\em a priori bounds})  
and implemented  in~\cite{ghh:2011, tmatrom}. 
For nonlinear inverse problems based applications that use the far-field (with phase or  phase-less DSCS) data, efficient simulation of far-fields (such as at each
Newton-type iterations) are important,   see for example the 2019 book~\cite{colton:inverse} and extensive references therein.

Far-field computations based numerical schemes for such  QoI and tools, in general, do not allow for configurations with  heterogeneous structures and unbounded regions.
This is mainly because numerical partial differential equation (PDE)  algorithms  mainly apply either the ubiquitous finite element method (FEM) or the 
boundary element method (BEM). The FEM requires a bounded domain for finite number of tessellations,  and the BEM requirement of the  fundamental solution  is practical
mainly for a homogeneous medium PDE. The FEM is applied to a variational formulation equation 
of the PDE that involves domain integrals, and the BEM discretizes 
an equivalent boundary integral equation (BIE) for the boundary unknown in a chosen ansatz  for the scattered field,  in conjunction with the fundamental solution~\cite{colton:inverse,Ned:2001}.

 Accordingly for wave propagation models with a bounded heterogeneous medium ($\Omega_0$), to apply the FEM,  an artificial truncation of the unbounded region using, typically,  a polygonal boundary 
 $\Sigma$ (as in Figure~\ref{fig:01})  is introduced and a wave absorbing boundary condition (ABC) on $\Sigma$ is imposed, ignoring the SRC. The standard  variational formulation on the boundary polygonal 
 domain (such as $\Omega_2$ in Figure~\ref{fig:01}) for the heterogeneous Helmholtz PDE with ABC on $\Sigma$ is sign-indefinite~\cite{Ihlenburg:1998, mg2018} and this  non-coercive restriction 
 was removed recently by developing, analyzing, and implementing a new sign-definite  variational formulation~\cite{mg2019}. Another option  to exactly incorporate the SRC is, for example,
 by ignoring the heterogeneity and  consider   instead the homogeneous  model exterior to  the polygonal boundary $\Sigma$ or exterior to a smooth boundary $\Gamma$ 
 (as in Figure~\ref{fig:01}). The artificial smooth boundary choice is especially suitable for developing high-order spectrally accurate approximation of the scattered field 
 in the homogeneous medium using high-order BEMs.

 For Janus-type configurations, it is important to avoid restrictions in either of the above two options by using both the FEM and BEM. 
 This can be achieved   by appropriately  coupling the FEM and BEM  solutions,
 depending on the choices of the artificial boundaries to compute the interior (heterogeneous) and exterior (homogeneous) media solutions.
Such FEM-BEM  coupling mathematical frameworks have been developed and analyzed only by a few authors, but several researchers implemented the
associated  FEM-BEM computational frameworks. 

The widely used FEM-BEM coupling is obtained by choosing one artificial (FEM appropriate) polygonal boundary $\Sigma$,
see for example the review article~\cite{sayas} and references therein. In addition to analysis difficulties~\cite{sayas}, this  approach introduces restricted 
regularity for the solution exterior to $\Sigma$. The high-order regularity based (BEM appropriate) single smooth boundary choice $\Gamma$ was subsequently
developed in~\cite{KiMo:1990,KiMo:1994}. For recent implementations of the single artificial boundary based FEM-BEM framework with high-order accuracy
we refer to~\cite{GaMor:2016, mg2018, Gillman2015}, and associated  analysis issues are highlighted in~\cite[Section 6]{GaMor:2016}. 
Using two artificial interfaces  a mathematical framework  was developed 
over four decades ago in~\cite{overlap-1} and was subsequently used in~\cite{CoyleMonk, overlap-2}. 
A domain-overlapped framework  [illustrated in Figure~\ref{fig:01} with overlapped region $(\mathbb{R}^2 \setminus \overline{\Omega}_1)\cap \Omega_2$] was developed recently by the authors in~\cite{DoGaSay:2019}  
that takes advantage of  a polygonal boundary  ($\Sigma$) required for a wide-class FEM, and a smooth boundary ($\Gamma$) for high-order  spectral BEM. 
Mathematical analysis of the continuous framework in~\cite{DoGaSay:2019} establishes the equivalence  and regularity of the decomposed model.

In this work we present rigorous numerical analysis of the  FEM-BEM adaptive coupling framework introduced in~\cite{DoGaSay:2019} for solving the Helmholtz
acoustic/electromagnetic wave propagation problem on the plane with a bounded heterogeneous region, and demonstrate efficiency of the numerical algorithm  for complex Janus-type heterogeneous configurations with computational experiments. The 
algorithm works on  a suitable partition of the plane defined from  a polygonal domain containing a smooth curve that surround all heterogeneity. 
The overlapped partition is made up of the    bounded polygonal domain containing the heterogeneity,  and the unbounded homogeneous region exterior to the smooth curve.  

On the  bounded domain of the partition we approximate  the solution by a FEM with classical continuous piecewise polynomials on triangular grids,  whereas a 
high-order Nystr\"om BEM is used to compute the scattered wave in the unbounded region. Both solutions are coupled by demanding to coincide in the two artificial boundaries that ensures the FEM and BEM solutions matching in the intersecting common domain of the partition. We prove that the convergence of the scheme in natural norms is of the same order as the best approximations of the projection in the FEM and BEM finite dimensional spaces. In addition, since the Nystr\"om method is super-algebraically convergent, only a relatively few degrees of freedom are required in the BEM solver to keep the error of the same order as the FEM solution, and hence the algorithm facilitates accurate far field, DSCS, and OA-DSCS computations.

The rest of this article is organized as follows: In Section~\ref{sec:model}, we recall the Helmholtz model and an equivalent decomposition decomposition framework, both at continuous level. In Section~\ref{sec:disc-framework}  we setup a discrete counterpart of the decomposition framework and the overlapped FEM-BEM algorithm. Rigorous numerical analysis of the FEM-BEM algorithm establishing optimal order convergence, of the hybridized numerical solutions,  in Section~\ref{sec:num-anal} forms the main theoretical contribution of the article that include references to several results proven in the two Appendix sections of this article. In Section~\ref{sec:num-exp}, we  computationally demonstrate the
described algorithm and theoretically  analysis
using three distinct sets of experiments, in conjunction with the theory and    practical applicability, including implementation of the algorithm for multiple-particle Janus-type configurations with non-smooth solutions, and also for complex structured heterogeneous regions.  

\section{Helmholtz model and decomposition framework} \label{sec:model}
Throughout the article, let $n$ be a piecewise-continuous refractive index  with heterogeneity restricted to a bounded domain $\Omega_0 \subset \mathbb{R}^2$, and in the 
exterior  we take  $n|_{\Omega_0^{\rm c}}\equiv 1$  so that the exterior $\Omega_0^{\rm c}~(:=  \mathbb{R}^2 \setminus \overline{\Omega_0}$) is 
a free-space unbounded medium.

The scattered field $u^s$ is induced by an incident wave $u^{\rm inc}$ (with wavenumber $k >0$) from the exterior region impinging on the heterogeneous medium $\Omega_0$. 
It is convenient to assume that that incident field satisfies the homogeneous  Helmholtz equation $\Delta v+k^2v =0$  in all of the plane  $\mathbb{R}^2$, although it is sufficient to take
it to satisfy  outside of a compact set in  $\mathbb{R}^2$ containing, say, a point-source. Physically appropriate incident waves such as the plane-wave and point-source have these properties. 

We seek the  total wave field  $u~(:= u^s + u^{\rm inc}) \in H^1_{\rm loc}(\mathbb{R}^2)$, representing acoustic or electromagnetic fields, satisfying the 
uniquely solvable problem governed by the variable coefficient  Helmholtz equation and the SRC {(with  $\xh := \x / |\x| \in \sphere$)}:
\begin{equation}
 \label{eq:theproblem}
 \left |
 \begin{array}{rcl}
 \Delta u + k^2 n^2\:u &=&0,\quad \text{in }\mathbb{R}^2, \\
 \partial_{{\xh}} u^s-\mathrm{i}k{u}^s&=&o(|{{\x}|^{-1/2}}),\quad \text{as }| {{\x}} |\to\infty.
 \end{array}
 \right.
\end{equation}
We note that in the heterogeneous medium $\Omega_0$,   $u^s$ is the interior unknown field and is the solution of inhomogeneous Helmholtz equation with inhomogeneous term 
$f = -(\Delta u^{\rm inc} + k^2 n^2\:u^{\rm inc})$.   For the unknown scattered field $u^s$  exterior to $\Omega_0$, in~\eqref{eq:theproblem},   the SRC
\begin{equation} \label{eq:radiation-condition}
\lim_{|\x| \to \infty} |\x|^{1/2} \left( \frac{\partial u^s(\x)}{\partial   {\xh}} - \mathrm{i} k u^s(\x) \right)  = 0.
\end{equation} 
holds uniformly in all directions $\xh = \x / |\x| \in \sphere$. The scattered field $u^s$ 
is a radiating field and, as a consequence of the SRC,   its behavior at 
infinity is captured by the far-field $u^\infty \in L^2(\sphere)$, where
\begin{equation}
\label{eq:far-field}
u^\infty(\xh) = \lim_{|\x| \to \infty} |\x|^{1/2}  e^{-\mathrm{i} k |\x|} u^s(\x). 
\end{equation}
For the total field $u$ induced by a plane wave with incident direction $\ddh \in \sphere$, it is appropriate to denote the associated
far-field as $u^\infty(\xh; \ddh) = u^\infty(\theta; \phi)$,  with $\xh = p(\theta) =
{(\cos\theta, \sin\theta )}
$ and  $\ddh = p(\phi) = {(\cos \phi, \sin\phi)}$%
, for 
some $\theta,\phi \in [0, 2\pi)$. The single-incident plane wave QoI DSCS and multiple-incident plane waves QoI OA-DSCS are then given by~\cite{Jan_sph_20_book, Jan_sph_20_pap} 
\begin{equation} \label{eq:OA-DSCS}
u_{\rm DSCS}(\theta; \phi) = |u^\infty(\theta; \phi)|^2; \qquad u^{\rm OA}_{\rm DSCS}(\theta)   = \frac{1}{2\pi} \int_0^{2\pi }|u^{\infty}(\theta; \phi)|^2\,{\rm d}\phi.
\end{equation}
Clearly, computation of the $u^{\rm OA}_{\rm DSCS}$ with high-accuracy   requires discretization of the above integral with over thousand discrete incident
direction angles $\phi$, leading to  solving the same number of Helmholtz model~\eqref{eq:theproblem}-\eqref{eq:radiation-condition} for {many} distinct inputs. 
This motivates the difficulty of evaluating   $u^{\rm OA}_{\rm DSCS}$ Janus-type configurations with piecewise-continuous refractive indices defined 
on non-trivial geometries.

For the given wave propagation problem \eqref{eq:theproblem},  next we recall an equivalent decomposition framework introduced and analyzed in~\cite{DoGaSay:2019}.
The framework introduce two artificial curves $\Gamma$ and $\Sigma$ with interior $\Omega_1$ and $\Omega_2$, respectively, satisfying $\overline{\Omega}_0\subset \Omega_1\subset \overline{\Omega}_1\subset\Omega_2$, with the assumption that $\Gamma$ is smooth and $\Sigma$ is a polygonal boundary. A sketch of the different domains is displayed in Figure \ref{fig:01}. We denote henceforth  $\Omega_1^{\rm c}:=\mathbb{R}^2\setminus\overline{\Omega}_1$. At continuous level, it is convenient to consider
the decomposition framework using operators defined on classical  Sobolev spaces. To this end, for a general domain $D \in \mathbb{R}^2$ with boundary 
$\partial D$ and for real  $m$,   let  $H^m(D)$ denote the classical Sobolev space. 

We also consider $H^s(\partial D)$ which is well defined for any $s$ if $D$ is smooth. Recall that in this case the trace operator $\gamma_{\partial D}: H^{s+1/2}(D)\to H^s(\partial D)$ is continuous for any $s>0$ as consequence of the Sobolev Trace Theorems, see for example \cite{MR937473} or \cite[Ch 4]{McLean:2000}. 


For Lipschitz domains $\Omega$ with boundary $\Sigma =\partial\Omega$, such as the chosen polygonal domain $\Omega_2$, $H^s(\Sigma)$ is defined only 
for $s\in[-1,1]$ \cite{AdFo:2003,Gr:2011,McLean:2000}. 
We then commit a (slight) abuse of notation and set for $s>1$  
\begin{equation}\label{eq:3.2}
H^s(\Sigma)=\left\{\gamma_\Sigma u\  :\ u\in H^{s+1/2}(D)\right\}.
\end{equation}
\changeSecondRefereeG{In~\eqref{eq:3.2} the open domain $D$ is chosen such that $\Sigma\subset \overline{D}$, and the
space is  endowed with the image norm. It is known that the space is independent of $D$  and, for $s\in(0,1)$, it is a classical Sobolev space.
Thus, {using the definitions,}   the trace operator $\gamma_\Sigma: H^{s+1/2}(D)\to H^{s}(\Sigma)$ is  continuous for  $s\in (0,1)\cup (1,\infty)$.}

The decomposition framework starts with an interior and an exterior Helmholtz problem: 
\begin{itemize}
\item ({\bf{Interior Dirichlet Helmholtz problem in $\Omega_2$ with polygonal boundary $\Sigma$}}): \\ Given 
$f_\Sigma \in H^{1/2}(\Sigma)$, find $\omega_{\rm int} \in H^1(\Omega_2)$ so that
\begin{equation}
 \label{eq:FEM:0}
 \left|
 \begin{array}{rcl}
 \Delta \omega_{\rm int} + k^2 n^2\:\omega_{\rm int} &=&0,\quad \text{in }\Omega_2,\\
 \gamma_\Sigma \omega_{\rm int}  &=&f_\Sigma.
 \end{array}
 \right.
\end{equation}
\item ({\bf{Exterior Dirichlet Helmholtz problem in $\Omega_1^{\rm c}$ with smooth boundary $\Gamma$}}): \\ Given $f_\Gamma\in H^{1/2}(\Gamma)$, 
find $\omega_{\rm ext} \in H^1_{\rm loc}(\Omega_1^{\rm c})$ so that
\begin{equation}
 \label{eq:BEM:0}
 \left|
 \begin{array}{rcl}
 \Delta \omega_{\rm ext} + k^2 \omega_{\rm ext} &=&0,\quad \text{in }\Omega_1^{\rm c},\\
 \gamma_\Gamma \omega_{\rm ext} &=&f_\Gamma,\\
 \partial_{{\xh}}\omega_{\rm ext} - {\rm i}k\omega_{\rm ext} &=& o({|\x|^{-1/2}}).
 \end{array}
 \right.
\end{equation} 
\end{itemize}
The SRC in~\eqref{eq:BEM:0} ensures that the exterior Dirichlet problem is uniquely solvable~\cite{colton:inverse,Ned:2001}
for all wavenumbers $k$. For the interior Dirichlet Helmholtz problem~\eqref{eq:FEM:0} the well-posedness {does not} 
hold for all wavenumbers $k$~\cite{Ihlenburg:1998}, and throughout this article
we assume that the wavenumber $k$ is such that {$\omega_{\rm int}$} is the unique solution of~\eqref{eq:FEM:0}. The well-posedness assumption
for the interior Dirichlet Helmholtz problems in $\Omega_2$ (and in the overlapped region $\Omega_{12} := \Omega_1^c \cap \Omega_2$) 
can be easily avoided, for example, by modifying  the artificial boundaries $\Sigma$ and $\Gamma$~\cite[Section~2.1]{DoGaSay:2019}.

It is convenient to use operators to describe the continuous  decomposition framework. To this end, corresponding to the interior Dirichlet problem~\eqref{eq:FEM:0} we consider two operators  $\KSigmaOmegaTwo, \KSigmaGamma$; and  associated with the exterior problem~\eqref{eq:BEM:0} we consider two operators $\KGammaOmegaOneC, \KGammaSigma$. These pairs of operators are defined, using the unique solution  $\omega_{\rm int}$ of~\eqref{eq:FEM:0} and   $\omega_{\rm ext}$ of~\eqref{eq:BEM:0} and their  traces 
respectively on $\Gamma$ and $\Sigma$, as follows:
\begin{equation}\label{eq:four_ops}
\KSigmaOmegaTwo f_\Sigma = \omega_{\rm int}, \qquad \KSigmaGamma f_\Sigma= \gamma_\Gamma \omega_{\rm int}; \qquad \quad
\KGammaOmegaOneC f_\Gamma  = \omega_{\rm ext}, \quad \KGammaSigma f_\Gamma =  \gamma_\Sigma \omega_{\rm ext}.
\end{equation}
In~\cite{DoGaSay:2019}, the authors proved that the unique solution $u$ of \eqref{eq:theproblem} can be constructed in two steps as follows:
\begin{subequations}
\label{eq:BEMFEM}
\begin{enumerate}
\item  Find   $f_\Sigma:\Sigma\to \mathbb{C}$ and $f_\Gamma:\Gamma\to \mathbb{C}$   so that
\begin{equation}
 \label{eq:BEMFEM:0}
 \left|
 \begin{array}{ccccrcl}
 f_\Sigma&-&\KGammaSigma f_\Gamma&=& \gamma_{\Sigma} u^{\rm inc},\\
 - \KSigmaGamma f_\Sigma&+&f_\Gamma &=& -\gamma_{\Gamma} u^{\rm inc}.
 \end{array}
 \right.
\end{equation} 

\item Construct 
\begin{equation}
 \label{eq:BEMFEM:01}
  u=\begin{cases}
     \KSigmaOmegaTwo f_\Sigma,\quad&\text{in $\Omega_2$},\\
     \KGammaOmegaOneC f_\Gamma+u^{\rm inc},\quad&\text{in $\Omega_1^{\rm c}$}.
    \end{cases}
 \end{equation}
\end{enumerate}
\end{subequations}

The boundary unknowns system~\eqref{eq:BEMFEM:0} can be written in matrix-valued operator form as 
\begin{equation}
 \label{eq:BEMFEM:0b}
 \left({\cal I}-{\cal K}\right)\begin{bmatrix}
                          f_\Sigma\\
                          f_\Gamma
                          \end{bmatrix}=\begin{bmatrix*}[r]
                          \gamma_{\Sigma} u^{\rm inc}\\
                          -\gamma_{\Gamma} u^{\rm inc}
                          \end{bmatrix*},\qquad {\cal K}:=\begin{bmatrix}
&\KGammaSigma\\
\KSigmaGamma &
\end{bmatrix}.
\end{equation}
(${\cal I}$ is obviously the identity matrix operator.)
It is not difficult to see that the off-diagonal block ${\cal K}: H^{0}(\Sigma)\times H^{0}(\Gamma) \to  H^{s}(\Sigma)\times H^{s}(\Gamma) $ is continuous for any 
$s \ge 0$\cite{DoGaSay:2019}. 

The following result summarizes the equivalence of the decomposition framework to the original problem~\eqref{eq:theproblem} at the continuous level.
\begin{theorem}\label{theo:main:01}
 Assume that the only solution to problem
\begin{equation}
 \label{eq:BEMFEM:02}
 \left| 
 \begin{array}{rcl}
 \Delta v + k^2 v &=&0,\quad \text{in }\Omega_{12} =\Omega_1\cap\Omega_2^c,\\
 \gamma_\Gamma v &=&0,\quad \gamma_\Sigma v\,=\,0
 \end{array}
 \right.
\end{equation} 
is the trivial one. Then ${\cal I}-{\cal K}: H^{s}(\Sigma)\times H^{s}(\Gamma)\to H^{s}(\Sigma)\times H^{s}(\Gamma)$ is invertible for any $s\ge 0$. Therefore \eqref{eq:BEMFEM:0} is uniquely solvable. Furthermore, $u$ defined by~\eqref{eq:BEMFEM:01} is the unique solution of the full Helmholtz problem \eqref{eq:theproblem}. 
\end{theorem}

\begin{proof} This result is proven in \cite{DoGaSay:2019}. We will give here a sketch of the proof for the sake of completeness.  By Fredholm alternative it suffices to show that ${\cal I}-{\cal K}$ is one-to-one. 
 We note that any if $(f_\Sigma^*, f_\Gamma^*)\in N({\cal I}-{\cal K})$ then 
  $\omega_{12}:=(\KSigmaOmegaTwo f_\Sigma^*- \KGammaOmegaOneC f_\Gamma^*)|_{\Omega_{12}}$ is a solution for \eqref{eq:BEMFEM:02}. By hypothesis,  
 $\omega_{12}=0$. Hence, $u$ in \eqref{eq:BEMFEM:01} is well defined for $u^{\rm inc}=0$. The  well-posedness  of the scattering problem \eqref{eq:theproblem}, implies that $u=0$. In particular, 
 $\KGammaOmegaOneC  f_\Gamma^*{|_{\Omega_2^c}}=0$, and from the  analytic continuation principle $\KGammaOmegaOneC  {f^\star_\Gamma}=0$.   Therefore, also ${\KSigmaOmegaTwo f_\Sigma^*|_{\Omega_{12}}}=0$  and hence $(f_\Sigma^*, f_\Gamma^*) ={\bf 0}$. 
\end{proof}

\begin{remark}\label{rem:multi} For the OA-DSCS calculations with a large number of incident plane waves $u^{\rm inc}$, a marked advantage of our equivalent
formulation~\eqref{eq:BEMFEM:0}-\eqref{eq:BEMFEM:01} compared to the original model problem~\eqref{eq:theproblem} [with a large number of inhomogeneous 
terms $f = -(\Delta u^{\rm inc} + k^2 n^2\:u^{\rm inc})$ in $\Omega_0$] is that the interior and exterior homogeneous problems 
$\KSigmaOmegaTwo f_\Sigma, \KGammaOmegaOneC f_\Gamma$ (and hence their traces 
$\KSigmaGamma f_\Sigma, \KGammaSigma f_\Gamma$) can be
setup \changeSecondRefereeG{ {\em independently} }of the incident waves. Since the unknowns $f_\Sigma$ and $f_\Gamma$ are defined, respectively, only on boundary curves $\Sigma$ and $\Gamma$, these unknowns can represented by a few boundary unknowns 
(or boundary basis   functions) and hence the   associated Helmholtz 
problems setup  $\KSigmaOmegaTwo f_\Sigma, \KGammaOmegaOneC f_\Gamma$ is a   {\em naturally parallel} process for computational purposes. 
\end{remark}

\section{Discrete decomposition FEM-BEM  framework}~\label{sec:disc-framework}

The numerical discretization of the continuous decomposition framework~\eqref{eq:BEMFEM} is essentially obtained by replacing the four continuous operators 
in~\eqref{eq:four_ops} using appropriate FEM and BEM based discrete counterparts.  For   $\KSigmaOmegaTwo$ we will use the standard FEM with continuous piecewise polynomial elements on a triangular conformal mesh of $\Omega_2$~\cite{Ihlenburg:1998}.

For the BEM certainly an extensive range of methods is at our disposal in the literature~\cite{colton:inverse,Ned:2001}. We will restrict ourselves to  the spectrally 
accurate Nystr\"om method~\cite{Kress:2014, colton:inverse} .This scheme provides a discretization of the key boundary integral operators of the associated Calderon calculus;  in this article we will use only make use of the discrete Single- and Double-Layer operators that  converge super-algebraically, and it is not difficult to implement. 
A disadvantage of the  Nystr\"om method is that it requires an accurate differentiable  parameterization of the boundary. This is because the method is based on splitting  the integral operators into regular and singular parts for which   appropriate decompositions and factorizations of the kernels of the operators are needed. This is not a severe restriction in our case since $\Gamma$ is an auxiliary user-chosen artificial curve and therefore can be taken to be simple and smooth. 
 
In the next two subsections we recall the standard FEM and Nystr\"om procedure and conclude this section with the discrete FEM-BEM decomposition framework required
for the main focus of this work on the numerical analysis of the FEM-BEM algorithm~\cite{DoGaSay:2019} .

\subsubsection*{The FEM procedure }
Let $\{\mathcal{T}_h\}_h$ be a sequence of regular triangular meshes with $h$ denoting the discrete parameter,  the diameter of the largest element of the grid. We then write  $h\to 0$  to mean that the maximum of the diameters of the elements tends to 0. A technical mesh assumption, not restrictive in practice,   described  in detail in  Assumption 1 below, will be
used in our proofs to ensure (i) a faster convergence of the FEM in stronger norms around $\Gamma$ (see Theorem \ref{theo:A02} in Appendix); and (ii) the stability of the full method. 

On ${\cal T}_h$, we construct the finite dimensional FEM space
\[
  \mathbb{P}_{h,d}:=\{{v}_h\in{\cal C}^0(\Omega_2) :\ u_h|_{T_h}\in\mathbb{P}_d,\ \forall T_h \in {\cal T}_h\},
 \]
 where $\mathbb{P}_d$ is the space of bivariate polynomial of degree $d$. Then given $f_{\Sigma}^h\in  \gamma_\Sigma \mathbb{P}_{h,d}$, we define 
 \[
\KSigmaOmegaTwo^h:\gamma_\Sigma \mathbb{P}_{h,d}\to  \mathbb{P}_{h,d}
 \]  
 as $\KSigmaOmegaTwo^hf_{\Sigma}^h := u_h$ that is the unique  of the standard FEM discrete system:
 \begin{equation}\label{eq:4.1}
  \left|
  \begin{array}{l}
  u_h\in  \mathbb{P}_{h,d}\\
b_{k,n}(u_h,v_h)=0,\quad \forall v_h\in \mathbb{P}_{h,d}\cap H_0^1(\Omega_2)\\
\gamma_\Sigma u_h = f_{\Sigma}^h,
  \end{array}
  \right. \qquad b_{k,n}(u,v):=\int_{\Omega_2}\nabla u \cdot\nabla v - k^2 \int_{\Omega_2} n^2\, u  v. 
 \end{equation}
It is  well known that $\KSigmaOmegaTwo^h$~\cite{Ihlenburg:1998} is well defined for any  sufficiently fine mesh. Notice also that $\KSigmaOmegaTwo^h$ is defined on the discrete space $\gamma_\Sigma  \mathbb{P}_{h,d} $, the   trace {finite element} space on the boundary $\Sigma$.  Hence, with the help of 
\begin{equation}\label{eq:def:QhSigma}
 {\rm Q}^h_{\Sigma}:{\cal C}^0(\Sigma)\to \gamma_\Sigma \mathbb{P}_{h,d}, 
\end{equation}
the  nodal interpolation operator on $\gamma_\Sigma  \mathbb{P}_{h,d}$ defined by the  finite element space, we have 
\[
 \KSigmaOmegaTwo^h  {\rm Q}^h_{\Sigma} \approx \KSigmaOmegaTwo,
\]
providing an optimal approximation in $\mathbb{P}_{h,d}$ for sufficiently smooth Dirichlet data $f_\Sigma^h$ (see Theorem \ref{theo:5.2} in the next section).


\subsubsection*{The BEM procedure }
Let 
\begin{equation}\label{eq:parameterization}
{\bf x}=(x_1(t),x_2(t)):\mathbb{R}\to\Gamma
\end{equation}
be a smooth $2\pi-$periodic parameterization of $\Gamma$. 
We denote by ${\rm SL}_k$, ${\rm DL}_k$, the (parameterized) layer potentials,  defined for ${\bm z}\in\mathbb{R}^2\setminus\Gamma$ as,
\begin{eqnarray*}
 ({\rm SL}_k\varphi_{{\rm per}})({\bm z})&:=&\int_0^{2\pi} \Phi_k({\bm z}-{\bf x}(t))\varphi_{{\rm per}}(t)\,{\rm d}t,\,\quad\\
 ({\rm DL}_kg_{{\rm per}})({\bm z})&:=&\int_0^{2\pi} \big(\nabla_{\bm y} \Phi_k({\bm z}-{\bm y})\big)\Big|_{{\bm y}={\bf x}(t)}\cdot \bm{\nu}(t)\,g_{{\rm per}}(t)\,{\rm d}t,
 \end{eqnarray*}
where $\Phi_k =\frac{\rm i}4 H_0^{(1)}(k|\cdot|)$ is the fundamental solution for the constant  coefficient   Helmholtz operator in $\mathbb{R}^2$ with wavenumber $k$; 
$\bm{\nu}(t):=(x_2'(t),-x_1'(t))$ is a non-normalized normal vector. 
Observe that $|{\bf x}'(t)|>0$ is then incorporated to the density in ${\rm SL}_k$ and to the kernel in ${\rm DL}_k$. We follow the same convention for the single and double boundary operator defined {for $2\pi$-periodic densities $\varphi_{{\rm per}}$ and $ g_{{\rm per}}$} as  
 \begin{eqnarray}\label{eq:Vk}
 \hspace{-0.3in} (\mathrm{V}_{k}\varphi_{{\rm per}})(s) \hspace{-0.1in} &:=&\left(\gamma_\Gamma  {\rm SL}_k\varphi_{{\rm per}} \right) ({\bf x}(s)) = \int_0^{2\pi} \Phi_k({\bf x}(s)-{\bf x}(t))\varphi_{{\rm per}}(t)\,{\rm d}t\,\quad\\  
\hspace{-0.3in}  (\mathrm{K}_{k}g_{{\rm per}})(s) \hspace{-0.1in}&:=&\pm\tfrac12 g_{{\rm per}}(s) + (\gamma^{\mp}_{\Gamma}{\rm DL}_k g_{{\rm per}})(s)
   =\int_0^{2\pi} \big(\nabla_{{\bm y}} \Phi_k({\bf x}(s)-{\bm y})\big)\Big|_{{\bm y}={\bf x}(t)} \hspace{-0.15in} \cdot \bm{\nu}(t)\,g_{{\rm per}}(t)\,{\rm d}t. 
   \label{eq:Dk}
 \end{eqnarray}
The Brakhage-Werner formulation, first introduced in \cite{MR190518} {(see also~\cite{colton:inverse, Ned:2001})} provides a robust representation for the solution of the exterior Dirichlet solution for the Helmholtz problem \eqref{eq:BEM:0}:
\begin {equation} \label{eq:potential}
\KGammaOmegaOneC f_\Gamma= ({\rm DL}_k - {\rm i}k{\rm SL}_k )( \tfrac12\mathrm{I}+\mathrm{K}_k-\mathrm{i}k\mathrm{V}_k)^{-1} (f_\Gamma \circ{\bf x}),
\end{equation}
with $\mathrm{I}$  obviously being the identity operator. The above representation of the exterior scattered field, satisfying the SRC, provides an exact ansatz for the associated 
far-field  as $({\cal {F\varphi}}_{\rm per})({\zh}), ~\zh \in \sphere$
defined, using the boundary density 
$\varphi_{\rm per} = ( \tfrac12\mathrm{I}+\mathrm{K}_k-\mathrm{i}k\mathrm{V}_k)^{-1} (f_\Gamma \circ{\bf x})$~\cite{colton:inverse}: 
	\begin{equation}\label{eq:far}
	\left({\cal F}\varphi_{\rm per}\right)(\zh):=
	\sqrt{\frac{k}{8\pi } } \exp\big(-\tfrac14\pi{\rm i}  \big)\int_{0}^{2\pi} \exp(
	-\mathrm{i} k (\widehat{\bm z}\cdot{\bf x}(t))) \big[\widehat{\bm z}\cdot(x_2'(t),-x_1'(t))+1\big]\varphi_{\rm per} (t)\,{\rm d}t. 
	\end{equation}
Thus  accurate computational approximations of $\varphi_{\rm per}$ provide spectrally
accurate approximations for both the scattered and the far-field. For 
computing
$\varphi_{\rm per}$,
the Nystr\"om BEM solver makes use of a decomposition of  the single- and double- layer operator into {\em logarithmic} and regular parts: 
\begin{eqnarray*}
 ({\rm V}_k\varphi_{{\rm per}})(s)&=&\int_{0}^{2\pi}A(s,t)\log\sin^2\tfrac{s-t}2\:\varphi_{{\rm per}}(t)\:{\rm d}t+\int_{0}^{2\pi}B(s,t)\varphi_{{\rm per}}(t)\:{\rm d}t,\\ 
 ({\rm K}_kg_{{\rm per}})(s)&=&\int_{0}^{2\pi}C(s,t)\log\sin^2\tfrac{s-t}2\:g_{{\rm per}}(t)\:{\rm d}t+\int_{0}^{2\pi}D(s,t)g_{{\rm per}}(t)\:{\rm d}t.
 \end{eqnarray*}  
Functions $A,\ B,\ C,\ D$ are smooth  and $2\pi-$biperiodic. 

Next with $N$ being a positive integer BEM discretization parameter, we consider the grid 
\[
\{t_j\}_{j\in\mathbb{Z}}\subset\mathbb{R}, \quad   t_j:= \tfrac{j \pi}{N}, 
 \] 
 and the $2N$-dimensional space of trigonometric polynomials defined by 
   \begin{equation}\label{eq:def:Tn}
  \mathbb{T}_N:= \spann\langle e_\ell\ :\ \ell\in\mathbb{Z}_N\rangle,\quad  e_{\ell}(t): =\exp({\rm i}\ell t) 
 \end{equation}
 where $\mathbb{Z}_N= \{-N+1,-N+2,\ldots, N\}$. The interpolation operator for $2\pi-$periodic functions $\varphi_{{\rm per}}$ 
 \[
\mathbb{T}_N\ni  \mathrm{Q}_N\varphi_{{\rm per}} \quad \text{s.t.}\quad (\mathrm{Q}_N\varphi_{{\rm per}})(t_j)=\varphi_{{\rm per}}(t_j),
 \]
 is  known to be well defined. 
 
 The Nystr\"om method is based on the following approximations for the single- and double-layer operators: 
\begin{eqnarray*}
 ({\rm V}_{k}^{N}\varphi_{{\rm per}} )(s)&:=& \int_{0}^{2\pi}{\rm Q}_N(A(s,\cdot)\varphi_{{\rm per}} \big)(t)\log\sin^2\tfrac{s-t}2\:{\rm d}t+\int_{0}^{2\pi}{\rm Q}_N(B(s,\cdot)\varphi_{{\rm per}} \big)(t)\:{\rm d}t\\ 
 ({\rm K}_{k}^{N}g_{\rm per} )(s)      &:=& \int_{0}^{2\pi} {\rm Q}_N(C(s,\cdot)g_{\rm per} \big)(t)\log\sin^2\tfrac{s-t}2\:{\rm d}t+\int_{0}^{2\pi}{\rm Q}_N(D(s,\cdot)g_{\rm per} \big)(t)\:{\rm d}t.
 \end{eqnarray*}
 We stress that the integrals above can be computed exactly. Indeed,
 \[
-\frac{1}{2\pi}\int_{0}^{2\pi} \log\sin^2\tfrac{t}2\:e_n(t)\:{\rm d}t = 
-\frac{1}{2\pi} \int_{0}^{2\pi} \log\sin^2\tfrac{t}2 \cos(n t)\:{\rm d}t=\begin{cases}
                                                          \log 4,      &n=0\\
                                                          \frac{1}{|n|}, &n\ne 0,
                                                         \end{cases}
 \]
 and
 \[
   \int_{0}^{2\pi} (\mathrm{Q}_N g_{\rm per}) (t)\:{\rm d}t =
 \frac{\pi}{N} \sum_{j=0}^{2N-1} (\mathrm{Q}_N g_{\rm per}) (t_j) =
 \frac{\pi}{N} \sum_{j=0}^{2N-1}   g_{\rm per}  (t_j)  ,\qquad 
 %
 \]
 i.e., for the regular part the suggested approximation  yields the trapezoidal/rectangular rule for $2\pi-$periodic functions.   
 
 The evaluation of the potentials, as integral operators with smooth kernel, is carried out in a similar way:
 \begin{equation}\label{eq:SLNDLN}
 \begin{aligned} 
 \big({\rm SL}_{k}^{N}\varphi_{{\rm per}} \big)({\bm z})\ &:=\ \int_{0}^{2\pi}
 {\rm Q}_N(\Phi_k({\bm z}-{\bf x}(\cdot))\varphi_{{\rm per}})(t)\,{\rm d}t\\ 
 \big({\rm DL}_{k}^{N}g_{\rm per}\big)({\bm z})\ &:=\ \int_{0}^{2\pi}
 {\rm Q}_N\big(\big(\nabla_{\bm y}\Phi_k({\bm z}-{\bm y})\big) \big|_{{\bm y}={\bf x}(\cdot)}\cdot \bm{\nu}(\cdot)\,g_{\rm per})
  (t)\,{\rm d}t.
  \end{aligned} 
 \end{equation}

 We are ready to present the discrete version for $\KGammaOmegaOneC^h$. Hence, in view of
\begin{equation}\label{eq:4.2}
\KGammaOmegaOneC f_\Gamma = ({\rm DL}_k - {\rm i}k{\rm SL}_k ){\cal L}_k(f_\Gamma \circ{\bf x}),\quad \text{with }
{\cal L}_k  :=( \tfrac12\mathrm{I}+\mathrm{K}_k-\mathrm{i}k\mathrm{V}_k)^{-1} , 
\end{equation} 
we define for $f_{\rm per} :=   f_\Gamma \circ{\bf x}$
 \begin{equation}\label{eq:4.3}  
 \KGammaOmegaOneC^N  {f_{\rm per}}:=(\mathrm{DL}_{k}^{N}-{\rm i}k\mathrm{SL}_{k}^{N})  {\cal L}_{k}^{N}  {f_{\rm per}},\quad \text{with }
 {\cal L}_{k}^{N}:= (\tfrac12\mathrm{I}+\mathrm{K}_{k}^{N}-\mathrm{i}k\mathrm{V}_{k}^{N})^{-1}.
 \end{equation} 
 Hence  $\varphi_{\rm per}^N :=   {\cal L}_{k}^{N} f_{\rm per} $ is the spectrally accurate approximation to the density
 $ \varphi ={\cal L}_{k} f_\Gamma $.  In addition, we can also introduce the associated far-field approximation, see \eqref{eq:far}, 	as  
	\begin{equation}\label{eq:far-z-N}
	\left({\mathcal{F}}_N{\varphi}_{{\rm per}}^N\right)(\widehat{\bm z}):=  \sqrt{\frac{k}{8\pi} }\exp\big(-\tfrac14\pi{\rm i}  \big)\frac{\pi}{N}  \sum_{j=-N+1}^{N} \exp(
	-\mathrm{i} k(\widehat{\bm z}\cdot{\bf x}(t_j))) \big[\widehat{\bm z}\cdot(x_2'(t_j),-x_1'(t_j))+1\big] {\varphi}_{{\rm per}}^N(t_j). 
	\end{equation}

\begin{remark}\label{remark:3.1}It is not difficult to show that
 \[
\varphi^N_{\rm per} = {\cal L}_{k}^{N}f_{\rm per} 
 \quad\Rightarrow\quad
 {\rm Q}_N
\varphi^N_{\rm per}  ={\rm Q}_N{\cal L}_{k}^{N}{\rm Q}_N f_{\rm per} . 
 \]
 \changeSecondRefereeG{We note that only pointwise values of the density $\varphi_{\rm per}$ are used for computation of
 the BEM-based potentials and QoI. Accordingly, we can  conclude that} the true unknowns of the Nystr\"om method are the values of the density at the grid points $\{\varphi_{\rm per}(t_j)\}_{j\in\mathbb{Z}}$ and that such values are computed by the algorithm using only the values of the right-hand-side $f_{\rm per}$ at the same grid.
 \end{remark}

\subsubsection*{The BEM-FEM  coupling method}

Let
\[
\KGammaSigma^N := \gamma_\Sigma \KGammaOmegaOneC^N,\quad
\KSigmaGamma^h := (\gamma_\Gamma \KSigmaOmegaTwo^h)\circ{\bf x},
\]
be the discrete counterparts of $\KGammaSigma$ and $\KSigmaGamma$. That is, the first operator computes the BEM solution and evaluate on the outer polygonal domain $\Sigma$, whereas {the second one} solves the \changeSecondRefereeG{FEM problem and evaluates} it on the interior and smooth $\Gamma$. 

Then our coupled FEM-BEM algorithm is:
\begin{subequations}
\begin{itemize}
\item  Solve \label{eq:NhBEMFEM}
\begin{equation}
\begin{aligned}
 \label{eq:NhBEMFEM:0} 
  (f^h_{\Sigma}, f^N_{\rm per})\in & \gamma_\Sigma \mathbb{P}_{h,d}\times \mathbb{T}_N\quad \text{s.t.}\\
  &
 \left({\cal I}-\begin{bmatrix}
                                       &{\rm Q}^h_{\Sigma} \KGammaSigma^N\\
                                       {\rm Q}_N\KSigmaGamma^h
                                      \end{bmatrix}
\right)\begin{bmatrix}
                          f^h_{\Sigma}\\
                          f^N_{\rm per}
                          \end{bmatrix}=\begin{bmatrix*}[r]
                          {\rm Q}^h_{\Sigma}\gamma_{\Sigma} u^{\rm inc}\\
                          - {\rm Q}_N \gamma_{\Gamma}( u^{\rm inc}\circ{\bf x} )
                          \end{bmatrix*}.
                          \end{aligned}
\end{equation} 

\item Construct 
\begin{equation}
 \label{eq:NhBEMFEM:01}
  w_h=\KSigmaOmegaTwo^h f_\Sigma^h,\quad 
  \omega_N=\KGammaOmegaOneC^Nf^N_{\rm per},\quad u_{h,N}=\begin{cases}
     w_h,\quad&\text{in $\Omega_2$},\\
     \omega_N+u^{\rm inc},\quad&\text{in $\Omega_1^{\rm c}$}.
    \end{cases}
 \end{equation}
\end{itemize}
\end{subequations}

\begin{remark} In view of Remark \ref{remark:3.1}, 
for implementation of the above algorithm, the pointwise values of the numerical solution $(f^h_{\Sigma}, f^N_{\rm per})$ at the  boundary nodes (of the FEM and BEM grids on $\Sigma$ and $\Gamma$) are the  true unknowns. Hence ~\eqref{eq:NhBEMFEM:0} leads to a relatively small algebraic system.

We also note  that, in view of the first-stage coupled and constrained solutions in~\eqref{eq:NhBEMFEM:0},  the overlapped algorithm is designed in such a way that  
for sufficiently fine grids, the final-stage FEM and BEM parts of the algorithm  in \eqref{eq:NhBEMFEM:01}
 lead to two numerically coinciding solutions in $\Omega_{12}=\Omega_1^{\rm c}\cap \Omega_2$. (We demonstrate this property using numerical experiments.)
\end{remark}


%
%
%

\section{Numerical analysis of overlapped FEM-BEM algorithm}~\label{sec:num-anal} 
We analyze the above FEM-BEM scheme through rigorous derivation of the  stability and convergence estimates in appropriate Sobolev norms. For the polygonal region we keep using the standard space $H^s(\Omega_1)$ and $H^s(\Sigma)$ (with the convention \eqref{eq:3.2} for $s>1$). For the smooth boundary $\Gamma$,  since we switch to the parameterized spaced (via ${\bf x}$), we will work with $2\pi-$periodic Sobolev spaces. \changeSecondRefereeG{To this end, we define 
\begin{equation}\label{eq:per_spa}
H^s_{{\rm per}}:=\{\varphi_{{\rm per}}\in  {\cal D}'(\mathbb{R})\ :\  \varphi_{{\rm per}} = \varphi_{{\rm per}}(\,\cdot\,+2\pi),\   \|\varphi_{{\rm per}}\|_{H^s_{{\rm per}}}<\infty\}.
\end{equation}}
\changeV{In~\eqref{eq:per_spa}, the distribution elements in the space are continuous linear functionals on ${\cal D}(\mathbb{R})$ (where
${\cal D}(\mathbb{R})$ is the space of the smooth compactly supported functions in $\mathbb{R}$ endowed with its natural topology), 
and the space is equipped with the norm}
\begin{equation}\label{eq:4.02}
\|\varphi_{{\rm per}}\|_{H^s_{{\rm per}}}^2=|\varphi_{{\rm per}}(0)|^2+\sum_{n\ne 0}|n|^{2s} |\widehat{\varphi}_{{\rm per}}(n)|^2,
\end{equation}
where $\widehat{\varphi}_{{\rm per}}(n)$ is the $n$th Fourier coefficient:
\[
\widehat{\varphi}_{{\rm per}}(n):=
\int_0^{2\pi} {\varphi}_{{\rm per}}(t)e_{-n}(t)\,{\rm d}t. 
\]
It is a well established result that $H^s(\Gamma)$ and $H^s_{\rm per}$ are isomorphic via the composition with ${\bf x}$  (see for example \cite[Th. 8.13]{Kress:2014}).

\subsection{Interpolation and some projection operators on the discrete spaces }

 For the implementation as well as for the theoretical analysis, {we need  some projections onto the discrete spaces defined on} the boundaries $\Sigma$ and $\Gamma$. In the first case we have already introduced   the Lagrange interpolation operator in $\mathbb{P}_{h,d}$  (see \eqref{eq:def:QhSigma}). For such operators it can be shown, as consequence of well-known results, that
\begin{subequations}\label{eq:4.5}
\begin{equation}\label{eq:4.5a}
 \| {\rm Q}^h_{\Sigma} f_\Sigma-f_\Sigma\|_{H^s(\Sigma)}\le C h_\Sigma^{t-s}\|f_\Sigma\|_{H^t(\Sigma)}
\end{equation} 
where $C$ depends only on $\Sigma$,  $s\in [0,1]$ and $s\le  t<d+1$  and $t>1/2$. 
Here and in what follows  $h_\Sigma$ is the maximum of the diameters of   elements in $\Sigma$ induced by  the mesh ${\cal T}_h$. The convergence of order $h_\Sigma^{d+1-s}$ can be attained if we assume an extra regularity for $f_\Sigma$: for any $s\in[0,1]$ and  $t>d+1$, there exists $C$ so that
\begin{equation}\label{eq:4.5b}
 \| {\rm Q}^h_{\Sigma} f_\Sigma-f_\Sigma\|_{H^s(\Sigma)}\le C h_\Sigma^{d+1-s}\|f_\Sigma\|_{H^t(\Sigma)}.
\end{equation} 
\end{subequations} 
We refer the reader to Appendix A for  {proofs} of such results. In this section we also {introduce}
  a more flexible projection on $\gamma_\Sigma \mathbb{P}_h^d$, which will be required for analysis purposes.  
Roughly speaking  this is a consequence of carrying out the analysis in $H^1(\Omega)$ and $H^{1/2}(\Sigma)$ norms which, although are the natural ones for the FEM method, contains discontinuous functions and for which the action of ${\rm Q}^h_{\Sigma}$ cannot be therefore considered.  Hence, we claim that there exists  ${\rm P}_\Sigma^h:H^{1/2}(\Sigma)\to  \gamma_\Sigma \mathbb{P}_{h,d}$  a  projection on $\gamma_\Sigma \mathbb{P}_{h,d}$ satisfying 
\begin{subequations}\label{eq:5.2}
\begin{eqnarray}
\|{\rm P}_\Sigma^hf_\Sigma -  f_\Sigma\|_{H^s(\Sigma)}&\le& C  h_\Sigma^{t-s}\|f_\Sigma\|_{H^{t}(\Sigma)},\quad  0\le s< 1,\quad s\le t< d+1,\quad  t\ge 1/2.
\\
\|{\rm P}_\Sigma^hf_\Sigma -  f_\Sigma\|_{H^s(\Sigma)}&\le& C  h_\Sigma^{d+1-s}\|f_\Sigma\|_{H^{t}(\Sigma)},\quad  0\le s< 1,  \quad t>d+1.
\end{eqnarray}
\end{subequations}   
Proofs of the  estimates \eqref{eq:4.5} and \eqref{eq:5.2}  are given \changeV{in Proposition \ref{prop:A1} and \ref{prop:A2} respectively} in Appendix A. 


The convergence of the trigonometric interpolation operator in the Sobolev frame is a well known result, see for example, \cite[Th. 8.3.1]{SaVa:2002}:
 \begin{equation}\label{eq:5.7}
 \|\mathrm{Q}_N\varphi_{{\rm per}}-\varphi_{{\rm per}}\|_{H_{\rm per}^s}\le C N^{s-t}\|\varphi_{{\rm per}}\|_{H_{\rm per}^t},
\end{equation}
for any $t\ge s\ge 0$ and $t>1/2$, where $C$ depends only on $s,t$. For similar reasons, and  again only for our analysis, we  use {the $L^2$-projection}{, which turns out to be the $H^s$-orthogonal projection for any $s$,}  on $\mathbb{T}_N$: 
\[
\mathbb{T}_N \ni \mathrm{P}_N\varphi_{{\rm per}},\quad\text{s.t.}\quad \widehat{ {\rm P}_N\varphi_{{\rm per}}}(n)
=\widehat{\varphi_{{\rm per}}}(n),\quad n\in \mathbb{Z}_N.  
\]  
It is straightforward to show that for any $t\ge s$, 
\begin{equation}\label{eq:5.6}
 \|\mathrm{P}_N\varphi_{{\rm per}}-\varphi_{{\rm per}}\|_{H_{\rm per}^s}\le N^{s-t}\|\varphi_{{\rm per}}\|_{H_{\rm per}^t}.
\end{equation}

With the help of these two projections we define 
\begin{equation}\label{eq:5.9}
 {\cal K}_{h,N}:=\begin{bmatrix}
                             & {\rm Q}_\Sigma^h \KGammaSigma^N {\rm P}_{N}\\
            {\rm Q}_{N}\KSigmaGamma^h {\rm P}_\Sigma^h&
          \end{bmatrix}
\end{equation}
so that the \eqref{eq:NhBEMFEM:0} can be set up as
\[
 ({\cal I}-{\cal K}_{h,N})\begin{bmatrix}
                          f_{\Sigma}\\
                          f_{{\rm per}}
                          \end{bmatrix}=\begin{bmatrix}
                          g_{\Sigma}\\
                          g_{{\rm per}}
                          \end{bmatrix} 
\]
for appropriate $(g_{\Sigma},g_{{\rm per}})^\top \in H^{1/2}(\Sigma)\times H_{{\rm per}}^{1/2}$. Observe that in the case that the right-hand-side  $(g_{\Sigma}, g_{{\rm per}})$ belongs to the discrete space $\gamma_\Sigma \mathbb{P}_{h,d}\times \mathbb{T}_N$ (as in \eqref{eq:NhBEMFEM:0}), so is the solution $(f_{\Sigma},f_{{\rm per}})$. In this case, ${\rm P}_\Sigma^h$ and ${\rm P}_{N}$ in \eqref{eq:5.9} are already acting on elements on the discrete space and can safely removed. Thus the role of these \changeV{projections is to facilitate} the analysis by setting up the equation in the continuous framework.

\subsection{Convergence for the FEM scheme}

We recall some classical convergence results as well as superconvergence phenomenon for the FEM solution in the following theorem.

 \begin{theorem} \label{theo:5.2} For fixed $g_\Sigma\in \ H^{1/2}(\Sigma)$, $g_\Sigma^h\in \gamma_\Sigma\mathbb{P}_{h,d}$, let
 \[
  u := \KSigmaOmegaTwo g_\Sigma,\quad 
  u_h: = \KSigmaOmegaTwo^h  g_\Sigma^h,
 \]
be  the solution of the interior Helmholtz problem~\eqref{eq:FEM:0} and the approximation given by the FEM \eqref{eq:4.1}. 
Then there exists $C>0$ independent of $g_\Sigma$, $g_\Sigma^h$ and $h$ so that  
\begin{equation}\label{eq:FEMconvergence}
 \|u-u_h\|_{H^1(\changeFirstReferee{\Omega_2})}\le C\Big[ \inf_{v_h\in \mathbb{P}_{h,d}}\|u-v_h\|_{H^1(\changeFirstRefereeG{\Omega_2})}+
 \|g_\Sigma- g_\Sigma^h\|_{H^{1/2}( \Sigma )}
 \Big].
\end{equation}
 
 Furthermore, let  $D'$,  $D$ be domains  with $D'\subset \overline{D}'
\subset D\subset  \overline{D}\subset \Omega_2\setminus\overline{\Omega_0}$ and $\{{\cal T}_h\}_h$ be a sequence of regular grids with $h\to 0$ which are \changeSecondRefereeG{quasi-uniform in $D$, that is 
\begin{equation}\label{eq:quasiuniform}
\min_{\substack{ K \in \mathcal{T}_K \\ K\cap D \neq \emptyset}} h_K \geq c \max_{\substack{ M \in \mathcal{T}_h \\ M \cap D \neq \emptyset}} h_M, 
\end{equation}
for some constant $c$ independent of  levels of discretization.
} Then there exists $\delta\in (1/2,1]$ such that  for any  $\varepsilon\in[0,1/2)$ and  for  any  fine enough grid ${\cal T}_h$,
\begin{equation}\label{eq:FEMSuperconvergence}
  \|u-u_h\|_{H^{1+\varepsilon}(D')}\le  C \big [ (h^{\delta} h^{-\varepsilon}_D +h_D ^{1-\varepsilon} ) \|u-u_h\|_{H^{1}(\changeFirstRefereeG{\Omega_2})} +  h^{-\varepsilon}_D \|g_\Sigma-g_\Sigma^h\|_{L^2(\Sigma)}
     +  h_{D}^{d-\varepsilon} \|u\|_{H^{d+1}(D)}\big]
 \end{equation}
  with  $C>0$ depending on $\varepsilon$,  $D$ and $D'$, and $h_{D}$ being the maximum of the diameters of  the elements of the grid contained in $D$.
\end{theorem} 
\begin{proof}
 For \eqref{eq:FEMconvergence} we refer to \cite{ScZh:1990} (see also \cite{MR2033124}). 
 Estimate  \eqref{eq:FEMSuperconvergence} can be derived  using the  superconvergence of the FEM solution in the interior of the computational domain where the solution is smooth (actually analytic) cf. \cite{MR0373325}.  We give a proof of these results in {Corollary \ref{cor:a.5} in  Appendix  B
 for the sake of completeness}. 
\end{proof}

Let us point out that the constant $\delta\in(1/2,1]$ in \eqref{eq:FEMSuperconvergence} depends on the regularity of the dual problem with homogeneous Dirichlet condition and a right-hand-side in $L^2(\Omega)$. Therefore, for convex polygonal domains in $\mathbb{R}^2$, we can take $\delta=1$~\cite{Gr:2011}.

\paragraph{Assumption 1} Assume that for some open domain $D\subset \Omega_2\setminus\overline{\Omega}_0$ with $\Gamma\subset D$   there exists $\varepsilon_0>0$ such that  the sequence of grids $\{{\cal T}_h\}_h$ {is quasi-uniform in $D$ and} satisfies  
\[
 h^{1/2}  h_D^{-\varepsilon_0}\to 0
\]
where $h_D$ is the maximum of the diameters of the elements of the grid ${\cal T}_h$ having non-empty intersection with $D$.  \hfill $\Box$ \\[1ex]

We note that this assumption allows locally refined grids but {introduces} a very weak restriction on the ratio between the larger element in $\Omega_2$ and the smaller element in $D$. However, since the exact solution is smooth on $D$, it is reasonable to expect that small elements are not  going to be used in this subdomain.  

\begin{lemma}\label{lemma:5.4}
Let $D$ and the sequence of grids $\{{\cal T}_h\}$ be as in Assumption 1. Then,
\changeFirstRefereeG{for any open subset $D'$ of $D$ containing $\Gamma$ and}
 for any $\varepsilon\in[0,1/2)$,   there exists $C>0$ such that 
\[
\begin{aligned}
 & \hspace{-0.5in} \|\KSigmaGamma f_\Sigma - \KSigmaGamma^h {\rm P}^h_{\Sigma}f_\Sigma   \|_{H^{1/2+\varepsilon}_{{\rm per}}}\\
  \le&\   C h^{1/2}  h_D^{-\varepsilon}\big[ \inf_{v_h\in\mathbb{P}_{h,d}}\|\KSigmaOmegaTwo  f_\Sigma-v_h\|_{\changeFirstRefereeG{H^1(\Omega_2)}}
 +\|  f_\Sigma-{\rm P}^h_{\Sigma}f_\Sigma  \|_{H^{1/2}(\Sigma)}\big]
 + C h_D^{d-\varepsilon}\| f_\Sigma\|_{L^2(\Sigma)}.
\end{aligned}
 \]
\end{lemma}
\begin{proof}
Notice that 
\[
 \|f_\Sigma  -  {\rm P}^h_{\Sigma}f_\Sigma  \|_{L^2(\Sigma)}=
  \|f_\Sigma  -  {\rm P}^h_{\Sigma}f_\Sigma-  {\rm P}^h_{\Sigma}  (f_\Sigma  -  {\rm P}^h_{\Sigma}{f_\Sigma})\|_{L^2(\Sigma)}\le 
 C h^{1/2}\|f_\Sigma  -  {\rm P}^h_{\Sigma}f_\Sigma  \|_{H^{1/2}(\Sigma)}.
\]
Take  $ D'$ so that $\Gamma\subset D'\subset \overline{D}'\subset D$.  We claim that there exist $C$ independent of $f_\Sigma$, so that
\begin{equation}\label{eq:4.13}
 \|\KSigmaD f_\Sigma\|_{H^{d+1}(D)}\le C \|f_\Sigma\|_{L^2(\Sigma)}.
\end{equation}
This can be seen as consequence of that the differential equations in $D$ becomes the homogeneous Helmholtz equation and $\Gamma$ is sufficiently far away from \changeFirstReferee{the boundary of} $D$. 

The continuity of the trace operator $\gamma_\Gamma: H^{1+\varepsilon}(D')\to H^{1/2+\varepsilon}(\Gamma)\sim H^{1/2+\varepsilon}_{{\rm per}}$ and a direct application of  (\ref{eq:FEMconvergence}-\ref{eq:FEMSuperconvergence}) yield 
\[
\begin{aligned}
 \|\KSigmaGamma f_\Sigma - &\KSigmaGamma^h {\rm P}^h_{\Sigma}f_\Sigma   \|_{H^{1/2+\varepsilon}_{{\rm per}}}\\
  \le&\   C( h^{\delta}  h_D^{-\varepsilon}+h_D^{1-\varepsilon}+
  h ^{1/2-\varepsilon})\big[ \inf_{v_h\in\mathbb{P}_{h,d}}\|\KSigmaOmegaTwo  f_\Sigma-v_h\|_{H^1( \Omega_2 )}
 +\|  f_\Sigma-{\rm P}^h_{\Sigma}f_\Sigma  \|_{H^{1/2}(\Sigma)}\big]\\
 & \ \ + C h_D^{d-\varepsilon}\| f_\Sigma\|_{L^2(\Sigma)}.
\end{aligned}
 \]
Using that $\delta \ge 1/2$ the result is proven.
\end{proof}

We are ready to prove the convergence, in operator norm, of the first off diagonal block in \eqref{eq:5.9} to the corresponding one in \eqref{eq:BEMFEM:0b}

\begin{proposition}\label{prop:5.4}
For any  $0< \varepsilon\le \varepsilon_0$ and for any $t\ge 0$ there exists $C>0$ independent of $f_\Sigma$, $h$ and $N$ such that 
\begin{equation}
  \label{eq:01:prop:5.5}
  \begin{aligned}
\|\KSigmaGamma f_\Sigma &- {\rm Q}_{N}\KSigmaGamma^h   {\rm P}^h_{\Sigma}f_\Sigma   \|_{H^{1/2 }_{{\rm per}}}\\
\le &\  
C \big(1+N^{ -\varepsilon } h_{D}^{-\varepsilon }\big)h^{1/2} \big[ \inf_{v_h\in\mathbb{P}_{h,d}}\|\KSigmaOmegaTwo  f_\Sigma-v_h\|_{H^1(\Omega_2)}
 +\|  f_\Sigma- {\rm P}^h_{\Sigma}f_\Sigma   \|_{H^{1/2}(\Sigma)}\big]\\
 & \ \ + C (N^{ -\varepsilon } h_D^{d-\varepsilon}+h_D^{d}+ N^{-t})\| f_\Sigma\|_{L^2(\Sigma)}.
\end{aligned}
\end{equation}
In particular,  we have the convergence in operator norm:
\begin{equation}\label{eq:01:prop:5.5b}
 \|     \KSigmaGamma -{\rm Q}_{N}\KSigmaGamma^h   {\rm P}^h_{\Sigma}  \|_{H^{1/2}(\Sigma)\to H^{1/2}_{{\rm per}}}\to 0\quad \text{as }(N,h)\to (\infty,0).
\end{equation}
\end{proposition}
\begin{proof}
The identity 
\begin{equation}\label{eq:5:13}
 \KSigmaGamma-
 {\rm Q}_{N}\KSigmaGamma^h {\rm P}^h_{\Sigma}       = \big(
{\rm I}-{\rm Q}_{N}\big) \KSigmaGamma+
 ({\rm I}-{\rm Q}_{N})\big(\KSigmaGamma^h {\rm P}^h_{\Sigma}   -  \KSigmaGamma \big)  +\big(\KSigmaGamma-\KSigmaGamma^h {\rm P}^h_{\Sigma}  \big), 
\end{equation}
and the estimates
\begin{eqnarray*}
\|  ({\rm I}-{\rm Q}_{N}) \KSigmaGamma f_\Sigma\|_{H^{1/2 }_{{\rm per}}}&\le& 
C N^{-t}\|\KSigmaGamma f_\Sigma\|_{H^{t+1/2 }_{{\rm per}}}\le 
C'N^{-t}\|f_\Sigma\|_{\changeV{L^2(\Sigma)}}\\
 \| ({\rm I}-{\rm Q}_{N})(\KSigmaGamma^h {\rm P}^h_{\Sigma}   -  \KSigmaGamma \big)f_\Sigma\|_{H^{1/2}_{{\rm per}}}&\le& 
 C_{\varepsilon_1} N^{ -{\varepsilon_1} }\|(\KSigmaGamma^h {\rm P}^h_{\Sigma}   -  \KSigmaGamma \big)f_\Sigma\|_{H^{1/2+{\varepsilon_1}}_{{\rm per}}},
 \end{eqnarray*}
which are consequences of \eqref{eq:5.7} and \eqref{eq:4.13}, yield
\begin{multline*}
\|   \KSigmaGamma f_\Sigma-{\rm Q}_{N}\KSigmaGamma^h {\rm P}^h_{\Sigma} f_\Sigma \|_{H^{1/2 }_{{\rm per}}}\\
\le   
 C_\varepsilon N^{ -\varepsilon_1}\|(\KSigmaGamma^h {\rm P}^h_{\Sigma}   -  \KSigmaGamma \big)f_\Sigma\|_{H^{1/2+\varepsilon_1}_{{\rm per}}} + C'N^{-t}\|f_\Sigma\|_{\changeV{L^2(\Sigma)}}
 + \|(\KSigmaGamma^h {\rm P}^h_{\Sigma}   -  \KSigmaGamma \big)f_\Sigma\|_{H^{1/2}_{{\rm per}}}
 \end{multline*}
Applying Lemma \ref{lemma:5.4} twice (with $\varepsilon=\varepsilon_1$ for the first term and $\varepsilon =0$ for the third one) yield \eqref{eq:01:prop:5.5}.  
Consequently, \eqref{eq:01:prop:5.5b} follows. 
\end{proof}

\subsection{Convergence for the BEM scheme}

The  inverse inequality
\begin{equation}\label{eq:inverseIneq}
\| \varphi_{{\rm per}}^N\|_{H^t_{{\rm per}}}\le N^{t-s}\| \varphi_{{\rm per}}^N\|_{H^s_{{\rm per}}},\quad   t\ge s,\quad  \forall  \varphi_{{\rm per}}^N\in\mathbb{T}_N, 
\end{equation}
that is straightforward to derive from \eqref{eq:def:Tn} and \eqref{eq:4.02}, will be used repeatedly in this section. 
%
The first result in this subsection summarizes the convergence of the BEM solver in a format that will be used later. This is based on the convergence in norm of the approximation operator ${\cal L}_k^N$ to the continuous counterpart. We recall that ${\cal L }_k:H^s_{{\rm per}}\to H^{s}_{{\rm per}}$ is continuous for any $s\in \mathbb{R}$.  

\begin{theorem}\label{Th:01}
Fix  $t\ge s >1/2$. Then,  there exists $C>0$ such that for any $N$ large enough and for any $f_{\rm per}\in H^t_{\rm per}$,
 \begin{equation}\label{eq:01:Th:01}
  \|{\cal L}_{k}f_{\rm per}-{\cal L}_{k}^{N}f_{\rm per}   \|_{H_{\rm per}^s}\le C N^{s-t-\alpha}\|f_{\rm per} \|_{H_{\rm per}^t}, \qquad  \text{with }\alpha=\min\{s,1\}.
 \end{equation}
Therefore, ${\cal L}_{k}^{N}:H^s_{{\rm per}}\to H^{s}_{{\rm per}}$ is uniformly continous in $N$ for any $s>1/2$.

 Moreover, for $s,t\ge 0$, with $t\ge \max\{s-1,0\}$ there exits $C>0$ such that, for $N$ large enough, 
\begin{equation}\label{eq:01b:Th:01}
  \| \mathrm{Q}_N{\cal L}_{k}\mathrm{P}_Nf_{\rm per}-\mathrm{Q}_N{\cal L}_{k}^{N}\mathrm{P}_Nf_{\rm per}\|_{H_{\rm per}^s}\le
  C N^{s-t-\alpha}\|{\mathrm{P}_Nf_{\rm per}}\|_{H_{\rm per}^t}\le C N^{s-t-\alpha}\|f_{\rm per}\|_{H_{\rm per}^t}. 
 \end{equation}

\end{theorem}
\begin{proof}
The estimates, recall \eqref{eq:4.2} and \eqref{eq:4.3}, 
\[
 \|\mathrm{K}_{k}^{N}-\mathrm{K}_{k}\|_{H_{\rm per}^t\to H_{\rm per}^s}
 + \|\mathrm{V}_{k}^{N}-\mathrm{V}_{k}\|_{H_{\rm per}^t\to H_{\rm per}^s}\le C N^{s-t-\alpha}
\]
with $s,\ t,\ \alpha$ as in the statement of the Theorem (see Chapter 12 and 13 in \cite{Kress:2014} or, for a more detailed proof, \cite[Th. 3.1]{MR3526814}) proves that $ {\cal L}_{k}^{N}:H_{\rm per}^s\to H_{\rm per}^s$ for $s>1/2$ is well defined, for $N$ large enough, and it is uniformly bounded. Estimate \eqref{eq:01:Th:01} follows now from the identity 
\[
 {\cal L}_{k}-{\cal L}_{k}^{N}= 
  {\cal L}_{k}^N\big[  \big({\cal L}_{k}^N)^{-1} -  {\cal L}_{k}^{-1}\big] {\cal L}_{k}=  
 {\cal L}_{k}^N\big[(\mathrm{K}^{N}_{k}-\mathrm{K}_{k})-{\rm i}k
(\mathrm{V}_{k}^{N}-\mathrm{V}_{k})\big] {\cal L}_{k}.
\]

To prove the second estimate \eqref{eq:01b:Th:01}, we proceed in two steps: For $s\in[1,\infty)$, $t\ge s-1$ we have
\begin{eqnarray}
   \|\mathrm{Q}_N{\cal L}_{k}^{N}\mathrm{P}_Nf_{\rm per}-\mathrm{Q}_N{\cal L}_{k}\mathrm{P}_Nf_{\rm per}\|_{H_{\rm per}^s}&\le&
   C\| ({\cal L}_{k}^{N}-{\cal L}_{k})\mathrm{P}_Nf_{\rm per}\|_{H_{\rm per}^s}\le C' N^{s-t-2}\|\mathrm{P}_Nf_{\rm per}\|_{H_{\rm per}^{t+1}}\nonumber\\
   &\le&   C' N^{s-t-1}\|\mathrm{P}_Nf_{\rm per}\|_{H_{\rm per}^{{t}}}.\label{eq:5.17}
\end{eqnarray} 
(Notice that in the last step we have used the inverse inequality \eqref{eq:inverseIneq}.) 
On the other hand, for $s\in[0,1]$ and $t\ge 0$, we make use of the bound $\|{\rm Q}_N g
_{{\rm per}}\|_{H_{\rm per}^0}\le C \|g_{{\rm per}}\|_{H_{\rm per}^1}$ to derive 
\begin{eqnarray}
 \|\mathrm{Q}_N{\cal L}_{k}^{N}\mathrm{P}_Nf_{{\rm per}}-\mathrm{Q}_N{\cal L}_{k}\mathrm{P}_Nf_{{\rm per}}\|_{H_{\rm per}^s}&\le& 
 C\| ({\cal L}_{k}^{N}-{\cal L}_{k})\mathrm{P}_Nf_{{\rm per}}\|_{H_{\rm per}^1}\le C' N^{-t-1}\|\mathrm{P}_Nf_{{\rm per}}\|_{H_{\rm per}^{t+1}}\nonumber\\
 &\le& C' N^{-t}\|\mathrm{P}_Nf_{{\rm per}}\|_{H_{\rm per}^{t}}.\label{eq:5.18}
\end{eqnarray}
Estimates \eqref{eq:5.17} and \eqref{eq:5.18} yield the desired result \eqref{eq:01b:Th:01}.
\end{proof}

\begin{lemma}\label{lemma:5.6}
For any domain $D$ with $\overline{D}\cap \Gamma =\emptyset$ and any   $r,s$ there exists $C>0$ such that for any $N$ and  $ \varphi_{{\rm per}}^N\in\mathbb{T}_N$,
 \begin{equation}\label{eq:02:Th:01}
 \|\mathrm{DL}_{k}^{N} \varphi_{{\rm per}}^N-\mathrm{DL}_{k} \varphi_{{\rm per}}^N\|_{H^s(D)}+
 \|\mathrm{SL}_{k}^{N}  \varphi_{{\rm per}}^N-\mathrm{SL}_{k} \varphi_{{\rm per}}^N\|_{H^s(D)}\le C N^{-r}\| \varphi_{{\rm per}}^N\|_{H_{\rm per}^0}.
 \end{equation}
 \end{lemma}
 \begin{proof} 
Since the kernels of the integral operators are smooth, the estimate follows from the aliasing effect of the trapezoidal rule for periodic functions. Indeed, for $|n|\le N$ we have for any $g$ smooth enough
\[
 \bigg| \frac{\pi}{N} \sum_{j=0}^{2N-1}  g_{{\rm per}}(t_j)e_{-n}(t_j)-\int_0^{2\pi} g_{{\rm per}}(t)e_{-n}(t)\,{\rm d}t\bigg|\le \frac{1}{2\pi}\sum_{\ell\ne 0} |\widehat{g}_{{\rm per}}(n+2\ell N)|,\quad 
 e_n(t):=\exp({\rm i}nt)
\]
(see for instance \cite{Kress:2014,SaVa:2002}). Hence, for any $\varphi_{{\rm per}}^N\in\mathbb{T}_{N}$ 
\[
  \bigg| \frac{\pi}{N} \sum_{j=0}^{2N-1}  g_{{\rm per}}(t_j){\varphi_{{\rm per}}^{N}(t_j)}-\int_0^{2\pi} g_{{\rm per}}(t){\varphi_{{\rm per}}^N(t)}\,{\rm d}t\bigg|\le \frac{1}{2\pi}\sum_{n\in\mathbb{Z}_N} |\widehat{\varphi}_{{\rm per}}(n)|\sum_{\ell\ne 0} |\widehat{g}_{{\rm per}}(-n+2\ell N)|.
\]
The last term can be easily bounded using the Cauchy-Schwarz inequality:
\begin{eqnarray*}
 \sum_{n\in\mathbb{Z}_N} | \varphi_{{\rm per}}^N(n)|\sum_{\ell\ne 0} | \widehat{g}_{{\rm per}}(-n+2\ell N)|&\le& {(2N)^{-r}}
\| \varphi_{{\rm per}}^N\|_{H_{\rm per}^0}  \Bigg[\sum_{n\in\mathbb{Z}_N}  \underbrace{ \Bigg(\sum_{\ell\ne 0} \frac{1}{|\ell-n/(2N)|^{2r} }\bigg)}_{=:C_r(n/(2N))}\\
&&\qquad \times 
 \bigg(\sum_{\ell\ne 0} |-n+2\ell N|^{2r}| \widehat{g} _{{\rm per}}(-n+2\ell N)|^2\bigg)\Bigg]^{1/2}\\
 &\le&C_r  N^{-r}\| \varphi_{{\rm per}}^N\|_{H_{\rm per}^0} \|g_{{\rm per}}\|_{H_{\rm per}^r} 
\end{eqnarray*}
which is valid for any $r>1/2$.  (We have used above that the series function $C_r(z)$ is bounded for   $|z|\le 1/2$)
\end{proof}

\begin{corollary}\label{cor:5.6}
For any $s\ge 0$, there exists $C>0$ so that for any $N$,
\[
 \|\mathrm{Q}_N{\cal L}_{k}^{N}\mathrm{P}_N f_{{\rm per}}\|_{H_{\rm per}^s}\le C \|\mathrm{P}_Nf_{{\rm per}}\|_{H_{\rm per}^s} \le  C \|f_{{\rm per}}\|_{H_{\rm per}^s} .
\]
\end{corollary}
\begin{proof}
The identity 
	\begin{equation}\label{eq:def:M}
	 {\cal L}_k =  2{\rm I} -{\rm M}_k, \quad \mathrm{M}_k :=2 {{\cal L}_k}({\cal L}_k^{-1}-\tfrac12{\rm I}) =2 {\cal L}_k({\rm K}_k-{\rm i}k {\rm V}_k)
	\end{equation}
and the mapping properties ${\rm V}_k-{\rm i}k {\rm K}_k:H^s_{{\rm per}}\to H_{{\rm per}}^{s+1}$ (see for instance, \cite[Section 6.2]{SaVa:2002})
yield
\begin{eqnarray}
    \|\mathrm{Q}_N{\cal L}_{k}\mathrm{P}_N f_{{\rm per}}-{\cal L}_{k} \mathrm{P}_N  f_{{\rm per}}\|_{H_{\rm per}^s}&=&
   \|\mathrm{Q}_N{\rm M}_{k}\mathrm{P}_N f_{{\rm per}}-{\rm M}_{k} \mathrm{P}_N f_{{\rm per}}\|_{H_{\rm per}^s}\nonumber\\
   &\le& 
   C' N^{-1} \|{\rm M}_{k} \mathrm{P}_N f_{{\rm per}}\|_{H_{\rm per}^{s+1}}\le C N^{-1}\|f_{{\rm per}}\|_{H_{\rm per}^{s}}\label{eq:5.19}
   \end{eqnarray}
where we have used also \eqref{eq:5.7} and \eqref{eq:5.6}.
The result follows readily from the decomposition
\begin{eqnarray*}
 \|\mathrm{Q}_N{\cal L}_{k}^{N}\mathrm{P}_N f_{{\rm per}}\|_{H_{\rm per}^s}&\le& 
 \|{\cal L}_{k}\mathrm{P}_N f_{{\rm per}}\|_{H_{\rm per}^s}+
 \|\mathrm{Q}_N{\cal L}_{k}^{N}\mathrm{P}_N f_{{\rm per}}-\mathrm{Q}_N{\cal L}_{k} \mathrm{P}_N  f_{{\rm per}}\|_{H_{\rm per}^s} \\ &&+ 
    \|\mathrm{Q}_N{\cal L}_{k}\mathrm{P}_N f_{{\rm per}}-{\cal L}_{k} \mathrm{P}_N  f_{{\rm per}}\|_{H_{\rm per}^s},
\end{eqnarray*}
and {the estimates}  \eqref{eq:5.19}  and \eqref{eq:01b:Th:01}  in Theorem \ref{Th:01}.
\end{proof}

We are ready to prove the convergence of the corresponding block in \eqref{eq:5.9}.

\begin{proposition}\label{prop:5.8}
For any $t\ge 0$, there exists $C>0$ such that for any $h$, $N$ and $f_{\rm per} \in H^t_{{\rm per}}$, 
\begin{equation}
  \label{eq:01:prop:5.8}
{\|  \KGammaSigma f_{\rm per}-{\rm Q}_\Sigma^h\KGammaSigma^N {\rm P}_{N}f_{\rm per}\|_{H^{1/2}(\Sigma)}}\le C ( N^{-t }\|f_{\rm per}\|_{H^{t}_{{\rm per}}} + h_\Sigma^{d+1/2} \|f_{\rm per}\|_{H^{0}_{{\rm per}}}). 
\end{equation}
In particular,  we have the convergence in operator norm:
\begin{equation}\label{eq:01:prop:5.8b}
 \| { \KGammaSigma f_{\rm per}-{\rm Q}_\Sigma^h\KGammaSigma^N {\rm P}_{N}f_{\rm per}}  \|_{ H^{1/2}_{{\rm per}}\to H^{1/2} (\Sigma)}
 \to 0\quad \text{as }(N,h)\to (\infty,0).
\end{equation}
\end{proposition}
\begin{proof}
For the purpose of this proof, we  define 
\begin{equation}
  \label{eq:02:prop:5.8}
 \RGammaSigma :=\gamma_\Sigma (\mathrm{DL}_{k}-{\rm i}k\mathrm{SL}_{k}),\quad
 \RGammaSigma^{N} :=\gamma_\Sigma (\mathrm{DL}_{k}^{N}-{\rm i}k\mathrm{SL}_{k}^{N}).
\end{equation}
Clearly  $ \RGammaSigma:H^0_{{\rm per}}\to H^{t}(\Sigma)$  is continuous, for any $t$, and, from Lemma \ref{lemma:5.6}, (see also \eqref{eq:4.2})  
\begin{equation}
  \label{eq:03:prop:5.8}
 \| \RGammaSigma{\rm Q}_N \varphi_{\rm per}- \RGammaSigma^N {\rm Q}_N \varphi_{\rm per}\|_{H^{m}(\Sigma)}
\le C_t N^{-t}\|{\rm Q}_N \varphi_{\rm per}\|_{H_{\rm per}^0}
\end{equation}
for any $t$ and $m$, with $C$ independent of $\varphi_{{\rm per}}$ and $N$. We also have for any compact domain $D$ far away from $\Gamma$  and for any $t\ge 0$, 
\begin{equation}
  \label{eq:03b:prop:5.8}
 \|\RGammaSigma \varphi_{\rm per}\|_{H^m(D)}\le C \|\varphi_{\rm per}\|_{H^{-t}_{\rm per}}
\end{equation}
which in particular implies cf. \eqref{eq:4.5b}
\begin{equation}
  \label{eq:03c:prop:5.8}
 \|({\rm I}-{\rm Q}_\Sigma^h )\RGammaSigma \varphi_{\rm per}\|_{H^{1/2}(\Sigma)}\le C h_{\Sigma}^{d+1/2}\|\varphi_{\rm per}\|_{H^{-t}_{\rm per}},\qquad
 \|{\rm Q}_\Sigma^h \RGammaSigma \varphi_{\rm per}\|_{H^{1/2}(\Sigma)}\le C\|\varphi_{\rm per}\|_{H^{-t}_{\rm per}}.
\end{equation}
 
Write now
\begin{equation}
  \label{eq:04:prop:5.8}
\begin{aligned} 
\KGammaSigma - 
 {\rm Q}_\Sigma^h\KGammaSigma^N {\rm P}_{N}  \ =\ & 
 \RGammaSigma{\cal L}_{k}-
 {\rm Q}_\Sigma^h  \RGammaSigma^{N} {\rm Q}_N{\cal L}_{k}^N \mathrm{P}_N  \\
 \ =\ & 
({\rm I}- {\rm Q}_\Sigma^h )  \RGammaSigma{\cal L}_{k} +{\rm Q}_\Sigma^h\RGammaSigma  {\cal L}_{k}(  {\rm I}-{\rm P}_N)\\
& +{\rm Q}_\Sigma^h  \RGammaSigma({\rm I}-{\rm Q}_N) {\cal L}_{k}{\rm P}_N  +
  {\rm Q}_\Sigma^h \RGammaSigma( {\rm Q}_N{\cal L}_{k}-{\rm Q}_N{\cal L}_{k}^N)\mathrm{P}_N\\
 \ &  +
 {\rm Q}_\Sigma^h(  \RGammaSigma-\RGammaSigma^{N}  ){\rm Q}_N{\cal L}_{k}^N \mathrm{P}_N\\
 =:\ &
 {\rm L}^{h}_{1}+{\rm L}^{h,N}_{2}+{\rm L}^{h,N}_{3}+{\rm L}^{h,N}_{4}+{\rm L}^{h,N}_{5}.
 \end{aligned}
 \end{equation}
 Let us bound these terms now. From \eqref{eq:03c:prop:5.8} and the continuity ${\cal L}_k: H_{{\rm per}}^s\to H_{{\rm per}}^s$, 
   \[ 
   \| {\rm L}^{h}_{1}f_{{\rm per}} \|_{H^{1/2}(\Sigma)}\le  C h_\Sigma^{d+1/2}  \|{\cal L}_k f_{{\rm per}}\|_{H^{0}_{{\rm per}}} \le  C'
 h_\Sigma^{d+1/2} \| f_{{\rm per}} \|_{H^{0}_{\rm per}}.
 \]
 Proceeding similarly we derive  from \eqref{eq:5.6}, 
  \[
   \| {\rm L}^{h,N}_{2}f_{{\rm per}} \|_{H^{1/2}(\Sigma)}\le C \| ( {\rm I}-{\rm P}_N )f_{{\rm per}} \|_{H^{-t}_{\rm per}}\le C N^{-t}\|f_{{\rm per}} \|_{H^{0}_{\rm per}}, 
 \]
 and, by \eqref{eq:5.7}, 
  \[
  \|{\rm L}^{h,N}_{3} f_{{\rm per}}\|_{H^{1/2}(\Sigma)}\le C \|( {\rm I}-{\rm Q}_N){\cal L}_{k}{\rm P}_Nf_{{\rm per}} \|_{H^{0}_{\rm per}} \le C''N^{-t}\|f_{{\rm per}}\|_{H^{t}_{\rm per}}. 
 \]
 The fourth term is bounded using  \eqref{eq:01b:Th:01} in Theorem \ref{Th:01},
 \[
  \|{\rm L}^{h,N}_{4} f_{{\rm per}}\|_{H^{1/2}(\Sigma)}\le C \|( {\rm Q}_N{\cal L}_{k}{\rm P}_N-{\rm Q}_N{\cal L}^N_{k}{\rm P}_N )f_{{\rm per}} \|_{H^{0}_{\rm per}} \le C''N^{-t}\|f_{{\rm per}}\|_{H^{t}_{\rm per}}. 
 \]
 Finally, \eqref{eq:03c:prop:5.8} again, \eqref{eq:03:prop:5.8} with Corollary \ref{cor:5.6} yield,
 \[
 \| {\rm L}^{h,N}_{5}f_{{\rm per}} \|_{H^{1/2}(\Sigma)}\le C N^{-t}\|  {\rm Q}_N {\cal L}_{k}^N{\rm P}_N f_{{\rm per}} \|_{H^{0}_{\rm per}}\le C N^{-t}\|f_{{\rm per}} \|_{H^{0}_{\rm per}}.
 \]
Gathering these bounds   yields the estimate \eqref{eq:01:prop:5.8}, and consequently  \eqref{eq:01:prop:5.8b}  holds. 
%
%
%
 
\end{proof}

\subsection{Convergence of the full scheme}
We are ready to prove the main result of this paper: stability and convergence for the FEM-BEM numerical algorithm.
\begin{theorem}
For any $N$ large enough and ${\cal T}_h$ sufficiently fine satisfying Assumption 1, the mapping
\[
 {\cal I}-{\cal K}_{h,N}: H^{1/2}(\Sigma)\times H^{1/2}_{{\rm per}}\to H^{1/2}(\Sigma)\times H^{1/2}_{{\rm per}}
\]
is uniformly bounded, invertible and with inverse uniformly bounded. 

Moreover, if $(f_\Sigma,f_\Gamma)$ is the solution of \eqref{eq:BEMFEM:0} and  $(f^h_\Sigma,f^N_{{\rm per}})$ that of \eqref{eq:NhBEMFEM:0}, then for any $r\ge 0$  and $0<\varepsilon\le \varepsilon_0$, with $\varepsilon_0$ as in Assumption 1,   we have the following estimate, with $f_{\rm per} = f_\Gamma\circ{\bf x}$, 
\begin{equation}\label{eq:5.26}
 \begin{aligned}
\|f_\Sigma-f_\Sigma^h\|_{H^{1/2}(\Sigma)} & +
  \|f_{\rm per}-f_{\rm per}^N\|_{H^{1/2}_{\rm per}}\\
   \le\   
  C&\Big[\|{\gamma_{\Sigma} u^{\rm inc}-{\rm Q}^h_{\Sigma}\gamma_{\Sigma} u^{\rm inc}}\|_{H^{1/2}(\Sigma)}+
  \|{u^{\rm inc}\circ{\bf x}-{\rm Q}_N  u^{\rm inc}\circ{\bf x}    }  \|_{ H^{1/2}_{\rm per}}\\
  & \ +      
  \big(h_D^{-\varepsilon}  N^{-\varepsilon}+1\big)h^{1/2} \big[\inf_{v_h\in\mathbb{P}_{h,d}}\|\KSigmaOmegaTwo  f_\Sigma-v_h\|_{H^1(\Omega_2)}+ \| f_\Sigma-{\rm P}_\Sigma^h f_\Sigma\|_{H^{1/2}(\Sigma)}\big]
 \\
& \ +   ( N^{-\varepsilon} h_D^{d-\varepsilon} +h_D^{d }+N^{-r})\| f_\Sigma\|_{L^2(\Sigma)}
   +  ( N^{-t}+ h_\Sigma^{d+1/2})\|f_{\rm per}\|_{H^{t}_{\rm per}}\Big]
 \end{aligned}
\end{equation}
with $C$ independent of $f_\Sigma$, $f_{{\rm per}}$, $h$ and $N$.
\end{theorem} 
\begin{proof}
 From \eqref{eq:01:prop:5.5b} and \eqref{eq:01:prop:5.8b} in  Propositions  \ref{prop:5.4} and \ref{prop:5.8} we conclude
 \[
\|{\cal K}-{\cal K}_{h,N}\|_{H^{1/2}(\Sigma)\times H^{1/2}_{\rm per}\to H^{1/2}(\Sigma)\times H^{1/2}_{\rm per}}\to 0,\quad \text{as }(h,N)\to (0,\infty)
 \]
 which proves the first part of the theorem. 
 
 On the other hand, for small enough $c>0$, independent of $h$ and $N$, it holds
 \begin{equation}
  \begin{aligned}
  c\big(  \|f_\Sigma-f_\Sigma^h\|_{H^{1/2}(\Sigma)} & +
  \|f_{\rm per}-f_{\rm per}^N\|_{H^{1/2}_{\rm per}}\big)\ \le\  \left\|\left(
  {\cal I}-{\cal K}_{h,N}\right)\begin{bmatrix}
                          f_\Sigma-f_\Sigma^h\\
                          f_{\rm per}-f_{\rm per}^N
                          \end{bmatrix}
  \right\|_{H^{1/2}(\Sigma)\times H^{1/2}_{\rm per}}\\
   \le\ & \left\| \begin{bmatrix*}[r]
                         \gamma_{\Sigma} u^{\rm inc}- {\rm Q}^h_{\Sigma}\gamma_{\Sigma} u^{\rm inc}\\
                         u^{\rm inc} \circ{\bf x} -{\rm Q}_N  u^{\rm inc}\circ{\bf x}   
                          \end{bmatrix*}\right\|_{H^{1/2}(\Sigma)\times H^{1/2}_{\rm per}}
                          \\ 
                          & +                          
   \left\|\left({\cal K}- {\cal K}_{h,N}\right)\begin{bmatrix}
                          f_\Sigma\\
                          f_{\rm per}
                          \end{bmatrix}
  \right\|_{H^{1/2}(\Sigma)\times H^{1/2}_{\rm per}}
    \end{aligned}
 \end{equation}
 and   the estimate \eqref{eq:5.26} follows   from Propositions \ref{prop:5.4} and \ref{prop:5.8}.
\end{proof}

\begin{theorem}\label{theo:5.10}
Let  $(u, \omega) = (\KSigmaOmegaTwo f_\Sigma,\KGammaOmegaOneC f_\Gamma)$ be the exact solution of \eqref{eq:BEMFEM},  and let  $(u_h, \omega_N) 
:=(\KSigmaOmegaTwo^h f^h_\Sigma,\KGammaOmegaOneC^N f^N_{\rm per})$  be that of the FEM-BEM system \eqref{eq:NhBEMFEM}.  Then, for any compact set  $D\subset\mathbb{R}^d\setminus \overline{\Omega}_1$, $r\ge 0$, $t>d+1/2$ there exist $C$ such that
\begin{eqnarray}
& &  \hspace{-0.5in} \|u-u_h\|_{H^{1}(\Omega_2)}+ \| \omega-\omega_N\|_{H^{r}(D)} \nonumber \\ &\le& C\big(h_D^{d-\varepsilon} N^{-\varepsilon}+h_\Sigma^{d+1/2}+N^{-t}+h_D^d\big)\|u^{\rm inc}\|_{H^{t+1}(\Omega_2)}
 +C\inf_{v_h\in\mathbb{P}_{h,d}}\|u-v_h\|_{H^1(\Omega_2)}.\qquad \label{eq:estimate:01}
\end{eqnarray}
\end{theorem}
\begin{proof}
Notice that as consequence of Theorem \ref{theo:main:01}, and with $f_{\rm per} =f\circ{\bf x}$  as before, 
\[
 \|f_\Sigma\|_{H^{t+1/2}(\Sigma)}+ \|f_{\rm per}\|_{H^{t+1/2}_{{\rm per}}}\le C_t \|u^{\rm inc}\|_{H^{t+1}(\Omega_2)}, 
\]
and, cf. \eqref{eq:4.5} and \eqref{eq:5.7},
\begin{eqnarray*}
 \|{\gamma_{\Sigma} u^{\rm inc}-{\rm Q}^h_{\Sigma}\gamma_{\Sigma} u^{\rm inc}}\|_{H^{1/2}(\Sigma)} &\le&  C h_\Sigma^{d+1/2}\| u^{\rm inc}\|_{H^{t+1}(\Omega_2)},\\
  \|{ u^{\rm inc}\circ{\bf x} - {\rm Q}_N  u^{\rm inc}\circ{\bf x}  } \|_{ H^{1/2}_{\rm per}} &\le& C N^{-t}\| u^{\rm inc}\|_{H^{t+1}(\Omega_2)}.
\end{eqnarray*}

Since
\[
  \| f_\Sigma-{\rm P}_\Sigma^h f_\Sigma\|_{H^{1/2}(\Sigma)}\le C' \inf_{p_h\in \gamma_\Sigma \mathbb{P}_{h,d}}
  \| f_\Sigma-p_h\|_{H^{1/2}(\Sigma)}\le C' 
 \|f_\Sigma-f_\Sigma^h\|_{H^{1/2}(\Sigma)}, 
 \]
the estimate \eqref{eq:5.26} yields 
\begin{eqnarray}
 \|f_\Sigma-f_\Sigma^h\|_{H^{1/2}(\Sigma)}+\|f_{\rm per}-f_{\rm per}^N\|_{H^{1/2}_{\rm per}}  &\le& c(h,N)
 \|f_\Sigma-f_\Sigma^h\|_{H^{1/2}(\Sigma)}\nonumber  \\
 &&+  C\big(h_D^{d-\varepsilon} N^{-\varepsilon}+h_D^d+N^{-t}+h_\Sigma^{d+1/2}\big)\|u^{\rm inc}\|_{H^{t+1}(\Omega_2)} \nonumber 
 \\&&
 +C'\inf_{v_h\in\mathbb{P}_{h,d}}\|u-v_h\|_{H^1(\Omega_2)}\label{eq:4:31}.
\end{eqnarray}  
with $\changeV{c(h,N)}\to 0$ as $(h,N)\to (0,\infty)$. (Actually $C'$ above can be shown to tend to zero as  $(h,N)\to (0,\infty)$). Taking $h$ sufficiently small and $N$ large enough,  say such that $c(h,N)<1/2$,  the estimate for the error $u-u_h$ follows from  \eqref{eq:FEMconvergence} in Theorem \ref{theo:5.2}.

Regarding the other term, we notice that by Lemma \ref{lemma:5.6}
\begin{eqnarray}
\| \omega-\omega_N\|_{H^{r}(D)} &=& 
\|({\rm DL}_k-{\rm i }k {\rm SL}_k){\cal L}_k f_{\rm per} -({\rm DL}_k^N-{\rm i }k {\rm SL}_k^N){\rm Q}_N{\cal L}_k^N f_{\rm per}^N \|_{H^{r}(D)}\nonumber \\
&\le& C\left(\|{\cal L}_k f_{\rm per}-{\rm Q}_N{\cal L}_k f_{\rm per}\|_{H^{0}_{\rm per}}+\|{\rm Q}_N {\cal L}_k f_{\rm per}-{\rm Q}_N{\cal L}_k^N {f^N_{{\rm per}}}\|_{H^{0}_{\rm per}}\right).\label{eq:5.28}
\end{eqnarray}
\changeV{Recalling that ${\rm M}_k$ in \eqref{eq:def:M}  is a pseudo-differential operator of order $-1$,  we obtain the bound}
\[
 \|{\cal L}_k f_{\rm per}-{\rm Q}_N{\cal L}_k f_{\rm per}\|_{H^{0}_{\rm per}}=\|({\rm I}-{\rm Q}_N){\rm M}_k f_{\rm per}\|_{H^{0}_{\rm per}}\le C N^{-t-3/2}\|f_{\rm per}\|_{H^{t+1/2}_{\rm per}}. 
\]
Next, we claim that for the second term in \eqref{eq:5.28} the following bound holds
\begin{equation}\label{eq:5.29}
\|{\rm Q}_N {\cal L}_k f_{\rm per}-{\rm Q}_N{\cal L}_k^N {f^N_{{\rm per}}} \|_{H^{0}_{\rm per}}\le C\|   \ {\rm P}_N f_{\rm per}-{f^N_{{\rm per}}}\|_{H^{1/2}_{\rm per}}+ N^{-t-1/2}\|f_{\rm per} \|_{H^{t+1/2}_{\rm per}} 
 \end{equation}
which, with estimate \eqref{eq:4:31}, should prove the result. Indeed, writing 
\begin{eqnarray*}
{\rm Q}_N {\cal L}_k f_{\rm per}-
 {\rm Q}_N{\cal L}_k^N {f^N_{{\rm per}}}&&\\
 &&\hspace{-3cm} = \ 
 {\rm Q}_N{\cal L}_k (  f_{\rm per}-{\rm P}_N f_{\rm per})  +
  (  {\rm Q}_N {\cal L}_k {\rm P}_N f_{\rm per}-{\rm Q}_N{\cal L}_k^N {\rm P}_N f_{\rm per}) 
 +
 {\rm Q}_N{\cal L}_k^N ({\rm P}_N f_{\rm per}-{f^N_{{\rm per}}}), 
\end{eqnarray*}
and using 
\begin{eqnarray*}
 \|{\rm Q}_N{\cal L}_k (  f_{\rm per}-{\rm P}_N f_{\rm per})\|_{H^0_{\rm per}}&\le& 
 \|{\cal L}_k (  f_{\rm per}-{\rm P}_N f_{\rm per})\|_{H^0_{\rm per}} + C N^{-1} \|{\cal L}_k (  f_{\rm per}-{\rm P}_N f_{\rm per})\|_{H^1_{\rm per}}\\
 &\le& C'N^{-t-1/2}\|f_{\rm per} \|_{H^{t+1/2}_{\rm per}} 
\end{eqnarray*}
 for the first term,  \eqref{eq:01b:Th:01} in Theorem \ref{Th:01} (with $s=\alpha = 0$) for the second one and 
Corollary \ref{cor:5.6}  for the  third term, \eqref{eq:5.29} follows.
\end{proof}

\section{Numerical experiments}\label{sec:num-exp} 
This section comprises three sets of numerical experiments to demonstrate our FEM-BEM algorithm and analysis.
As proved in Theorem~\ref{theo:5.10}, convergence of the FEM-BEM numerical solution is dictated 
by the best approximation  last term in~\eqref{eq:estimate:01}. The best approximation accuracy depends on the
smoothness of the exact solution $u$ of~\eqref{eq:theproblem}, induced by smoothness of the refractive index $n$ in the heterogeneous region. 

The first set of experiments is for the smooth solution case to observe the fast convergence of our method under optimal conditions. 
Motivation for the second and third set of experimstents are from the Janus particle configurations with non-smooth (only piecewise-continuous)
refractive indices.  For the second and third experiments, the total wave has limited regularity and belongs to $H^2(\Omega_2)\setminus H^{5/2}(\Omega)$. 
\changeSecondRefereeG{(The {\em discontinuity}    of the refractive index ${n}$ in $\Omega_2$ leads to
 the interior wave-field limitation that $\Delta u\not\in H^{1/2}(\Omega_2)$.)}
For the latter cases, we also demonstrate that our FEM-BEM algorithm  converges, with lower convergence rates, and good accuracy can be obtained using high-order finite elements
(such as  the FEM space spanned by quadratic or cubic splines).  
In particular for the multi-particle Janus-type configuration  experiments, we demonstrate that even quadratic FEM is not sufficient, highlighting the difficulty associated
with Janus configurations for wave propagation models and the need for efficient high-order FEM-BEM to compute approximations for the practical QoI such as the 
DSCS and OA-DSCS. We recall that 
these QoI are defined  in~\eqref{eq:OA-DSCS} and in the algorithm, the far-field is computed using the density based ansatz in~\eqref{eq:far-z-N}.

Basic setup of the three sets of  experiments is similar. Starting from an initial, coarse mesh ${\cal T}_H$, we 
consider a set of meshes ${\cal T}_H\subset {\cal T}_{H/2}\subset \cdots  {\cal T}_{h}\subset {\cal T}_{h/2}\subset\cdots {\cal T}_{h_e}$ obtained from successive uniform  refinement where each triangle is divided into four elements. For the BEM solver we proceed analogously, consider an initial, relatively low number $N = N_0$, and then we double the points: $\{N_0, 2N_0,\ldots, N_e\}$. 

One of  aims of the numerical experiments is  to demonstrate the error estimates proved  in Theorem \ref{theo:5.10}. This is carried out as follows: Let ${\cal T}_{h_e}$ be the finest 
of the chosen finite element grids and  let $N_e$ be the largest of the chosen BEM  discrete parameters. We then compute the solution $(u_{h_e/2},\omega_{2N_e})$ using 
next refinement steps. 
We will use this pair as our reference exact solution, i.e. with the notation of Theorem \ref{theo:5.10},  $(u_{h_e/2},\omega_{2N_e})\approx (u,\omega)$ so that we compute 
\begin{equation}\label{eq:D-err}
 \|u_h-u_{h_{e}/2}\|_{H^1(\Omega_2)} ; \qquad  \qquad \| {\rm Q}_{h_{e}/2} (\omega_N-\omega_{2N_e})\|_{H^1(D)},
 \end{equation}
 as the $H^1-$error for the finite element solution (total wave) and boundary element solution (scattered wave). Here, ${{\rm Q}_{h_{e}/2}}$ is the Lagrange interpolation operator on $\mathbb{P}_{h_e/2,d}$
which means  ${\rm Q}_{h_{e}/2} u_{h_{e}/2} = u_{h_e/2}$ and ${{\rm Q}_{h_{e}/2}} u_h =u_h$ for any $u_h$ in our list of experiments. Since both $\omega_N$ and $\omega$ are smooth functions  as long as $\overline{D}\subset{\rm ext}(\Gamma)$ (the region exterior to $\Gamma$), such a strategy gives a sufficiently good estimate for the true error of the numerical solution.  

In our computations we have taken $D$ in~\eqref{eq:D-err} to be  ${\cal T}_{h_e/2}\cap {\rm Ext}(r\Gamma)\cap \Omega_2$, with $r=1.1$ to ensure no instabilities for evaluation of potentials in~\eqref{eq:4.3}. Mesh generation in our experiments were obtained using the open-source  package  GMSH~\cite{GeRe:2009},  that is efficient  for generation of triangular meshes on complex geometries $\Omega_0$. Our choice of a complex  $\Omega_0$ for the third set of  experiments is  inspired by  the complex structure of Baby-Yoda 
({apprentice Grogu} from Disney's Mandalorian TV series),
and we use this geometry as  an illustration of  the flexibility and efficiency of our FEM-BEM approach, and the ability to approximate non-smooth fields from complex heterogeneous media. The  total field  induced by the Baby-Yoda based  piecewise continuous refractive index illustrates the complexity of the  wave propagation
model~\eqref{eq:theproblem}.

We simulated numerical experiments using the quadratic ($\mathbb{P}_2$) and  cubic ($\mathbb{P}_3$)  spline elements  for FEM,
and for several different values for the BEM parameter $N$.
Thus the  number of BEM degrees of freedom (DoF) for the $\Gamma$ boundary unknown function is $2N$ and, throughout this section, we denote $L$ as the FEM DoF in $\Omega_2$ and $M$ as the number of Dirichlet constrained nodes on $\Sigma$.  
Clearly $M << L$ (since $M \approx \sqrt{L}$),  and because $\Gamma$ is smooth boundary,  the spectral accuracy of the  Nystr\"om BEM implies that $2N << M$.  
Our FEM-BEM framework involving only relatively small algebraic linear systems~\eqref{eq:NhBEMFEM:0}, for the ($2N+M$)  boundary unknowns,  were solved 
iteratively using GMRES with a very low tolerance of $10^{-9}$,  to ensure that the reported errors corresponds  to  the method itself and not from an approximation of the algebraic system {solutions}. 

 For a desired level of accuracy, choices of the discretization parameters crucially depend on the  wavelength  of the problem, which in turn depends 
on the variable coefficient $k^2n^2$ of the model~\eqref{eq:theproblem} and the size of computational regions for the main unknowns (that determine the DoF in algebraic systems). For the heterogeneous and unbounded region model 
problem~\eqref{eq:theproblem},  our efficient FEM-BEM framework reduces the computational regions to just  bounded  regions $\overline{\Omega_2}$ and $\Sigma$. We recall that 
 if $\Omega_{2, {\rm diam}}$ denotes  the diameter  of the  domain $\Omega_2$ and $n_{\max}$ is the maximum value of the refractive index $n$, for our FEM-BEM framework, the FEM  interior problem wavelength   is $\left (k\times n_{\max} \times \Omega_{2, {\rm diam}}\right )/(2 \pi)$; and the BEM  exterior problem wavelength is $\left (k\times \Gamma_{{\rm diam}} \right )/(2 \pi)$, where $\Gamma_{{\rm diam}}$ is the length of the artificial boundary curve $\Gamma$.

\subsection{Experiment \#1 (Smooth refractive index and smooth solution)}
\begin{figure}[!ht]
\[
\includegraphics[width=0.6\textwidth]{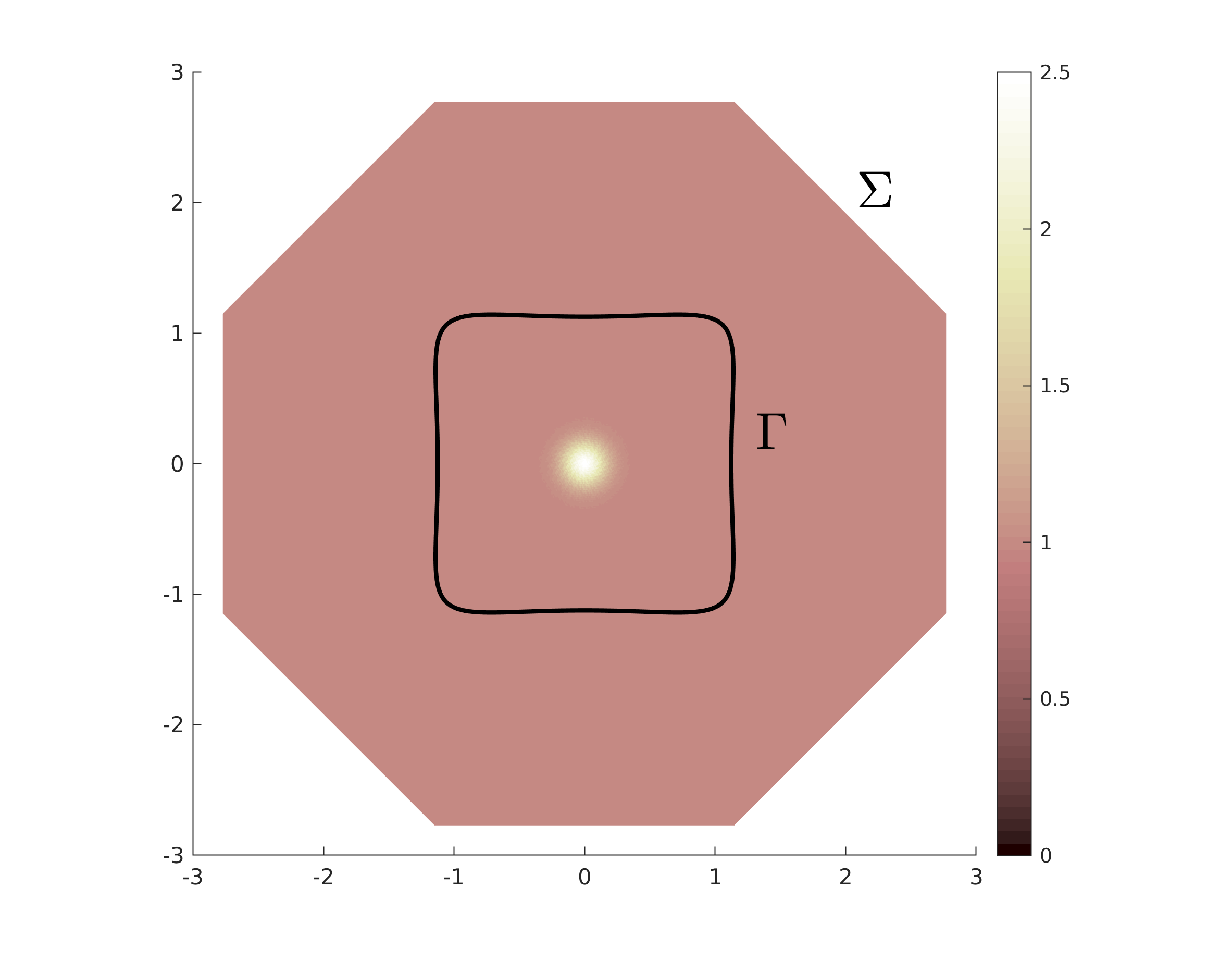}
\]

\vspace{-0.5in}
\caption{\label{fig:exp:01} {\em Experiment \#1  smooth radial refractive index function (with maximum at origin)  and  a decomposition framework comprising curves $\Gamma$ and $\Sigma$.  The octagonal shaped polygonal domain $\Omega_2$, with  boundary  $\Sigma$,  is centered at origin and has circumradius 3.  
$\Gamma$ is a smooth  rounded-squared curve. The framework illustrate flexibility of  artificial  curve  choices $\Gamma$ and $\Sigma$.}}
\end{figure}


In the first set of experiments we choose the example refractive index induced by  a Gaussian smoothing window function in the radial variable $r = |\x|,~ \x = (x,y) \in \mathbb{R}^2$:
\[
 n^2(x,y)\equiv n^ 2( r) = \begin{cases}
 1+1.5\exp(-40 r^2),& r\in[0,1],\\
 1,& r>1.
               \end{cases}
\]
 Since  the value above function at $r = 1$  is less than $1 +\epsilon$ (where $\epsilon$ is the machine-epsilon), the refractive index can be considered as 
numerically smooth.
This example function and our numerical experiments FEM-BEM decomposition framework with artificial curves are demonstrated in Figure~\ref{fig:exp:01}. 
For the first set of experiments to demonstrate accuracy of the FEM-BEM approximations, the incident wave is a plane wave with wavenumber $k = 5$ and direction $\ddh = (1,0)$.  
The rounded-square smooth curve $\Gamma$ in  Figure~\ref{fig:exp:01} is obtained using the \changeV{$2\pi$-periodic} parametrization 
\begin{equation}\label{eq:Gamma:exp:01}
{\bf x}(t) = \frac{4}{5\sqrt{2}}( (1+\cos^2 t)\cos t+(1+\sin^2 t)\sin t , -(1+\cos^2 t)\cos t+(1+\sin^2 t)\sin t),\quad t\in[0,2\pi].
\end{equation}

The initial (coarsest) grid comprise  4,512 triangles   for both $\mathbb{P}_2$ and $\mathbb{P}_3$ elements  and we used  up to five uniform refinements.
For $N$, the discrete parameter for the BEM part, the initial value $N_0= 20$ was refined five times (by doubling). 
Accuracy of the the first set of  simulated  FEM and BEM solutions in the $H^1$-norm are displayed in Table~\ref{tab:exp:01}.
According to Theorem~\ref{theo:5.10}, for the Experiment \#1  choices, the  estimated
rate of convergence of the combined FEM-BEM solution in the $H^1$-error is dominated by the optimal rate $h^d$ in the  $\mathbb{P}_d$ element space
for $d = 2,3$. 
For fixed and sufficiently high $N (= 160)$ case, this estimated convergence can be observed from the last row in Table~\ref{tab:exp:01}, as the FEM mesh size $h$ is reduced by two with mesh refinement 
(and corresponding $L$ increase)  the FEM solution errors decrease  approximately by $(1/2)^d$-times, for $d = 2, 3$. The spectral accuracy of
the BEM solution can be observed from the last column in Table~\ref{tab:exp:01} with high accuracy achieved using relative  small BEM DoF. 
For the  first set of experiments, the GMRES iterations convergence was attained using a small number ($m$)  of iterations  with $16 \leq m \leq 20$.  


 \begin{table}[!ht]
 \begin{center} \tt  \setlength{\tabcolsep}{4pt}\footnotesize
 \begin{tabular}{l|c|c|c|c|c|c|c|c|c|c}
    \changeSecondReferee{\diagbox[width=\dimexpr \textwidth/16\relax, height=.8cm]{$N$}{$L$}}
   &          
\multicolumn{2}{c|}{9,185} &  \multicolumn{2}{c|}{ 36,417   } &  \multicolumn{2}{c|}{  145,025} & \multicolumn{2}{c|}{    578,817  } 
              &  \multicolumn{2}{c }{2,312,705}\\ 
 \hline 
              &       {\rm FEM}   &  {\rm BEM} &   {\rm FEM}  &  {\rm BEM}  &   {\rm FEM}  &  {\rm BEM} &  {\rm FEM}   &  {\rm BEM} & {\rm FEM}  &  {\rm BEM} \\
   020        &        2.4e-01   &  6.4e-02   &    6.1e-02  &   5.5e-02   &  1.5e-02    &    5.4e-02  &   3.7e-03  &   5.4e-02  &  2.2e-05  &  5.4e-02 \\
   040        &        2.4e-01   &  3.7e-02   &    6.1e-02  &   5.7e-03   &  1.5e-02    &    7.0e-04  &   3.7e-03  &   3.6e-04  &  2.1e-05  &  3.5e-04 \\
   080        &        2.4e-01   &  2.8e-02   &    6.1e-02  &   4.3e-03   &  1.5e-02    &    4.6e-04  &   3.7e-03  &   6.7e-05  &  2.1e-05  &  4.8e-06 \\
   160        &        2.4e-01   &  2.7e-02   &    6.1e-02  &   2.7e-03   &  1.5e-02    &    2.4e-04  &   3.7e-03  &   3.2e-05  &  2.1e-05  &  2.5e-06 \\
  \end{tabular}
\end{center}

   \begin{center} \tt  \setlength{\tabcolsep}{4pt}\footnotesize
 \begin{tabular}{r|c|c|c|c|c|c|c|c|c|c }
    \changeSecondReferee{\diagbox[width=\dimexpr \textwidth/16\relax, height=.8cm]{$N$}{$L$}}
 &  \multicolumn{2}{c|}{20,545}&
\multicolumn{2}{c|}{81,697} &  \multicolumn{2}{c|}{  325,825} & \multicolumn{2}{c|}{    1,301,377  } 
              &  \multicolumn{2}{c }{5,201,665}\\ 
\hline 
              &      {\rm FEM}  &  {\rm BEM} & {\rm FEM}   &  {\rm BEM} &   {\rm FEM}  &  {\rm BEM}   &     {\rm FEM} &  {\rm BEM} & {\rm FEM}     &  {\rm BEM} \\ 
   020        &      1.2e-02   &  5.6e-02   &  1.5e-03   & 5.6e-02    &   2.6e-04   &    5.6e-02   &     2.1e-04  &   5.6e-02  &     2.0e-04  &   5.06-02\\
   040        &      1.2e-02   &  1.3e-03   &  1.4e-03   & 3.8e-04    &   1.7e-04   &    3.4e-04   &     2.1e-05  &   3.4e-04  &     1.1e-06  &   3.4e-04\\
   080        &      1.2e-02   &  1.1e-03   &  1.4e-03   & 1.3e-04    &   1.7e-04   &    6.9e-06   &     2.1e-05  &   9.2e-08  &     1.1e-06  &   2.7e-08\\
   160        &      1.3e-02   &  1.2e-03   &  1.4e-03   & 9.5e-05    &   1.7e-04   &    3.5e-06   &     2.1e-05  &   7.0e-08  &     1.1e-06  &   1.9e-08 
  \end{tabular}                     
\end{center}                                                                                      

\caption{\label{tab:exp:01}  {\em Experiment \#1 results ($k = 5$). Estimated $H^1-$ error for  total wave in  $\Omega_2$ (FEM-part) and for scattered wave  away from $\Gamma$ (BEM-part) for 
$\mathbb{P}_2$  (top), $\mathbb{P}_3$ (bottom) elements.} }
     \end{table}

 For the case $k =5$, the Experiment \#1 interior problem  wavelength 
is approximately $6.2$. We also performed  higher frequency simulations, with $k = 10, 20$, we observed similar FEM-BEM convergence rates 
similar to that in Table~\ref{tab:exp:01} for the smooth solution Experiment \#1. For the non-smooth solution cases, as we shall demonstrate 
below for $k = 5, 10, 20$ cases, sufficiently fine FEM mesh and also  high-order finite elements are required to obtained good accuracy. 
      
\subsection{Experiment \#2 (Janus-type configuration non-smooth solution)}
\begin{figure}[!ht]
\[
 \includegraphics[width=0.6\textwidth]{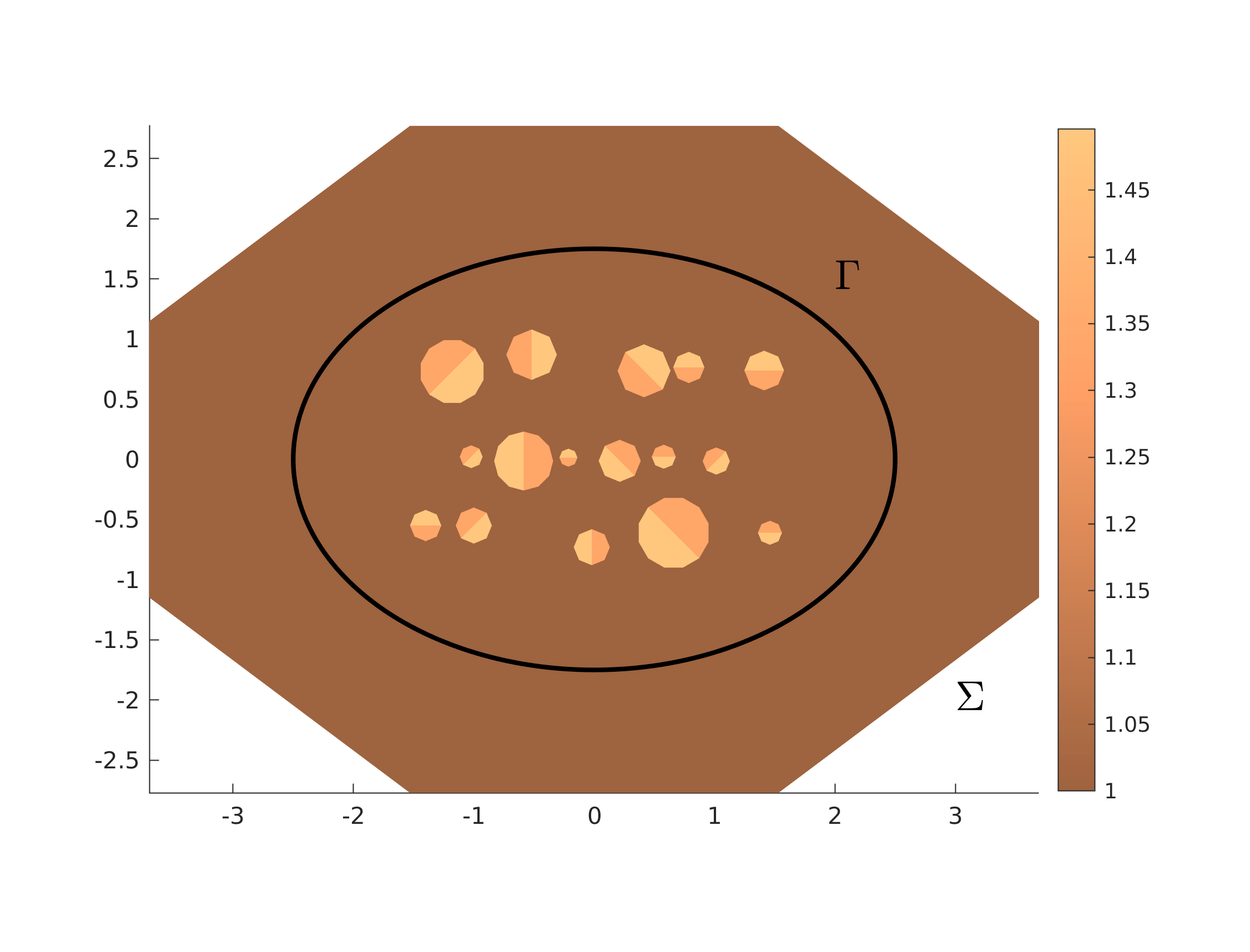} 
 \]
 
 \vspace{-0.5in}
\caption{\label{fig:janus:01} {\em Experiment \#2 setup comprises a Janus-type configuration with $16$ Janus particles (each having a local piecewise-constant refractive index) of distinct sizes and non-smooth shapes.  The FEM  octagonal shaped polygonal domain $\Omega_2$, with  artificial boundary  $\Sigma$,  is centered at origin and has circumradius 3.5.  
The BEM smooth artificial  boundary  $\Gamma$ is an elliptical curve and the ellipse circumscribes the multi-particle  heterogeneous Janus configuration.} }
\end{figure}

For this set of experiments, a disconnected heterogeneous medium illustrated in Figure~\ref{fig:janus:01}, is induced by a collection of non-smooth (polygonal)  Janus particles of different sizes,
and each particle is designed using two distinct materials/liquids~\cite{jan_appl_4}. As described in Section~\ref{sec:intro}, this is a model problem for the so-called Janus configuration which has received much attention in recent years. 
For the  Figure~\ref{fig:janus:01} Janus configuration based second set of numerical experiments, we used the following  non-smooth (piecewise-continuous) refractive index  
function, defined in the  full unbounded wave propagation medium, for $(x,y) \in \mathbb{R}^2$,   as 
\begin{equation}\label{eq:janus-n}
 n(x,y) = \begin{cases}
           1.333  & \text{$(x,y)\in$ first halves of Janus particles},\\
           1.496  & \text{$(x,y) \in$ second halves of Janus particles},\\
           1      & \text{Outside the collection of  Janus particles}.
          \end{cases}
\end{equation}
In~\eqref{eq:janus-n}, we recall that  the numbers $1.333$ and $1.496$ correspond to, respectively, the refractive index of water and toluene  at 20C; 
such two chemicals mix  have been used in the literature to build  Janus configurations~\cite{jan_appl_4}. Janus configuration 
based numerical  investigations~\cite{jan_appl_1, jan_appl_2, jan_appl_3, jan_appl_4,Jan_sph_20_book, Jan_sph_20_pap} have been mainly restricted to simple shapes, and results in this section further demonstrate that such (shape, size, and 
number of particles) restrictions can  be removed using efficient algorithms that can accommodate complex heterogeneous  structures with efficient decomposition framework.  The simulated Experiment \#2 configuration with a decomposition framework is depicted in Figure~\ref{fig:janus:01}.

For these  experiments we chose medium and high frequencies with $k= 5, 10, 20$ (corresponding to, respectively, approximate interior problem  wavelengths $8.3, 16.6, 33.2$)
with the coarse mesh for  FEM comprising 1401 triangles. For a fixed incident plane wave case,  we took the direction to be $\ddh = (1,0)$. For simulating  the OA-DSCS of the Janus configuration, we chose  one thousand equally-spaced incident directions $\ddh(\phi)$ surrounding the configuration that corresponds to one
thousand of the configuration orientations, with equally spaced  orientation angles  $\phi \in [0, 2\pi)$, starting from  the initial Janus configuration in Figure~\ref{fig:janus:01} (with $\phi = 0$).  

The choice of equally spaced direction angles facilitates high-order approximations to the  integral in~\eqref{eq:OA-DSCS} to compute the OA-DSCS using the rectangle quadrature.
We recall Remark~\ref{rem:multi} to highlight that  our discrete FEM-BEM algorithm is efficient 
for computing solutions with a large number of incident waves. The first set of experiments for the Janus configuration is performed with $\phi = 0$ to ensure that high-order approximations of the 
DSCS integrand in~\eqref{eq:OA-DSCS} can be computed using our FEM-BEM algorithm. We recall
that computation of the  far-field in~\eqref{eq:OA-DSCS}  is a smooth post-processing of the field density on $\Gamma$, and hence we expect  that the accuracy of  far-field approximations should be better, after ensuring convergence and good accuracy of the Janus configuration based FEM-BEM solution.

Our simulation results (for $k = 5, 10, 20$), demonstrating the  FEM-BEM solution accuracy for the ($\phi = 0$) Janus configuration in Figure~\ref{fig:janus:01},
obtained using the  $\mathbb{P}_2$ elements are in Table~\ref{tab:exp2:P2},  and the counterpart $\mathbb{P}_3$ elements results are in Table~\ref{tab:exp2:P3}. 
Because of the restricted regularity of the exact scattered field induced by only  piecewise-continuity of the refractive index in~\eqref{eq:janus-n}, it can be observed 
from Tables~\ref{tab:exp2:P2}-\ref{tab:exp2:P3}, the 
limited (second-order) rate of convergence of the FEM solution does not improve by using $\mathbb{P}_3$ elements compared to that for $\mathbb{P}_2$ elements. However,
because of the restricted regularity, it is important to use high-order $\mathbb{P}_3$ elements especially for the higher frequency ($k >10$) cases. This is because, for example 
with $k=20$ (despite using over one million FEM DoF $L$), the $\mathbb{P}_2$ elements based FEM solutions 
with over $100\%$ errors observed in the last row of Table~\ref{tab:exp2:P2} are not acceptable, and for the same situation the 
$\mathbb{P}_3$ elements provide much better accurate solutions.

In Figure~\ref{fig:janus-fields} we visualize the total field
in $\Omega_2$ and the scattered field in the overlapped region $(\mathbb{R}^2 \setminus \overline{\Omega}_1)\cap \Omega_2$,  and in
Figure~\ref{fig:janus-fields-error}, we demonstrate the constraint of (numerically) matching the FEM and BEM solutions  in 
the overlapped region by plotting the error in the region in log-scale. 
These sets of experiments illustrate the computational difficulties 
associated with  Janus configurations based wave models, especially taking into account that linear  $\mathbb{P}_1$  elements are standard for  pure FEM algorithms  (that do not  satisfy the SRC)~\cite{Ihlenburg:1998}.  

Further advantages of our FEM-BEM algorithm to compute high-order approximations to  smooth far-fields, and hence
the DSCS (with fixed incident direction angle $\phi = 0$),  are  demonstrated in Table~\ref{tab:DSCS}.  Recall that the far-field is a smooth function defined on $\sphere$, and the
uniform norm errors in the far-field  were approximated by evaluating the DSCS at $1440$ equally distributed points in  $\sphere$.
Because of the smooth post-processing of the FEM-BEM solutions to
compute far-fields, similar to the  $k=5$  results in Table~\ref{tab:DSCS}, we observed higher accuracy for the $k = 10, 20$ cases compared to the FEM-BEM solutions accuracy.  We conclude the Experiment \#2 results in Figure~\ref{fig:YAF} showing the  OA-DSCS  as a  function of the observed scattering  DSCS angles (in degrees) for $k = 5, 10, 20$.

\begin{table}
 \begin{center} \tt  \setlength{\tabcolsep}{4pt}\footnotesize
 \begin{tabular}{r|c|c|c|c|c|c|c|c|c|c }
   
    \changeSecondReferee{\diagbox[width=\dimexpr \textwidth/16\relax, height=.8cm]{$N$}{$L$}}                &  \multicolumn{2}{c|}{5,531}&
\multicolumn{2}{c|}{21,981 } &  \multicolumn{2}{c|}{  87,641 } & \multicolumn{2}{c|}{   350,001  } & \multicolumn{2}{c }{   1,398,881}\\ 
\hline 
              &      {\rm FEM}  &  {\rm BEM} & {\rm FEM}   &  {\rm BEM} &   {\rm FEM}  &  {\rm BEM}   &     {\rm FEM} &  {\rm BEM} & {\rm FEM}     &  {\rm BEM} \\ 
   020        &      1.6e+00   &  2.0e+00   &  3.2e-01   & 2.7e-01    &   7.8e-02   &    2.3e-01   &     1.9e-02  &   2.3e-01  &     5.2e-03  &   2.3e-01\\
   040        &      1.6e+00   &  1.9e+00   &  3.2e-01   & 1.3e-01    &   7.8e-02   &    1.3e-02   &     1.9e-02  &   1.2e-03  &     4.7e-03  &   6.1e-04\\
   080        &      1.6e+00   &  1.8e+00   &  3.1e-01   & 1.2e-01    &   7.7e-02   &    9.6e-03   &     1.9e-02  &   8.7e-04  &     4.7e-03  &   8.4e-05\\
   160        &      1.6e+00   &  1.8e+00   &  3.1e-01   & 1.1e-01    &   7.7e-02   &    7.6e-03   &     1.9e-02  &   7.9e-04  &     4.7e-03  &   5.4e-05 
  \end{tabular}                     
\end{center}         
 \begin{center} \tt  \setlength{\tabcolsep}{4pt}\footnotesize
 \begin{tabular}{r|c|c|c|c|c|c|c|c|c|c }
   
    \changeSecondReferee{\diagbox[width=\dimexpr \textwidth/16\relax, height=.8cm]{$N$}{$L$}}                &  \multicolumn{2}{c|}{5,531}&
\multicolumn{2}{c|}{21,981 } &  \multicolumn{2}{c|}{  87,641 } & \multicolumn{2}{c|}{   350,001  } & \multicolumn{2}{c }{   1,398,881}\\ 
\hline 
              &      {\rm FEM}  &  {\rm BEM} & {\rm FEM}   &  {\rm BEM} &   {\rm FEM}  &  {\rm BEM}   &     {\rm FEM} &  {\rm BEM} & {\rm FEM}     &  {\rm BEM} \\ 
   020        &      7.9e+01   &  7.8e+01   & 4.9e+01    &  4.4e+01   &    4.3e+01  &   4.0e+01    &     4.0e+01  &     4.0e+01&     1.6e+01  &   4.0e+01  \\
   040        &      7.2e+01   &  5.5e+01   & 1.7e+01    &  1.7e+01   &    5.3e+00  &   5.7e+00    &     5.6e-01  &     5.6e-01&     5.4e-01  &   3.8e-02  \\
   080        &      7.3e+01   &  5.5e+01   & 1.6e+01    &  1.7e+01   &    5.1e+00  &   5.4e+00    &     4.8e-01  &     4.8e-01&     4.8e-01  &   3.2e-02  \\
   160        &      7.3e+01   &  5.6e+01   & 1.6e+01    &  1.7e+01   &    4.8e+00  &   5.1e+00    &     4.5e-01  &     4.5e-01&     4.5e-01  &   2.8e-02   
  \end{tabular}      
\end{center} 
   \begin{center} \tt  \setlength{\tabcolsep}{4pt}\footnotesize
 \begin{tabular}{r|c|c|c|c|c|c|c|c|c|c }
   
    \changeSecondReferee{\diagbox[width=\dimexpr \textwidth/16\relax, height=.8cm]{$N$}{$L$}}                &  \multicolumn{2}{c|}{5,531}&
\multicolumn{2}{c|}{21,981 } &  \multicolumn{2}{c|}{  87,641 } & \multicolumn{2}{c|}{   350,001  } & \multicolumn{2}{c }{   1,398,881}\\ 
\hline 
              &      {\rm FEM}  &  {\rm BEM} & {\rm FEM}   &  {\rm BEM} &   {\rm FEM}    &  {\rm BEM}   &     {\rm FEM}  &  {\rm BEM} & {\rm FEM}     &  {\rm BEM} \\ 
   020        &      6.0e+02   &  2.3e+02   &  1.5e+03   &  5.75e+02    &   1.1e+03   &    4.6e+02   &    1.6e+03  &   6.3e+02   &   1.4e+03   &   5.4e+02\\
   040        &      4.5e+02   &  1.5e+02   &  2.8e+02   &  1.35e+02    &   3.6e+02   &    1.7e+02   &    1.3e+03  &   7.2e+02   &   1.6e+03   &   8.5e+02\\
   080        &      3.3e+02   &  1.2e+02   &  2.3e+02   &  1.14e+02    &   4.1e+01   &    2.2e+01   &    1.2e+01  &   4.7e+00   &   1.2e+00   &   4.6e-01\\
   160        &      3.2e+02   &  1.2e+02   &  2.4e+02   &  1.17e+02    &   4.0e+01   &    2.2e+01   &    1.2e+01  &   4.5e+00   &   1.1e+00   &   4.2e-01 
  \end{tabular}     
\end{center} 
\caption{\label{tab:exp2:P2} {\em Experiment \#2 results using $\mathbb{P}_2$   elements with $k =5$ (top), $k=10$ (mid), and $k= 20$ \changeSecondReferee{\changeSecondReferee{(bottom)}}. Estimated $H^1-$ error for the total wave in $\Omega_2$  (FEM-part) and for the scattered wave away from $\Gamma$ (BEM-part).}}

\end{table}

 \begin{table}
 \begin{center} \tt  \setlength{\tabcolsep}{4pt}\footnotesize
 \begin{tabular}{r|c|c|c|c|c|c|c|c|c|c }
   
    \changeSecondReferee{\diagbox[width=\dimexpr \textwidth/16\relax, height=.8cm]{$N$}{$L$}}                &  \multicolumn{2}{c|}{12,391}&
\multicolumn{2}{c|}{49,351} &  \multicolumn{2}{c|}{ 196,981} & \multicolumn{2}{c|}{   787,081  } & \multicolumn{2}{c }{   3,146,641}\\ 
\hline 
              &      {\rm FE}  &  {\rm BEM} & {\rm FE}   &  {\rm BEM} &   {\rm FE}  &  {\rm BEM}   &     {\rm FE} &  {\rm BEM} & {\rm FE}     &  {\rm BEM} \\ 
   020        &      1.4e-01    &  3.3e-01   &  1.7e-02   & 3.1e-01    &   3.1e-03  &   3.1e-01    &     2.1e-03  &   3.1e-01  &     2.9e-03  &   3.1e-01\\
   040        &      1.3e-01    &  9.4e-02   &  1.7e-02   & 5.2e-03    &   2.3e-03  &   8.2e-04    &     3.9e-04  &   8.0e-04  &     7.9e-05  &   8.0e-04\\
   080        &      1.3e-01    &  7.8e-02   &  1.7e-02   & 4.2e-03    &   2.3e-03  &   2.1e-04    &     3.9e-04  &   1.1e-05  &     7.9e-05  &   1.5e-06\\
   160        &      1.3e-01    &  7.8e-02   &  1.7e-02   & 3.7e-03    &   2.3e-03  &   2.0e-04    &     3.9e-04  &   1.2e-05  &     7.9e-05  &   1.5e-06 
  \end{tabular}                     
\end{center}

 \begin{center} \tt  \setlength{\tabcolsep}{4pt}\footnotesize
 \begin{tabular}{r|c|c|c|c|c|c|c|c|c|c }
   
    \changeSecondReferee{\diagbox[width=\dimexpr \textwidth/16\relax, height=.8cm]{$N$}{$L$}}                &  \multicolumn{2}{c|}{12,391}&
\multicolumn{2}{c|}{49,351} &  \multicolumn{2}{c|}{ 196,981} & \multicolumn{2}{c|}{   787,081  } & \multicolumn{2}{c }{   3,146,641}\\  
\hline 
              &      {\rm FE}  &  {\rm BEM} & {\rm FE}   &  {\rm BEM} &   {\rm FE}  &  {\rm BEM}   &     {\rm FE} &  {\rm BEM} & {\rm FE}     &  {\rm BEM} \\ 
   020        &      4.9e+01   &  5.7e+01   & 4.3e+01    &  5.2e+01   &   4.3e+01  &   5.2e+01    &      4.3e+01  &     5.2e+01&     4.3e+01  &   5.2e+01  \\
   040        &      8.3e+00   &  1.2e+01   & 9.8e-01    &  1.3e+00   &   5.2e-02  &   5.8e-02    &      7.5e-03  &     9.6e-03&     5.2e-03  &   8.9e-03  \\
   080        &      7.9e+00   &  1.1e+01   & 1.0e+00    &  1.4e+00   &   4.1e-02  &   3.7e-02    &      4.3e-03  &     1.3e-03&     5.8e-04  &   7.8e-05  \\
   160        &      8.0e+00   &  1.1e+01   & 9.6e-01    &  1.3e+00   &   4.0e-02  &   3.5e-02    &      4.2e-03  &     6.5e-04&     5.8e-04  &   5.0e-05   
  \end{tabular}      
\end{center}

   \begin{center} \tt  \setlength{\tabcolsep}{4pt}\footnotesize
 \begin{tabular}{r|c|c|c|c|c|c|c|c|c|c }
   
    \changeSecondReferee{\diagbox[width=\dimexpr \textwidth/16\relax, height=.8cm]{$N$}{$L$}}                &  \multicolumn{2}{c|}{12,391}&
\multicolumn{2}{c|}{49,351} &  \multicolumn{2}{c|}{ 196,981} & \multicolumn{2}{c|}{   787,081  } & \multicolumn{2}{c }{   3,146,641}\\ 
\hline 
              &      {\rm FE}  &  {\rm BEM} & {\rm FE}   &  {\rm BEM} &   {\rm FE}    &  {\rm BEM}   &     {\rm FE}  &  {\rm BEM} & {\rm FE}  &  {\rm BEM} \\ 
   020        &      1.5e+03  & 6.7e+02  &    2.5e+03  &   1.2e+03  &   1.4e+03    &   6.7e+02   &    1.4e+03    &    6.5e+02   &   1.4e+03   &    6.5e+02\\
   040        &      9.6e+02  & 5.3e+02  &    7.5e+02  &   4.1e+02  &   1.5e+03    &   1.0e+03   &    1.6e+03    &    1.1e+03   &   1.6e+03   &    1.1e+03\\
   080        &      6.3e+02  & 3.8e+02  &    6.3e+01  &   4.1e+01  &   1.4e+00    &   6.9e-01   &    7.9e-02    &    2.5e-02   &   1.4e-02   &    6.5e-03\\
   160        &      6.6e+02  & 3.9e+02  &    6.0e+02  &   3.9e+01  &   1.4e+00    &   7.3e-01   &    7.7e-02    &    2.4e-02   &   8.2e-03   &    9.9e-04 
  \end{tabular}    
\end{center}  
\caption{\label{tab:exp2:P3} {\em Experiment \#2 results using $\mathbb{P}_3$   elements with $k =5$ (top), $k=10$ (mid), and $k= 20$ \changeSecondReferee{(bottom)}. Estimated $H^1-$ error for the total wave in $\Omega_2$  (FEM-part) and for the scattered wave away from $\Gamma$ (BEM-part).}}

 \end{table}

 \begin{figure}     
 \[
  \includegraphics[width=0.52\textwidth]{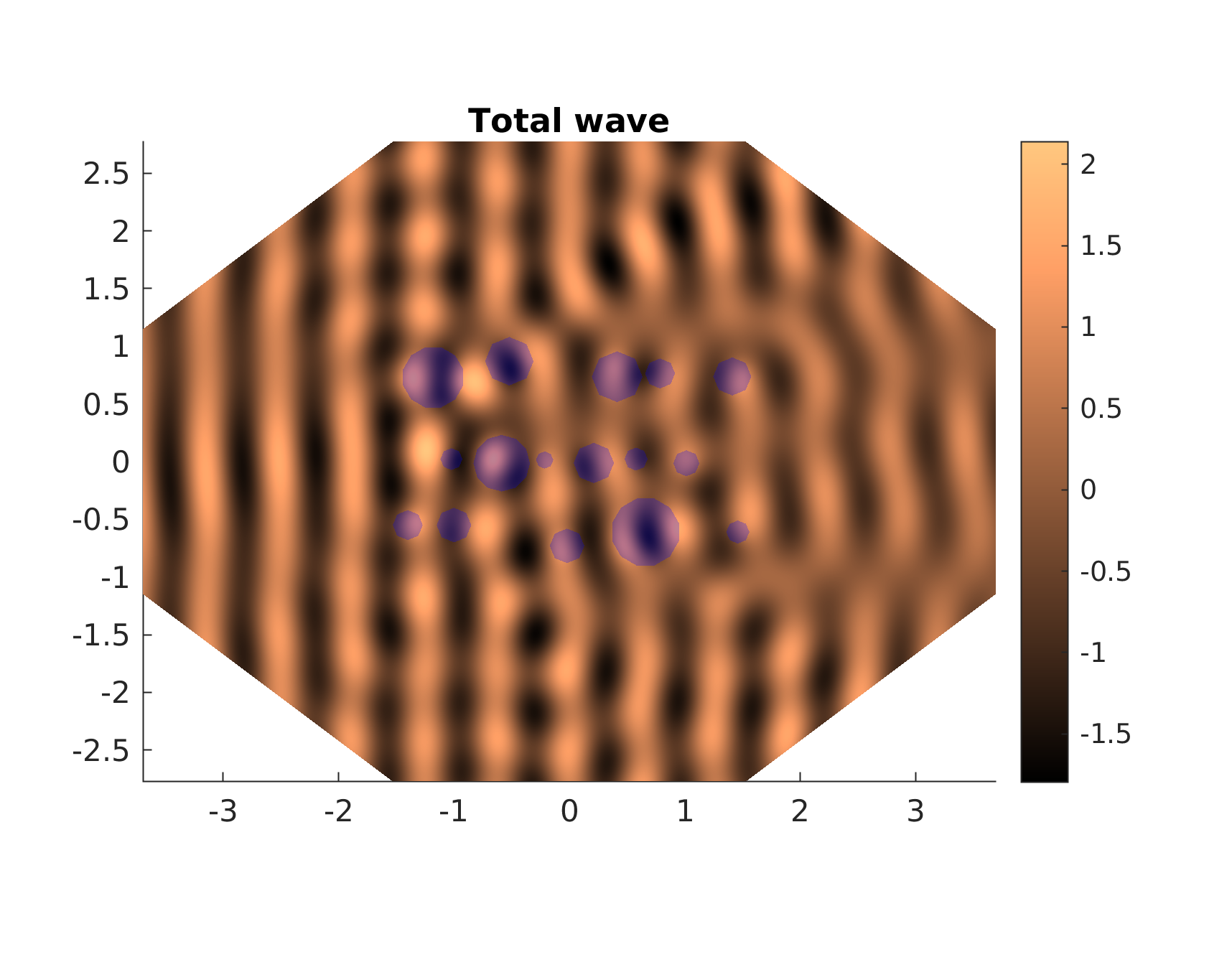}
  \includegraphics[width=0.52\textwidth]{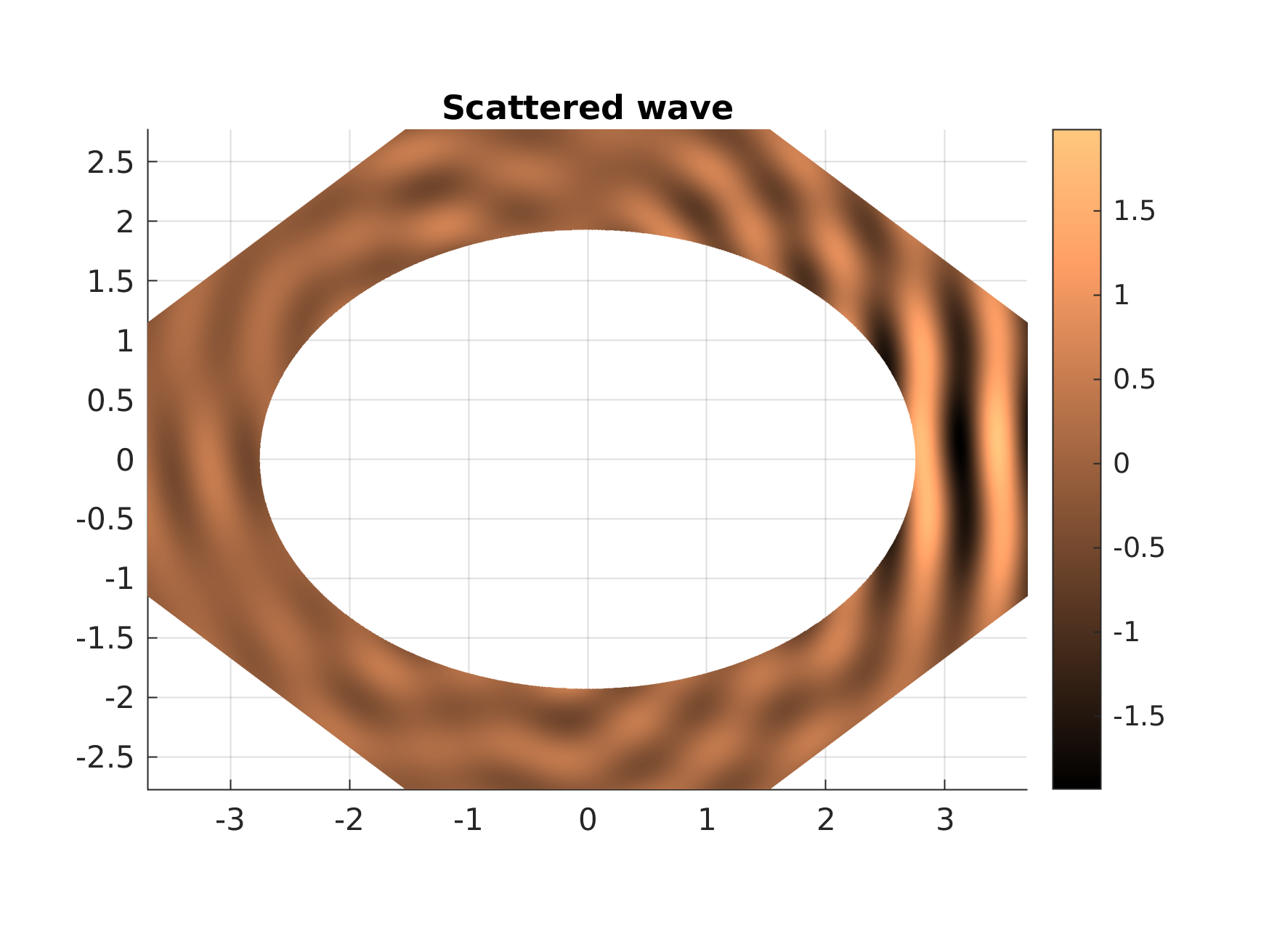}
  \]
\caption{\label{fig:janus-fields}  {\em Experiment \#2 ($k = 20$). Computed total wave (left) and scattered wave (right) using $\mathbb{P}_3$-finite elements with $698, 880$  elements  and  $1,398,881$ nodes, and using a  spectral BEM with $N=80$.}}
  \[
  \includegraphics[width=0.66\textwidth]{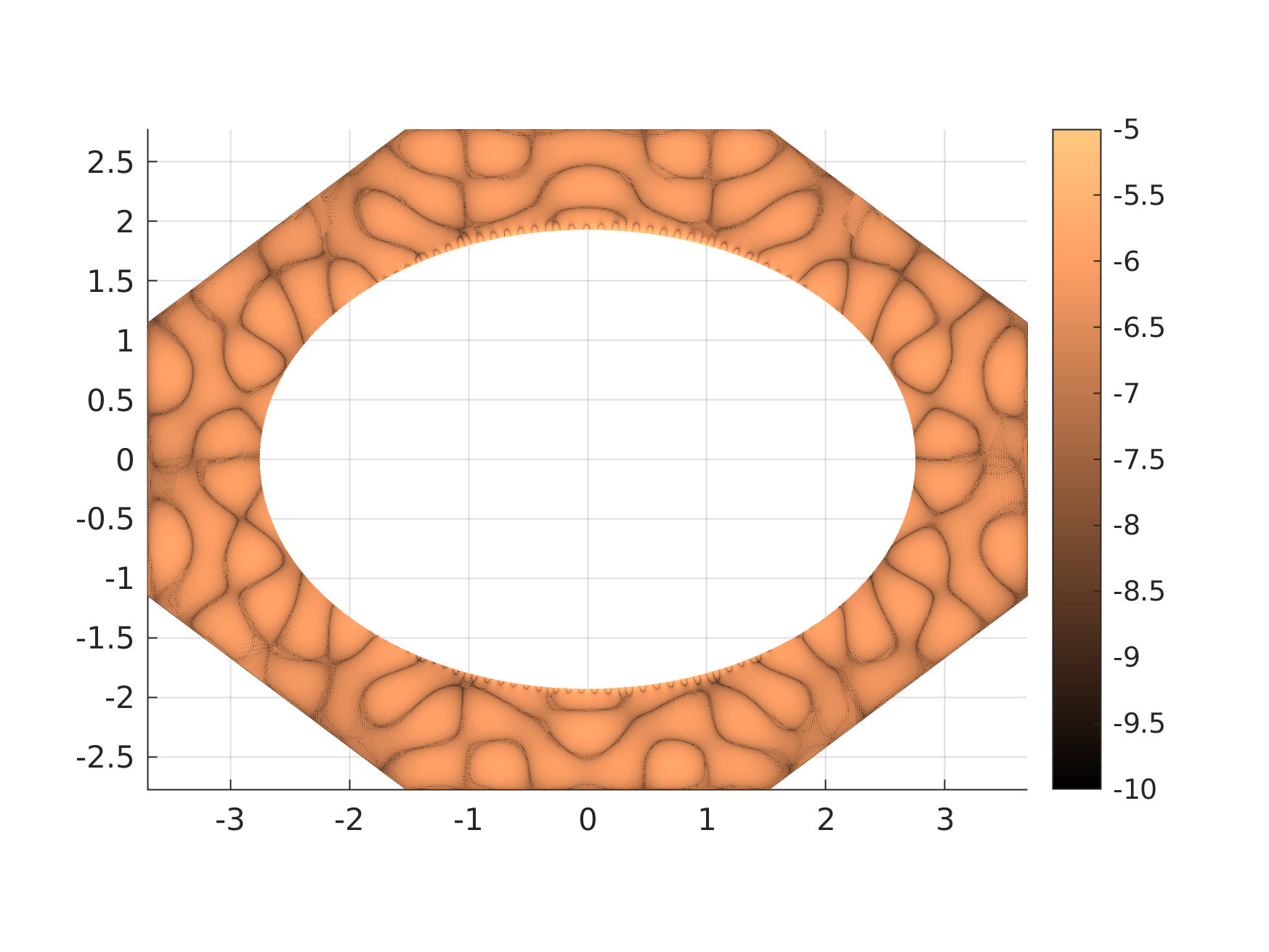}
  \]
\caption{\label{fig:janus-fields-error} {\em Experiment \#2 (with parameter values as in Figure~\ref{fig:janus-fields}. Absolute difference in log-scale between BEM-computed and FEM-computed approximation solutions in the overlapped region $(\mathbb{R}^2 \setminus \overline{\Omega}_1)\cap \Omega_2$.}}
 \end{figure}

\begin{table}
    \begin{center} \tt  \setlength{\tabcolsep}{4pt} \small 
 \begin{tabular}{r|ccccc}
   
    \changeSecondReferee{\diagbox[width=\dimexpr \textwidth/16\relax, height=.8cm]{$N$}{$L$}}                
          &   {5,531}    & {21,981 } &{87,641}  & {350,001} & 1,398,881\\ 
 \hline     
   010    &   5.15e+00   &  4.91e+00 & 4.89e+00 &  4.89e+00 & 4.89e+00\\ 
   020    &   4.37e-01   &  2.77e-02 & 3.14e-03 &  1.65e-03 & 1.33e-03\\ 
   040    &   4.12e-01   &  2.53e-02 & 1.04e-03 &  2.93e-04 & 3.34e-05 \\
  080    &   4.08e-01   &  2.60e-02 & 1.58e-03 &  1.26e-04 & 2.76e-05 \\
   160    &   4.04e-01   &  2.58e-02 & 1.67e-03 &  1.39e-04 & 2.83e-05 \\
\hline
   \end{tabular}
 
 \ \\[2ex]
                    
    \begin{tabular}{r|ccccc}
   $N$/${L}$                
          &   {12,391}  & {49,351}  &{196,981} & {787,081} & {3,146,641}\\ 
 \hline  
  010    &   4.89e+00  &  4.89e+00 & 4.89e+00 &  4.89e+00 & 4.89e+00\\ 
   020    &   5.01e-03  &  1.17e-03 & 1.31e-03 &  1.36e-03 & 1.36e-03\\ 
   040    &   1.03e-02  &  4.89e-04 & 2.15e-05 &  3.87e-07 & 7.79e-08 \\
   080    &   1.23e-02  &  4.47e-04 & 3.18e-05 &  7.21e-07 & 8.25e-08 \\
   160    &   1.30e-02  &  3.45e-04 & 3.88e-05 &  1.07e-06 & 4.94e-08 \\
   \end{tabular}
   \end{center}
\caption{\label{tab:DSCS} {\em Experiment \#2 estimated DSCS  uniform norm errors, for $k = 5$,  with  $\mathbb{P}_2$ (top) 
and $\mathbb{P}_3$ (bottom) elements and $\phi =0$ (i.e., with $\bm{d}=(1,0)$ in the incident plane wave). }}
   \end{table}
 
 \begin{figure}[h]
\[
 \includegraphics[width=0.68\textwidth]{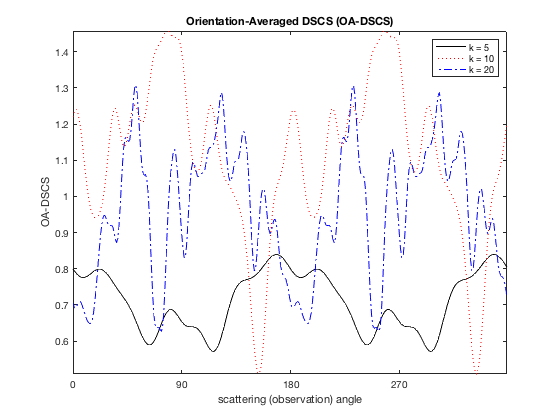} 
 \]
\caption{\label{fig:YAF}Orientation-Averaged DSCS for Experiment \#2}
\end{figure}

\clearpage
 \subsection{Experiment \#3 (A Baby-Yoda shaped  heterogeneous region)}

In this set of experiments,   we demonstrate flexibility of our  FEM-BEM algorithm with complex structured heterogeneous regions. To this end, we consider a  Baby-Yoda like 
domain, depicted in Figure~\ref{fig:babyYoda} on which the refraction index function is defined. The function is taken to be  piecewise-constant at distinct parts of
the domain with values indicated in Figure~\ref{fig:babyYoda}, leading to the  total field solution $u$ with limited regularity.   In Figure~\ref{fig:babyYoda} we also 
show the curve $\Gamma$ taken  for our BEM computations, which is similar to the smooth curve in Experiment \#1 with some rescaling;  and  the choice of $\Sigma$ is such that  the  FEM computational domain  $\Omega_2$  is a rectangle. Similar to the first two sets of experiments, we simulate the model with wavenumbers $k = 5, 10, 20$, and for the configuration in Figure~\ref{fig:babyYoda} 
these \changeV{values}  respectively correspond to $9.2, 18.4,  36.8$ interior wavelengths region $\Omega_2$. 
 
 \begin{figure}[h]
\[
 \includegraphics[width=0.68\textwidth]{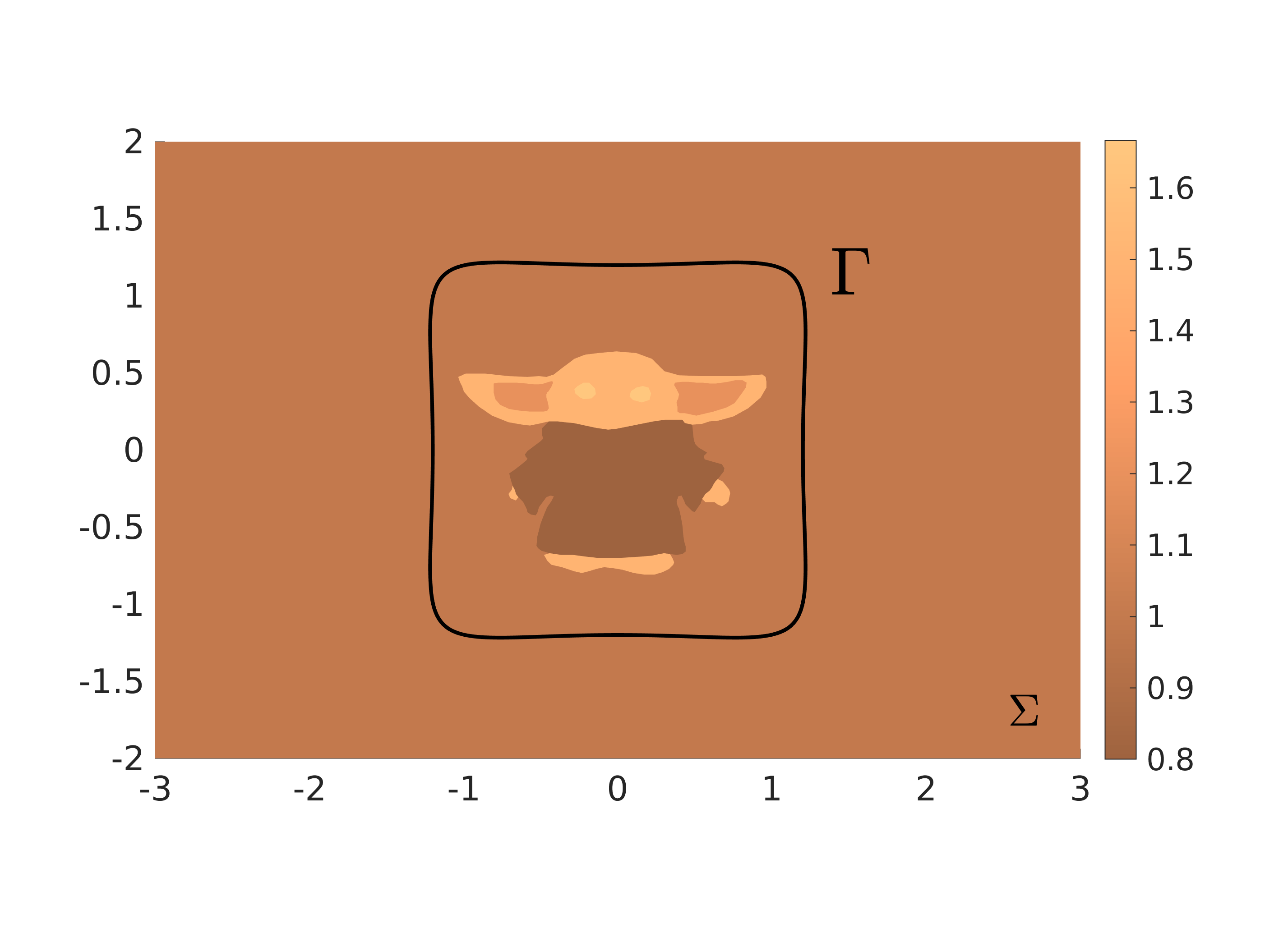} 
 \]
 
 \vspace{-0.5in}
\caption{\label{fig:babyYoda} {\em Experiment \#3 setup with Baby-Yoda heterogeneous region.}}
\end{figure}

Similar to the first two sets of experiments, sample results in Table~\ref{table:exp03:P3b} demonstrate the power  of our overlapped FEM-BEM algorithm even for the 
case of simulation of non-smooth wave fields induced by only piecewise-continuous refractive index in complicated heterogeneous regions. 
Another marked advantage  of our FEM-BEM algorithm  is the low number of iterations required for  convergence of the GMRES to iteratively solve  the resulting 
interface algebraic linear systems~\eqref{eq:NhBEMFEM:0}. In particular, as can be observed in Table~\ref{tab:GMRES-yoda}, for each fixed wavenumber $k$ the number
of required GMRES iterations is independent of the numbers of FEM-BEM DoFs (level of discretizations). In addition,  even as the frequency is doubled the required
iterations grow only mildly (to achieve a fixed error tolerance), and because of the well-conditioning the interface system  ~\eqref{eq:NhBEMFEM:0}, even for
the high-frequency case, less than $100$ GMRES iterations are required without using any preconditioner. We conclude the numerical experiments section with
visualizations in   Figure~\ref{fig:yoda-fields} for (i) the total field
in $\Omega_2$; (ii) the scattered field in the overlapped region $\Omega_{12}=(\mathbb{R}^2 \setminus \overline{\Omega}_1)\cap \Omega_2$;  and (iii) accurate
matching of the FEM and BEM solutions in the overlapped region in Figure~\ref{fig:yoda-fields-error}.

   \begin{table}

 \begin{center} \tt  \setlength{\tabcolsep}{4pt}\small 
  \begin{tabular}{r|c|c|c|c|c|c}
    \changeSecondReferee{\diagbox[width=\dimexpr \textwidth/16\relax, height=.8cm]{$N$}{$L$}} & \multicolumn{2}{c|}{114,094 } &  \multicolumn{2}{c|}{    455,785 } &  \multicolumn{2}{c}{   1,821,961} \\ 
\hline &  \multicolumn{2}{c|}{ } &  \multicolumn{2}{c|}{ }  &  \multicolumn{2}{c}{  }  \\         
              &       {\rm FE}   &  {\rm BEM}  &  {\rm FE}  &  {\rm BEM}  &   {\rm FE}  &  {\rm BEM}  \\
   020        &      3.1e-02     & 3.0e-01   & 2.0e-02    &  3.0e-01    &   1.6e-02   & 3.0e-01  \\  
   040        &      2.7e-02     & 2.2e-03   & 1.4e-02    &  2.0e-03    &   6.2e-03   & 2.0e-03  \\  
   080        &      2.7e-02     & 9.6e-04   & 1.4e-02    &  2.4e-04    &   6.2e-03   & 5.2e-05  \\  
   160        &      2.7e-02     & 9.5e-04   & 1.4e-02    &  2.4e-04    &   6.2e-03   & 5.2e-05       
  \end{tabular}
\end{center}

 \begin{center} \tt  \setlength{\tabcolsep}{4pt}\small 
  \begin{tabular}{r|c|c|c|c|c|c}
    \changeSecondReferee{\diagbox[width=\dimexpr \textwidth/16\relax, height=.8cm]{$N$}{$L$}} & \multicolumn{2}{c|}{114,094 } &  \multicolumn{2}{c|}{    455,785 } &  \multicolumn{2}{c}{   1,821,961} \\ 
\hline &  \multicolumn{2}{c|}{ } &  \multicolumn{2}{c|}{ }  &  \multicolumn{2}{c}{  }  \\         
              &       {\rm FE}   &  {\rm BEM}  &  {\rm FE}  &  {\rm BEM}  &   {\rm FE}  &  {\rm BEM}  \\
   020        &      2.1e+01     & 1.9e+01   & 2.1e+01    &  1.9e+01    &   2.1e+01   & 1.9e+01  \\  
   040        &      8.5e-02     & 2.8e-02   & 2.9e-02    &  5.7e-02    &   1.3e-02   & 5.1e-03  \\  
   080        &      8.5e-02     & 2.7e-02   & 2.9e-02    &  2.1e-03    &   1.3e-02   & 2.1e-04  \\  
   160        &      8.5e-02     & 2.6e-02   & 2.9e-02    &  2.0e-03    &   1.3e-02   & 2.1e-04       
  \end{tabular}
\end{center}
 
 \begin{center} \tt  \setlength{\tabcolsep}{4pt}\small 
  \begin{tabular}{r|c|c|c|c|c|c}
    \changeSecondReferee{\diagbox[width=\dimexpr \textwidth/16\relax, height=.8cm]{$N$}{$L$}} & \multicolumn{2}{c|}{114,094 } &  \multicolumn{2}{c|}{    455,785 } &  \multicolumn{2}{c}{   1,821,961} \\ 
\hline &  \multicolumn{2}{c|}{ } &  \multicolumn{2}{c|}{ }  &  \multicolumn{2}{c}{  }  \\                
              &       {\rm FE}   &  {\rm BEM}  &  {\rm FE}  &  {\rm BEM}  &   {\rm FE}  &  {\rm BEM}  \\
   020        &      1.7e+02     & 1.1e+02   & 1.7e+02    &  1.1e+02    &   1.7e+02   & 1.1e+02  \\  
   040        &      5.1e+00     & 4.3e+00   & 5.0e+01    &  4.4e+00    &   5.0e+01   & 4.4e+00  \\  
   080        &      1.0e+00     & 2.7e-01   & 1.3e-01    &  6.5e-03    &   3.0e-02   & 8.8e-04  \\  
   160        &      1.0e+00     & 2.7e-01   & 1.3e-01    &  6.2e-03    &   3.0e-02   & 8.6e-04       
  \end{tabular}
\end{center}
 \caption{\label{table:exp03:P3b} {\em Experiment \#3 results using $\mathbb{P}_3$   elements with $k =5$ (top), $k=10$ (mid), and $k= 20$ \changeSecondReferee{(bottom)}. Estimated $H^1-$ error for the total wave in $\Omega_2$  (FEM-part) and for the scattered wave away from $\Gamma$ (BEM-part).}} 
 \end{table}

 \begin{table}

   \begin{center} \tt  \setlength{\tabcolsep}{4pt}\footnotesize
 \begin{tabular}{r||c|c|c||c|c|c||c|c|c||c|c|c}
     \changeSecondReferee{\diagbox[width=\dimexpr \textwidth/16\relax, height=.8cm]{$N$}{$L$}} & \multicolumn{3}{c||}{48,137} &  \multicolumn{3}{c||}{    192,153  } &  \multicolumn{3}{c||}{767,825  } 
              &  \multicolumn{3}{c}{3,069,729}\\ 
\hline 
              &   k=5   &  k=10  &  k=20  & k=5 &  k=10 & k=20  &   k=5  &  k=10  &  k=20 &   k=5  &  k=10   &  k=20      \\ 
   020        &   22    &   38   &   40     & 22  &  38   &   40   &  22  &   38  &   040  &  22   &   38    &   40        \\
   040        &   22    &   38   &   67     & 22  &  38   &   67   &  22  &   38  &   067  &  22   &   38    &   67        \\
   080        &   22    &   38   &   67     & 22  &  38   &   67   &  22  &   38  &   067  &  22   &   38    &   67        \\
   160        &   22    &   38   &   67     & 22  &  38   &   67   &  22  &   38  &   067  &  22   &   38    &   67    
  \end{tabular}  
\end{center} 

   \begin{center} \tt  \setlength{\tabcolsep}{4pt}\small
 \begin{tabular}{r||c|c|c||c|c|c||c|c|c}
     \changeSecondReferee{\diagbox[width=\dimexpr \textwidth/16\relax, height=.8cm]{$N$}{$L$}} & \multicolumn{3}{c||}{114,094 } &  \multicolumn{3}{c||}{    455,785 } &  \multicolumn{3}{c}{1,821,961 } \\
\hline 
              &   k=5   &  k=10  &  k=20  & k=5 &  k=10 & k=20  &   k=5  &  k=10  &  k=20        \\ 
   020        &   22    &   38   &   40     & 22  &  38   &   40   &  22  &   38  &   40        \\
   040        &   22    &   38   &   67     & 22  &  38   &   67   &  22  &   38  &   67       \\
   080        &   22    &   38   &   67     & 22  &  38   &   67   &  22  &   38  &   67       \\
   160        &   22    &   38   &   67     & 22  &  38   &   67   &  22  &   38  &   67   
  \end{tabular}  
\end{center}

  \caption{\label{tab:GMRES-yoda} 
\em Experiment \#3. Total number of GMRES iterations for convergence (with tolerance  $10^{-8}$) using  $\mathbb{P}_2$ (top) and $\mathbb{P}_3$ (bottom) elements. For each fixed $k$,   the total number of required GMRES iterations  is independent of $N$ and $L$  and grows mildly with respect to increase in the wavenumber $k$.}

 \end{table}                                                                                                          
                                                            
  \begin{figure}
    \[
  \includegraphics[width=0.52\textwidth]{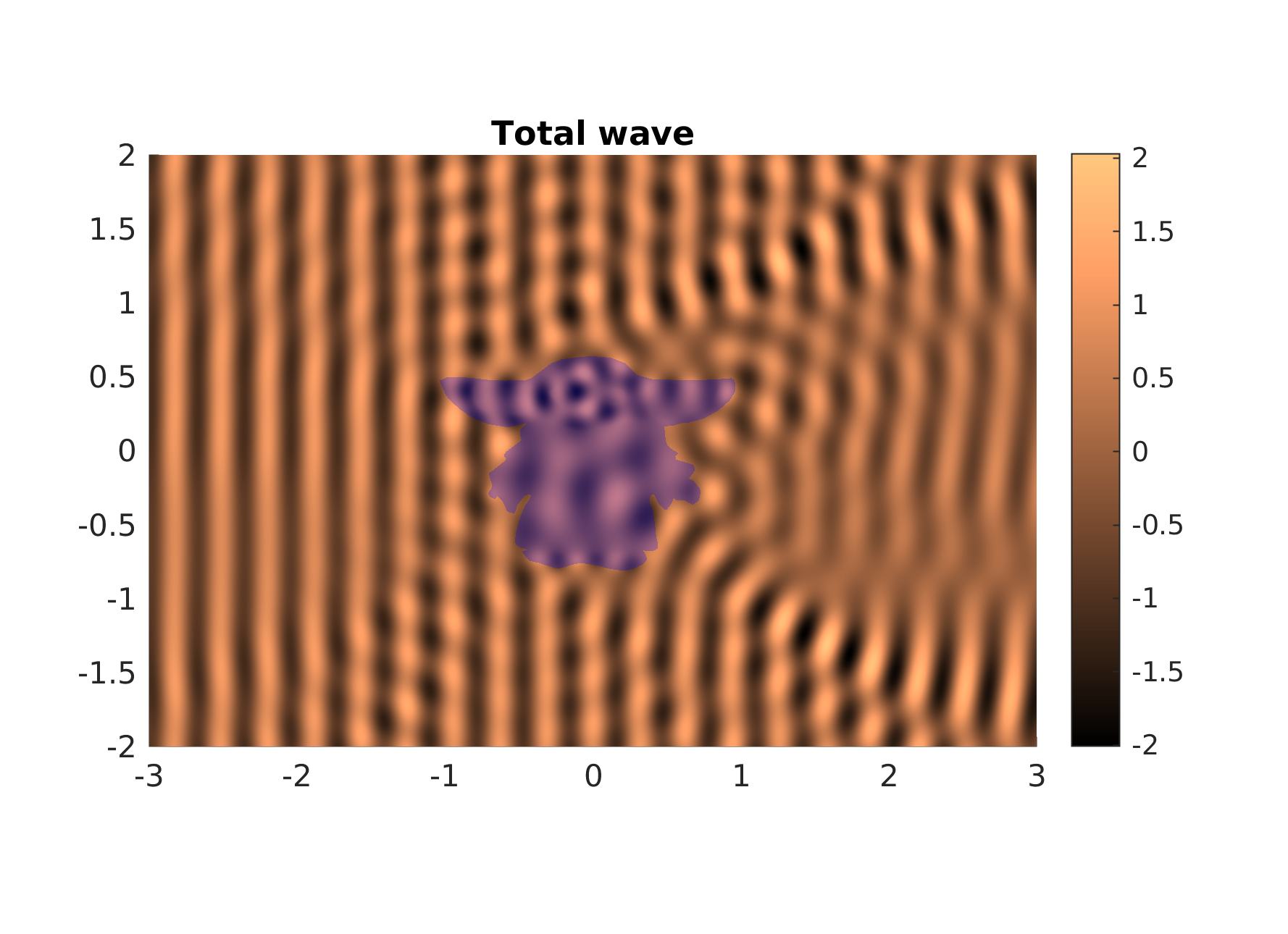}
  \includegraphics[width=0.52\textwidth]{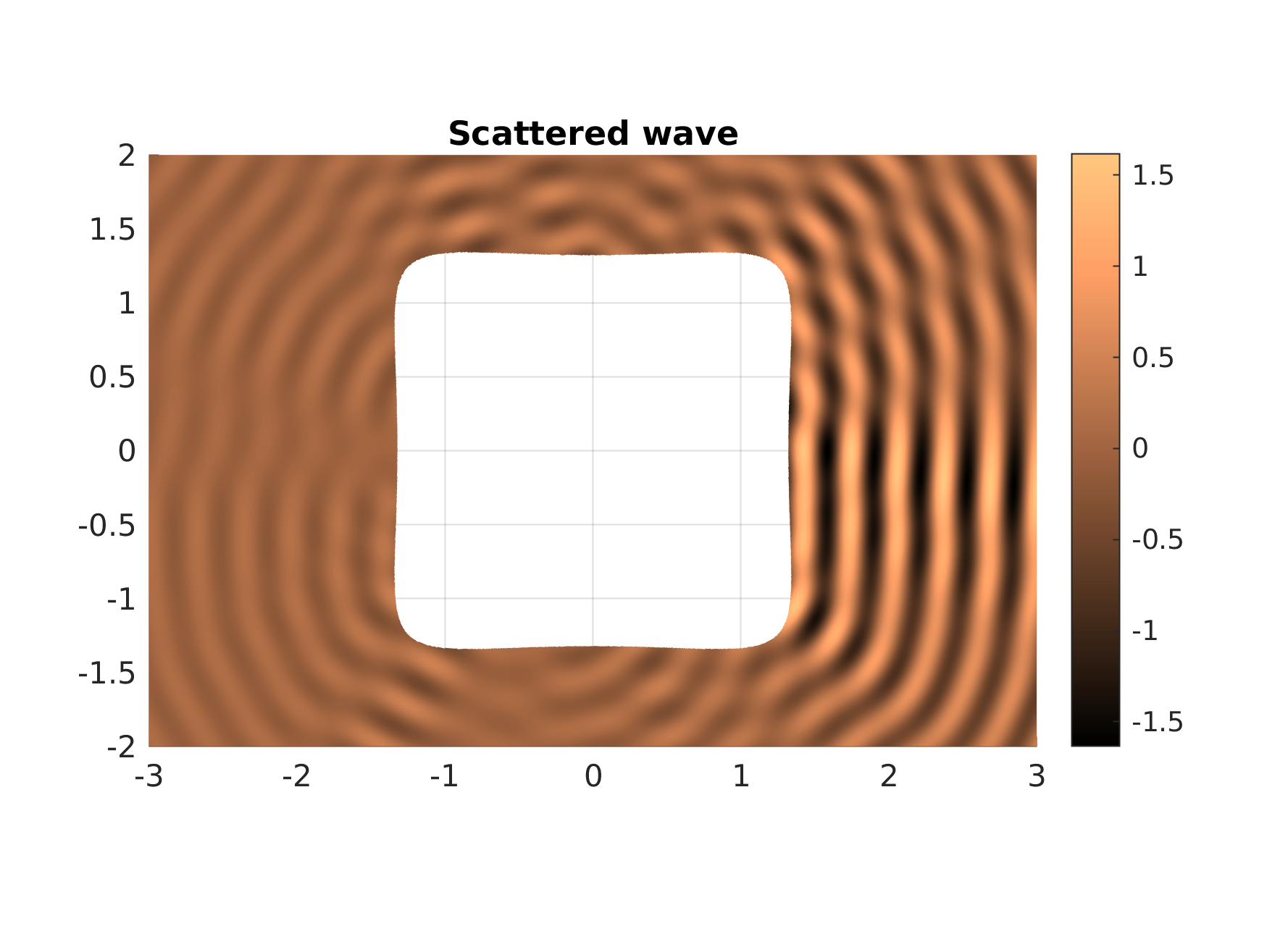}
  \]
\caption{\label{fig:yoda-fields}{\em Experiment \#3 ($k = 20$). Computed total wave (left) and scattered wave (right) using $\mathbb{P}_3$-finite elements with 
$478,464$   elements  and  $1,821,961$ nodes, and using a  spectral BEM with   $N=80$.}}
  \[
  \includegraphics[width=0.66\textwidth]{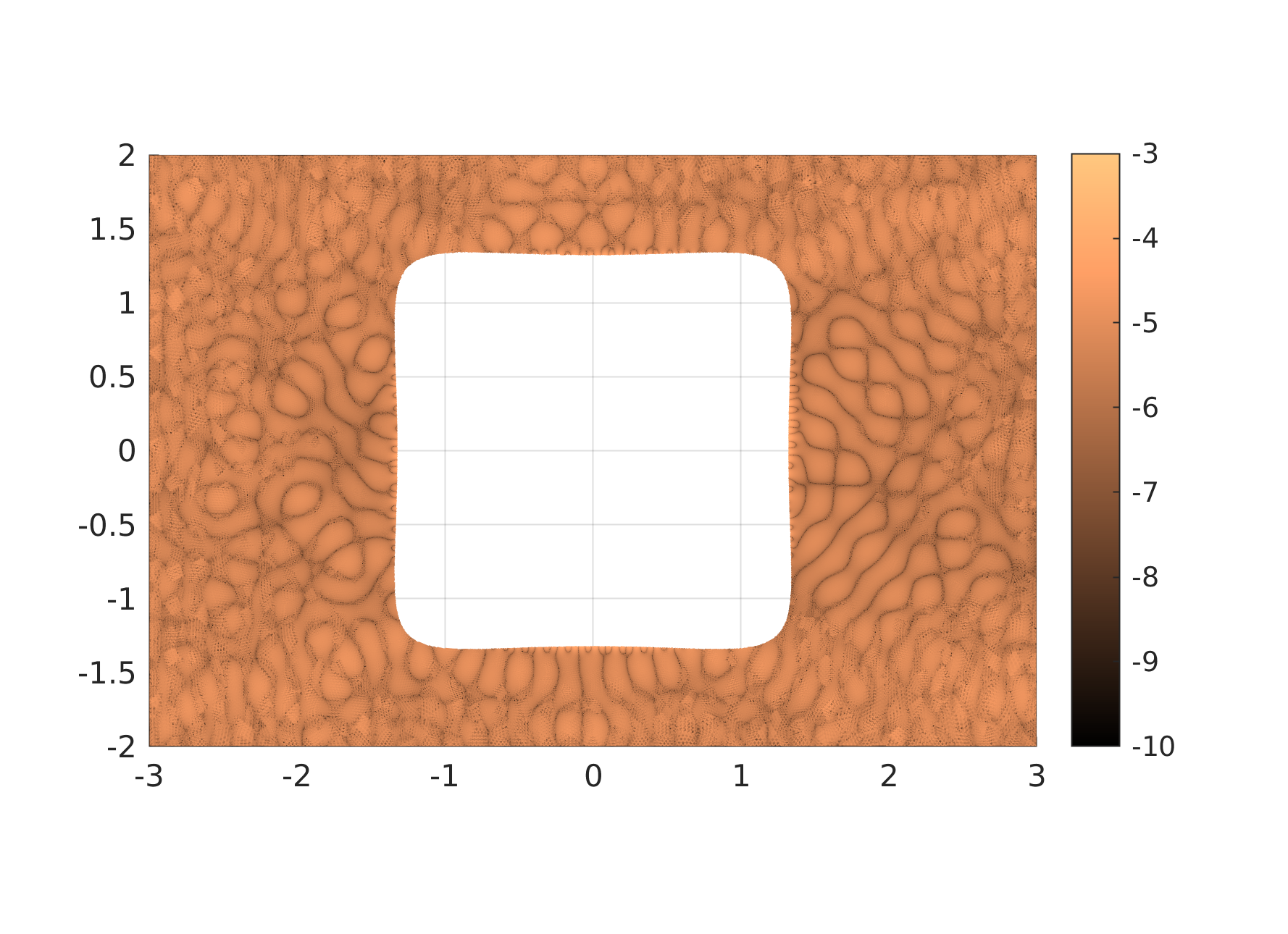}
  \]
\caption{\label{fig:yoda-fields-error} {\em Experiment \#3 (with parameter values as in Figure~\ref{fig:yoda-fields}. Absolute difference in log-scale between BEM-computed and FEM-computed approximation solutions in the overlapped region $ \Omega_{12}= (\mathbb{R}^2 \setminus \overline{\Omega}_1)\cap \Omega_2$.}}
  \end{figure}

\changeV{
\subsection*{Acknowledgement}
We sincerely thank the anonymous referees for  suggestions which helped to improve the paper. The first author (Dom\'{\i}nguez) is supported by the project  MTM2017-83490-P. 
The second author (Ganesh) gratefully acknowledges the support of the Simons Foundation.}


\begin{thebibliography}{10}

\bibitem{AdFo:2003}
R.A. Adams and J.J.F. Fournier.
\newblock {\em Sobolev spaces}, volume 140 of {\em Pure and Applied Mathematics
  (Amsterdam)}.
\newblock Elsevier/Academic Press, Amsterdam, second edition, 2003.

\bibitem{MR1735891}
S.~Bertoluzza.
\newblock The discrete commutator property of approximation spaces.
\newblock {\em C. R. Acad. Sci. Paris S\'er. I Math.}, 329(12):1097--1102,
  1999.

\bibitem{MR190518}
H.~Brakhage and P.~Werner.
\newblock \"{U}ber das {D}irichletsche {A}ussenraumproblem f\"{u}r die
  {H}elmholtzsche {S}chwingungsgleichung.
\newblock {\em Arch. Math.}, 16:325--329, 1965.

\bibitem{BrSc:2002}
S.C. Brenner and L.R. Scott.
\newblock {\em The mathematical theory of finite element methods}, volume~15 of
  {\em Texts in Applied Mathematics}.
\newblock Springer, New York, third edition, 2008.

\bibitem{Ci:2002}
P.G. Ciarlet.
\newblock {\em The Finite Element Method for Elliptic Problems}.
\newblock Classics in Applied Mathematics. Society for Industrial and Applied
  Mathematics, 2002.

\bibitem{colton:inverse}
D.~Colton and R.~Kress.
\newblock {\em Inverse Acoustic and Electromagnetic Scattering Theory}.
\newblock Springer, 4th edition, 2019.

\bibitem{MR937473}
M.~Costabel.
\newblock Boundary integral operators on {L}ipschitz domains: elementary
  results.
\newblock {\em SIAM J. Math. Anal.}, 19(3):613--626, 1988.

\bibitem{CoyleMonk}
J.~Coyle and P.~Monk.
\newblock Scattering of time-harmonic electromagnetic waves by anisotropic in
  homogeneous scatterers or impenetrable obstacles.
\newblock {\em SIAM J. Math. Anal.}, 37:1590--1617, 2004.

\bibitem{jan_lect}
P.~G. de~Gennes.
\newblock Soft {M}atter ({N}obel {L}ecture).
\newblock {\em Angew. Chem. Int. Ed. Engl.}, 31:842--845, 1992.

\bibitem{DoGaSay:2019}
V.~{Dom\'{\i}nguez}, M.~{Ganesh}, and F.J. {Sayas}.
\newblock {An overlapping decomposition framework for wave propagation in
  heterogeneous and unbounded media: Formulation, analysis, algorithm, and
  simulation}.
\newblock {\em J. Comput. Phys.}, 403:109052, 2020.

\bibitem{MR2033124}
V.~Dom\'{\i}nguez and F.-J. Sayas.
\newblock Stability of discrete liftings.
\newblock {\em C. R. Math. Acad. Sci. Paris}, 337(12):805--808, 2003.

\bibitem{MR3526814}
V.~Dom\'{\i}nguez and C.~Turc.
\newblock High order {N}ystr\"om methods for transmission problems for
  {H}elmholtz equations.
\newblock In {\em Trends in differential equations and applications}, volume~8
  of {\em SEMA SIMAI Springer Ser.}, pages 261--285. Springer, [Cham], 2016.

\bibitem{tmatrom}
M.~Ganesh and S.~C. Hawkins.
\newblock Algorithm 975: {TMATROM}---a {T-matrix} reduced order model software.
\newblock {\em ACM Trans. Math. Softw.}, 44:9:1--9:18, 2017.

\bibitem{ghh:2011}
M.~Ganesh, S.~C. Hawkins, and R.~Hiptmair.
\newblock Convergence analysis with parameter estimates for a reduced basis
  acoustic scattering {T-matrix} method.
\newblock {\em IMA J. Numer. Anal.}, 32:1348--1374, 2012.

\bibitem{GaMor:2016}
M.~Ganesh and C.~Morgenstern.
\newblock High-order {FEM}-{BEM} computer models for wave propagation in
  unbounded and heterogeneous media: application to time-harmonic acoustic horn
  problem.
\newblock {\em J. Comput. Appl. Math.}, 307:183--203, 2016.

\bibitem{mg2018}
M.~Ganesh and C.~Morgenstern.
\newblock High-order {FEM} domain decomposition models for high-frequency wave
  propagation in heterogeneous media.
\newblock {\em Comp. Math. Appl. (CAMWA)}, 75:1961--1972, 2018.

\bibitem{mg2019}
M.~Ganesh and C.~Morgenstern.
\newblock A coercive heterogeneous media {H}elmholtz model: formulation,
  wavenumber-explicit analysis, and preconditioned high-order {FEM}.
\newblock {\em Numerical Algorithms}, 83:1441--1487, 2020.

\bibitem{GeRe:2009}
C.~Geuzaine and J.-F. Remacle.
\newblock Gmsh: A 3-{D} finite element mesh generator with built-in pre- and
  post-processing facilities.
\newblock {\em International Journal for Numerical Methods in Engineering},
  79:1309 -- 1331, 2009.

\bibitem{Gillman2015}
A.~Gillman, A.~H. Barnett, and P.-G. Martinsson.
\newblock A spectrally accurate direct solution technique for frequency-domain
  scattering problems with variable media.
\newblock {\em BIT Numerical Mathematics}, 55(1):141--170, 2015.

\bibitem{Gr:2011}
P.~Grisvard.
\newblock {\em Elliptic problems in nonsmooth domains}, volume~69 of {\em
  Classics in Applied Mathematics}.
\newblock Society for Industrial and Applied Mathematics (SIAM), Philadelphia,
  PA, 2011.

\bibitem{Jan_sph_20_pap}
S.~C. Hawkins, T.~Rother, and J.~Wauer.
\newblock Numerical study of acoustic scattering by {J}anus spheres.
\newblock {\em J. Acoust. Soc. Am.}, 147:4097, 2020.

\bibitem{overlap-2}
C.~Hazard and M.~Lenoir.
\newblock On the solutions of time-harmonic scattering problems for maxwell's
  equations.
\newblock {\em SIAM J. Math. Anal.}, 27:1597--1630, 1996.

\bibitem{HsiWen:2008}
G.~C. Hsiao and W.~L. Wendland.
\newblock {\em Boundary integral equations}, volume 164 of {\em Applied
  Mathematical Sciences}.
\newblock Springer-Verlag, Berlin, 2008.

\bibitem{Ihlenburg:1998}
F.~Ihlenburg.
\newblock {\em Finite element analysis of acoustic scattering}, volume 132 of
  {\em Applied Mathematical Sciences}.
\newblock Springer-Verlag, New York, 1998.

\bibitem{overlap-1}
A.~Jami and M.~Lenoir.
\newblock A variational formulation for exterior problems in linear
  hydrodynamics.
\newblock {\em Comput. Methods Appl. Mech. Engrg.}, 16:341--359, 1978.

\bibitem{KiMo:1990}
A.~Kirsch and P.~Monk.
\newblock Convergence analysis of a coupled finite element and spectral method
  in acoustic scattering.
\newblock {\em IMA J. Numer. Anal.}, 10(3):425--447, 1990.

\bibitem{KiMo:1994}
A.~Kirsch and P.~Monk.
\newblock An analysis of the coupling of finite-element and {N}ystr\"om methods
  in acoustic scattering.
\newblock {\em IMA J. Numer. Anal.}, 14(4):523--544, 1994.

\bibitem{Kress:2014}
R.~Kress.
\newblock {\em Linear integral equations}, volume~82 of {\em Applied
  Mathematical Sciences}.
\newblock Springer, New York, third edition, 2014.

\bibitem{jan_appl_1}
M.~Lattuada and A.~Hatton.
\newblock Synthesis, properties and applications of janus nanoparticles.
\newblock {\em Nano Today}, 6:286--308, 2011.

\bibitem{McLean:2000}
W.~McLean.
\newblock {\em Strongly elliptic systems and boundary integral equations}.
\newblock Cambridge University Press, Cambridge, 2000.

\bibitem{Ned:2001}
J.-C. N\'{e}d\'{e}lec.
\newblock {\em Acoustic and electromagnetic equations}, volume 144 of {\em
  Applied Mathematical Sciences}.
\newblock Springer-Verlag, New York, 2001.
\newblock Integral representations for harmonic problems.

\bibitem{MR0373325}
J.~A. Nitsche and A.~H. Schatz.
\newblock Interior estimates for {R}itz-{G}alerkin methods.
\newblock {\em Math. Comp.}, 28:937--958, 1974.

\bibitem{Jan_sph_20_book}
T.~Rother.
\newblock {\em Sound Scattering on Spherical Objects}.
\newblock Springer, New York, 2020.

\bibitem{jan_appl_4}
T.~M. Ruhland, A.~H. Gr\"oschel, N.~Ballard, T.~S. Skelhon, A.Walther, A.~H.~E.
  M\"uller, and S.~A.~F. Bon.
\newblock Influence of {J}anus particle shape on their interfacial behavior at
  liquid-liquid interfaces.
\newblock {\em Langmuir}, 29:1388--1394, 2013.

\bibitem{SaVa:2002}
J.~Saranen and G.~Vainikko.
\newblock {\em Periodic integral and pseudodifferential equations with
  numerical approximation}.
\newblock Springer Monographs in Mathematics. Springer-Verlag, Berlin, 2002.

\bibitem{sayas}
F.~J. Sayas.
\newblock The validity of {J}ohnson-{N}\'{e}d\'{e}lec's {BEM-FEM} coupling on
  polygonal interfaces.
\newblock {\em SIAM Review}, 55:131--146, 2013.

\bibitem{ScZh:1990}
L.R. Scott and S.~Zhang.
\newblock Finite element interpolation of nonsmooth functions satisfying
  boundary conditions.
\newblock {\em Math. Comp.}, 54(190):483--493, 1990.

\bibitem{jan_appl_3}
M.~Vafaeezadeh and W.~R. Thiel.
\newblock Janus interphase catalysts for interfacial organic reactions.
\newblock {\em J. Mol. Liq.}, 315:113735, 2020.

\bibitem{jan_appl_2}
J.~Zhang, A.~Gryzbowski, and Granick.
\newblock Janus particle synthesis, assembly and application.
\newblock {\em Langmuir}, 33:6964--6977, 2017.

\end{thebibliography}

\appendix

\section{Sobolev convergence estimates for some projections}

 In this section we collect some {useful} results for projections on finite element spaces. 
 The first set of results is concerned with finite element spaces 
 on polygonal compact closed boundaries in $\mathbb{R}^2$.
There finite element spaces are inherited by taking  Dirichlet  trace of finite element spaces on triangular meshes. We finish this section proving a commutation property for Scott-Zhang type projections, which is required in Appendix B for deriving superconvergence results of the finite element solution in stronger norms
that are also valid in three dimensions. 
We think that most of the described results here belong to the  folklore in finite element analysis. We  include results below for the sake of completeness.

\subsection{ Projections on polygonal boundary finite element spaces}
Let $\Sigma$ be a polygonal simply connected closed curve with interior $\Omega\subset\mathbb{R}^2$. We consider $\{{\cal T}_h\}_{h\ge 0}$ a sequence of regular triangular meshes of $\Omega$ and denote by $\{\tau_h\}_{h\ge 0}$ that inherited on $\Sigma$. Without loss of generality we can assume that $h_\Sigma\approx h$, where $h_\Sigma$ and $h$ are the maximum, respectively, of the diameters of the elements of the grids ${\cal T}_h$ and ${\tau}_h$.  In addition, let 
\[
  \mathbb{P}_{h,d}:=\{{u}_h\in{\cal C}^0(\Omega) :\ u_h|_{T_h}\in\mathbb{P}_d\},\qquad \gamma_\Sigma \mathbb{P}_{h,d}
 \]
denote the space of continuous finite elements space in $\Omega$ and its boundary $\Sigma$. We denote by 
\[
\mathrm{Q}_{h}:{\cal C}(\Omega)\to   \mathbb{P}_{h,d},\qquad {\rm Q}^h_{\Sigma} : {\cal C}(\Sigma)\to   \gamma_\Sigma \mathbb{P}_{h,d}
\]
 the corresponding nodal (Lagrange) interpolation operators. Note that
 \[
\gamma_\Sigma   \mathrm{Q}_{h}u_h= {\rm Q}^h_{\Sigma} \gamma_\Sigma u_h,\quad \forall u_h\in\mathbb{P}_{h,d}. 
 \]
Our objectives in this section are twofold: (a) derive convergence estimates in the Sobolev norms $\|\cdot\|_{H^s(\Sigma)}$ for ${\rm Q}^h_{\Sigma}$ and  $s\ge 0$ (see \eqref{eq:3.2} for the definition we have taken for these spaces for $s> 1$);  (b) define an alternative   stable and convergent projection in weaker norms. In particular, we are interested in working with  functions $H^{1/2}(\Sigma)$, the trace space, for which the interpolant cannot be defined  since the space contains discontinuous functions.   Next, we  start with the first objective:

\begin{proposition}\label{prop:A1}
 For any $s\in [0,1]$ and $s\le t<d+1$ with $t> 1/2$ there exists $C>0$ such that
 \begin{equation}\label{eq:A:01}
  \|{\rm Q}^h_{\Sigma} f_\Sigma-f_\Sigma\|_{ H^s(\Sigma)}\le Ch_\Sigma^{t-s}\|f\|_{H^t(\Sigma)}.
 \end{equation}
 Furthermore, for any $t>d+1$ 
 \begin{equation}\label{eq:A:02}
  \|{\rm Q}^h_{\Sigma} f_\Sigma-f_\Sigma\|_{ H^s(\Sigma)}\le C h_\Sigma^{d+1-s}\|f\|_{H^t(\Sigma)}
 \end{equation}
 with $C>0$ depending only on $s\in[0,1]$ and $t$.
\end{proposition}
\begin{proof}
 
We start with a classical result: for  $s\in\{0,1\}$ and $t>\max\{1/2,s\}$~\cite[Th. 3.1.6]{Ci:2002}
\[
\|{\rm Q}^h_{\Sigma} f_\Sigma-f_\Sigma\|_{ H^s(\Sigma)} \le 
C'  h_\Sigma^{t-s}
\| f_\Sigma  \|_{{\cal H}^t(\Sigma)} 
\]
where $\{\Sigma_\ell\}_\ell$ are the edges of $\Sigma$ and $H^t(\Sigma_\ell)$ the corresponding classical Sobolev space, and 
\[
  \| f_\Sigma  \|_{{\cal H}^t(\Sigma)}^2:=\sum_{\ell } 
\|f_\Sigma\|^2_{ H^t(\Sigma_\ell)}.
\]
For $t>0$, with $t\ne 1,2,\ldots, $ we then use that the trace operator is also continuous in this norm
\begin{equation}\label{eq:A:03}
\| \gamma_\Sigma u  \|_{{\cal H}^t(\Sigma)} \le C\|u\|_{H^{t+1/2}(\Omega)},\quad \forall u\in H^{t+1/2}(\Omega)
\end{equation}
cf. \cite[Th. 1.5.2.8]{Gr:2011} (the bound in \eqref{eq:A:03} breaks precisely for $t=0,1,2\ldots$; see also \cite[Th. 4.2.7]{HsiWen:2008}) which proves \eqref{eq:A:01} for the considered non-integer values of $t$. By interpolation of Sobolev spaces, we can extend the result for any $t<d+1$ and $s\in[0,1]$, and hence \eqref{eq:A:01} holds. 

The bound \eqref{eq:A:02} can be proven by starting from
\[
 \|{\rm Q}^h_{\Sigma} f_\Sigma-f_\Sigma\|_{ H^s(\Sigma)} \le C h_\Sigma^{d+1-s}   \| f_\Sigma  \|_{{\cal H}^{d+1+\varepsilon}(\Sigma)},
\]
with $\varepsilon\in(0,1)$ 
and again using \eqref{eq:A:03}. 
\end{proof}

We will focus next on constructing a projection ${\rm P}_\Sigma^h:H^{1/2}(\Sigma)\to  \gamma_\Sigma \mathbb{P}_{h,d}$ with same convergence rates in an extended Sobolev scale. A key fact in this construction is the existence of a Scott-Zhang-type projection cf. \cite{ScZh:1990} (see also \cite{MR2033124}). Specifically, there exists a continuous linear mapping   $\Pi_{h,d}: H^s(\Omega)\to \mathbb{P}_{h,d}$, with $s>{1/2}$, satisfying
\begin{subequations} \label{eq:Ph}
\begin{enumerate}
  
 \item $\Pi_h$ is a projection:
 \begin{equation}\label{eq:Ph:01}
   \Pi_h v_h=v_h,\quad \forall v_h\in \mathbb{P}_{h,d}.
   \end{equation} 
   \item The image of any element with null trace has null trace as well: 
 \begin{equation}\label{eq:Ph:02}
\gamma_\Sigma \Pi_h v=0,\quad \text{if}\quad \gamma_\Sigma v=0.
   \end{equation}
   
\item It is quasi-local: for any triangle $K\in{\cal T}_h$, 
\begin{equation}\label{eq:Ph:04c}
  \Pi_h  v|_T = 0, \quad  \text{if}\quad v|_{S_T}=0,  \text{ with } S_T:=\bigcup_{\stackrel{T'\in{\cal T}_h}{\overline{T}'\cap \overline{T}\ne \emptyset} } \overline{T}'.
\end{equation}
Furthermore, for any $0\le  s\le t $ with $s\in[0,3/2)$ and  $1/2<t\le d+1$ there exists $C=C(s,t)$ so that
  \begin{equation}\label{eq:Ph:04}
   \|\Pi_h v- v\|_{H^s(T)} \le C h_T^{t-s}\|v\|_{H^t(S_T)}.
  \end{equation}
The constant $C(s,t)$ depends only on the chunkiness parameter of the grid. As consequence of \eqref{eq:Ph:04} we have
 \begin{equation}\label{eq:Ph:03}
 \left[\sum_{T\in{\cal T}_h} h_T^{2(s-t)} \|\Pi_h v- v\|^2_{H^s(T)}\right]^{1/2} \le  
C
\|v\|_{H^t(\Omega)}.
 \end{equation}
\end{enumerate}
\end{subequations}
Since for $s\in[0,1/2)\cup [1,3/2)$, we have 
  \begin{equation}\label{eq:splitSobolevNorm}
  \|v\|_{H^s(\Omega)}\le C \left[\sum_{T\in{\cal T}_h}  \|v\|^2_{H^s(T)}\right]^{1/2}, 
  \end{equation}
the simpler estimate 
  \begin{equation}\label{eq:Ph:05}
    \|\Pi_h v- v\|_{H^s(\Omega)} \le C h^{t-s}\|v\|_{H^t(\Omega)}.
  \end{equation}
 can be  derived from \eqref{eq:Ph:04}, first for $s\in[0,1/2)\cup [1,3/2)$, and then can be extended for  $s\in[1/2,1)$ by interpolation of Sobolev spaces.  
 
 {
 \begin{remark} The proof of \eqref{eq:splitSobolevNorm}  is based on working  with the Slobodeckij form of the Sobolev norm: for non-integer $s> 0$ and for any Lipschitz domain $D\subset\mathbb{R}^m$
 \[
 \|u\|_{H^s(D)}^2 = 
 \|u\|_{H^{\underline{s}}(D)}^2 + \sum_{|\bm\alpha|=\underline{s}} \int_D\int_D 
 \frac{|\partial_{\bm\alpha}u(\bm{x})-\partial_{\bm\alpha}u(\bm{y})|}{|{\bm x}-{\bm y}|^{m+2({s}-\underline{s})}}{\rm d}\bm {x}\,
 {\rm d}{\bm y},
 \]
 where $\underline{s}$ is the largest integer less than $s$.  Then, for $t\in (0,1/2)$, it is easy to derive the estimate
 \[
 \|v\|^2_{H^t(\Omega)}\le C_{t,\Omega} \left[\sum_{T\in{\cal T}_h}  \|v\|^2_{H^t(T)} +\int_{T} \frac{|v(\bm{x})|^2}{\rho(\bm{x},\partial T) ^{2t}}\,{\rm d}\bm{x}\right] 
 \]
 where
 \[
 \rho({\bm x},\partial T)=\inf_{\bm{y}\in\partial T}|{\bm x}-{\bm y}|. 
 \]
 A Hardy-type inequality (see \cite[Lemma 3.32]{McLean:2000}) allows to bound  the above integral  term by 
\[
\int_{T} \frac{|v(\bm{x})|^2}{\rho(\bm{x},\partial T) ^{2t}}\,{\rm d}\bm{x}\le C \|v\|^2_{H^t(T)}.
\]
The constant $C$ appearing above depends on $t$ and on the chunkiness parameter of $T$. This inequality does not hold for $t\in[1/2,1)$ since the last integral in the left hand side is expected to be non-convergent. The result for $t>1$ is a simple extension of this argument.

 \end{remark}
 }
The last ingredient is a right inverse of the trace operator ${\rm R}_{\Sigma \Omega }:H^s(\Sigma)\to  H^{s+1/2}(\Omega)$  for $s\in(0,1]$ . For instance, one can take
\[
 {\rm R}_{\Sigma \Omega}  g_\Sigma := u,\quad \text{with $u$ satisfying } \Delta u =0,\quad\text{$\gamma_\Sigma u =g_\Sigma $}, 
\]
the Dirichlet solution for the Laplace operator. 
Such operator  is continuous, not only for $s\in (0,1)$, but it attains the end point $s=1$ as well (cf.  \cite[Chapter 6]{McLean:2000}; see Theorem 6.12 and the discussion following it). 

We are ready to define the desired projection on the finite element space on the boundary $\gamma_\Sigma \mathbb{P}_{h,d}$. To this end, we first set
\[
 {\rm P}_\Sigma^h:=\gamma_\Sigma\Pi_h  {\rm R}_{\Sigma \Omega}.
\]

\begin{proposition}\label{prop:A2}
 Let ${\rm P}_\Sigma^h: H^s(\Sigma)\to \gamma_\Sigma\mathbb{P}_{h,d}$ as above. Then 
 \begin{equation}\label{eq:5.1}
\|{\rm P}_\Sigma^hf_\Sigma -  f_\Sigma\|_{H^s(\Sigma)}\le C  h_\Sigma^{t-s}\|f_\Sigma\|_{H^{t}(\Sigma)},\quad  0\le s<1,\quad s\le t<d+1,\quad t>0,
\end{equation}
where $C$ is independent of $f_\Sigma$ and $h$.  Furthermore, for any $s\in [0,1) $ and $t>d+1$, there exists $C>0$ such that 
 \begin{equation}\label{eq:A:02b}
  \|{\rm P}^h_{\Sigma} f_\Sigma-f_\Sigma\|_{ H^s(\Sigma)}\le C h_\Sigma^{d+1-s}\|f\|_{H^t(\Sigma)}.
 \end{equation} 
\end{proposition} 
\begin{proof}
 By construction,   $\Pi_h $ is a projection. Indeed, if ${\rm Q}_h:{\cal C}(\Omega)\to \mathbb{P}_{h,d}$ is the nodal $d-$Lagrange interpolation operator, and since $\gamma_\Sigma {\rm Q}_h = {\rm Q}^h_{\Sigma}\gamma_\Sigma$, we have 
 that  
\begin{eqnarray*}
 \gamma_\Sigma \Pi_h
 {\rm R}_{\Sigma \Omega} f_\Sigma^h-f_\Sigma^h&=& 
 \gamma_\Sigma (\Pi_h -{\rm Q}_h) {\rm R}_{\Sigma \Omega }  f_\Sigma^h
 = \gamma_\Sigma \Pi_h\underbrace{({\rm  I} -{\rm Q}_h) {\rm R}_{\Sigma \Omega }  f_\Sigma^h}_{\in H_0^1(\Omega)}  = 0
\end{eqnarray*}
by \eqref{eq:Ph:02}.
 
On the other hand, for $t\in(1/2,1]$ with $t\ge s$, and $s\in (0,1)$, and by the continuity of the trace operator,  
\begin{eqnarray}
 \|{\rm P}_\Sigma^hf_\Sigma -  f_\Sigma\|_{H^s(\Sigma)} 
 &\le& 
  C \|\Pi_h 
 {\rm R}_{\Sigma \Omega }  f_\Sigma -    {\rm R}_{\Sigma \Omega }  f_\Sigma\|_{H^{s+1/2}(\Omega)}
 \nonumber\\
  &\le&
  C' h_\Sigma^{t-s}\|
 {\rm R}_{ \Omega \Sigma }  f_\Sigma\|_{H^{t+1/2}(\Omega)}\le 
  C'' h_\Sigma^{t-s}\| f_\Sigma\|_{H^{t}(\Sigma)} \label{eq:5.5}.
\end{eqnarray}

To prove the estimate in $H^0(\Sigma)=L^2(\Sigma)$ we recall  the trace inequality cf \cite[Th. 1.6.6]{BrSc:2002}
\[
 \|\gamma_\Sigma u\|_{L^2(\Sigma)}^2\le C \big(\|u\|_{L^2(\Omega)}^2+ \|u\|_{L^2(\Omega)}\|u\|_{H^1(\Omega)}\big)\le  C(1+\tfrac12h_\Sigma^{-1})\|u\|_{L^2(\Omega)}^2+ \tfrac12 C h_\Sigma \|u\|^2_{H^1(\Omega)}
\]
which yields 
\begin{eqnarray*}
 \|{\rm P}_\Sigma^hf_\Sigma -  f_\Sigma\|_{L^2(\Sigma)}&\le &C h_\Sigma^{1/2}
\|\Pi_h 
 {\rm R}_{ \Omega \Sigma }  f_\Sigma -    {\rm R}_{ \Omega \Sigma }  f_\Sigma\|_{H^{1}(\Omega)}+C h_\Sigma^{-1/2}\|\Pi_h 
 {\rm R}_{ \Omega \Sigma }  f_\Sigma -    {\rm R}_{ \Omega \Sigma }  f_\Sigma\|_{H^{0}(\Omega)} \\\
 &\le& C'h_\Sigma^{t}\|{\rm R}_{ \Omega \Sigma }  f_\Sigma\|_{H^{t+1/2}(\Omega)}\le 
 C'h^t_\Sigma\| f_\Sigma\|_{H^t(\Sigma)}. 
\end{eqnarray*}
We have then proved  \eqref{eq:5.1} for $0\le s<1$ and $\max\{s,1/2\}<t\le 1$ and therefore it only remains to extend this result for $t>1$.  But, 
\begin{eqnarray*}
 \|{\rm P}_\Sigma^hf_\Sigma -  f_\Sigma\|_{H^s(\Sigma)}&=&
 \|{\rm P}_\Sigma^h(f_\Sigma -  {\rm Q}_\Sigma^h f) -  (f_\Sigma -  {\rm Q}_\Sigma^h f)\|_{H^s(\Sigma)}\le  C h_\Sigma^{1-s}\|f_\Sigma -  {\rm Q}_\Sigma^h f_\Sigma\|_{H^1(\Sigma)}
\end{eqnarray*}
and the result follows from Proposition \ref{prop:A1}.

\end{proof}

Note that, using the above arguments, \eqref{eq:5.1} cannot be extended to $s=1$.

\subsection{Commutator properties for Scott-Zhang projections} 
We end this section showing a commutation property for $\Pi_h$. We stress out that the result is also valid in 3D for tetrahedral meshes with minor but direct modifications.

\begin{lemma}\label{lemma:A2}
For any $\varpi\in{\cal C}^\infty(\Omega)$ 
there exists  $C>0$ so that for $d\ge 2$ and $s\in [0,3/2)$
\[
\|\Pi_h(\varpi w_h)-\varpi w_h\|_{H^s(\Omega)}\le 
C h \|w_h\|_{H^s(\Omega)},\quad \forall w_h\in\mathbb{P}_{h,d}.
\]
For $d=1$ (linear finite elements), we have instead
\[
\|\Pi_h(\varpi w_h)-\varpi w_h\|_{H^s(\Omega)}\le 
C h^{\min\{1,2-s\}} \|w_h\|_{H^s(\Omega)},\quad \forall w_h\in\mathbb{P}_{h,1}.
\]
\end{lemma}
\begin{proof}
Consider ${\rm Q}_h:{\cal C}^0\to \mathbb{P}_{h,d}$ the classical nodal interpolant on
 $\mathbb{P}_{h,d}$ which satisfies
 \begin{equation}\label{eq:la_eq}
 \|u -{\rm Q}_h u\|_{H^s(T)}\le  C_{s,t}h_T^{t-s}  \|u \|_{H^t(T)}, \qquad \forall T\in{\cal T}_h
 \end{equation}
 for any $0\le s<3/2$, and $s\le t\le d+1$ with $t>1$ for triangular meshes in bidimensional polygonal domains ($t>3/2$ for 3D polygonal domains). 
 In~\eqref{eq:la_eq}, $h_T$ is the diameter of the element $T$, and $C_{s,t}$ is a constant independent of $T$ and $u_h$.

Recall  also the inverse inequalities which hold locally on each triangle:
\begin{equation}\label{eq:inv:inequalities}
 \|w_h\|_{H^s (T)}\le C_{s,t} h_T^{ t- s}\|w_h\|_{H^t(T)},\quad \forall w_h\in\mathbb{P}_{h,d}
\end{equation}
with $C$ depending again only on $d$, $s\ge t$ and the chunkiness parameter of the grid. Then,  locally on each element it holds
 \[
\|\varpi w_h\|_{H^{d+1}(T)}\le C_{\varpi}\|w_h\|_{H^{d+1}(T)}=
   C_{\varpi}\|w_h\|_{H^{d}(T)}\le C'_{\varpi,s} h_T^{s-d-1}\|w_h\|_{H^{s}(T)}
 \]
 for any $s\in[0,d]$.  
  
We  start with the $s\in [0,1/2)$ case. Using   \eqref{eq:splitSobolevNorm}, the convergence properties of $\Pi_h$ \eqref{eq:Ph:04}, the interpolant operator ${\rm Q}_h$
and the  inverse inequality \eqref{eq:inv:inequalities},
 \begin{align*}
\|\Pi_h(\varpi w_h)  - & \varpi w_h\|_{H^s(\Omega)}^2 \\
 & \le  C_s \sum_{T\in{\cal T}_h}
\|\Pi_h(\varpi w_h-{\rm Q}_h(\varpi w_h))-(\varpi w_h-{\rm Q}_h(\varpi w_h))\|_{H^s(T)}^2\\
&\le 
C_s' \sum_{T\in{\cal T}_h} h_T^{2-2s}
\|  \varpi w_h-{\rm Q}_h(\varpi w_h)\|_{H^1(T)}^2 \le C'_s\sum_{T\in{\cal T}_h} h_T^{2d+2-2s}
 \| \varpi w_h\|_{H^{d+1}(T)}^2 \\
 & \le 
 C''_ {s,\varpi}\sum_{T\in{\cal T}_h}  h_T^{2}
 \| w_h\|_{H^{s}(T)}^2 \le C'_{s,\varpi} h^{2}\|w_h\|^2_{H^s(\Omega)}.
  \end{align*}

For $s\in[1,3/2)$, we proceed in a similar manner: Using the stability of the Scott-Zhang type projection 
   \begin{align*}
\|\Pi_h(\varpi w_h)  -  \varpi w_h\|_{H^s(\Omega)}^2  & \le  C_s \|{\rm Q}_h(\varpi w_h)  -   \varpi w_h\|_{H^s(\Omega)}^2\\
 & \le C'_ {s}\sum_{T\in{\cal T}_h}  h_T^{2d-2s}
 \| \varpi  w_h\|_{H^{d+1}(T)}^2\le C_ {s,\varpi}\sum_{T\in{\cal T}_h}  h_T^{2d-2s}
 \|  w_h\|_{H^{d}(T)}^2. 
  \end{align*}
  The result follows now using either the direct bound
  \[
  \|  w_h\|_{H^{d}(T)}\le \|  w_h\|_{H^{s}(T)},\quad \text{for $d =1$},
  \]
  or the inverse inequality \eqref{eq:inv:inequalities} for $d\ge 2$. 
  
 We have then proven the result for $s\in[0,1/2)\cup [1,3/2)$. For the intermediate values of $s\in(1/2,1)$ we just  invoke the Interpolation Theory of Sobolev spaces. 
\end{proof}

%
\begin{remark} We note that in~\cite{MR1735891} the weaker estimate is proven:
\[
 \|\Pi_h(\varpi v_h)-\varpi v_h\|_{H^s(\Omega)}\le 
 C h^s \|v_h\|_{H^s(\Omega)},\quad \forall v_h\in\mathbb{P}_{h,d}
 \]
 for $s\in(1/2,1]$ and for any projection on finite element spaces satisfying rather general assumptions which include, in particular, our case.
 \end{remark}

\section{Convergence/superconvergence  of FEM in stronger  norms}

In this appendix   $\Omega$ is a polygonal domain with boundary $\Sigma$ in both $\mathbb{R}^m$ with $m=2,3$ and
\[
 b_{n,\Omega}(u,v):=(\nabla u,\nabla v)_\Omega -  (u,v)_{n,\Omega}:= \int_{\Omega}\nabla u\cdot\nabla v-  \int_{\Omega}  n \,uv
\]
is the bilinear form associated to the Dirichlet problem 
\[
\left|
\begin{array}{rcl}
 \Delta u+ n u &=&f\\
 \gamma_\Sigma u&=&g_\Sigma.
 \end{array}
 \right.
\]
Here $1-n$ is a $L^2-$function with compact support $\Omega_0$ in $\Omega$.  As before we  consider a sequence of regular grids made up of conformal triangular/tetrahedral elements ${\cal T}_h$. The parameter $h$ refers again to the maximum of the diameters of the elements in such a way that we write $h\to 0$ to mean that the diameters of the elements tend to zero. On ${\cal T}_h$ we construct a continuous finite element space of the elements which   are polynomials of degree $d$ on each $T\in{\cal T}_h$  and denote it by $\mathbb{P}_{h,d}$. 

We will work with the numerical solution given by the usual FEM scheme:
\[
\left|
\begin{array}{rcl}
u_h\in \mathbb{P}_{h,d}\\
 b_{n,\Omega}(u_h,v_h)&=&-(f,v_h)_\Omega,\quad \forall v_h\in\mathbb{P}_{h,d} \cap H_0^1(\Omega)\\
 \gamma_\Sigma u&=&g_\Sigma^h.
 \end{array}
 \right.
\]
Here, $  \gamma_\Sigma\mathbb{P}_{h,d} \ni g_\Sigma^h\approx g_\Sigma$. 
Some preliminary results can be listed at this point. Provided that $g_\Sigma\in H^{1/2}(\Sigma)$, $f\in L^2(\Omega)$ then  $u\in H^1(\Omega)$ with 
\[
 \|u\|_{H^1(\Omega)}\le C\big(\|f\|_{L^2(\Omega)}+\|g_\Sigma\|_{H^{1/2}(\Sigma)}\big).
\]
We  recall  the standard error result~\cite{ScZh:1990,MR2033124}
  \[
  \|u-u_h\|_{H^1(\Omega)}\le C\Big[ \inf_{v_h\in \mathbb{P}_{h,d}}\|u-v_h\|_{H^1(\Omega)}+
 \|g_\Sigma- g_\Sigma^h\|_{H^{1/2}(\Omega)}
 \Big]
 \]
which gives us a  convergence estimate in the natural norm. We however want to explore convergence estimates in stronger norms in the subdomains where the solution $u$ is more regular. The following theorem presents the main result of this appendix. The proof of this result uses similar ideas as those presented in \cite{MR0373325} for proving superconvergence in $H^1$ for a more general class of partial differential equations. We include it in this work for the sake of completeness.

\begin{theorem}\label{theo:A02} Let  $D'$,  $D$ be two domains  with $D'\subset \overline{D}'
\subset D\subset  \overline{D}\subset \Omega\setminus\Omega_0$ and let $\{{\cal T}_h\}$ be a sequence of regular grids with $h\to 0$ which are  \changeSecondReferee{quasi-uniform in $D'$ cf. \eqref{eq:quasiuniform}}.  Then for any  $\varepsilon\in[0,1/2)$,  $u\in H^{r+1}(D')$ with $r\in[\varepsilon, d]$,  there exists $C>0$ depending on $\varepsilon$, $r$, $D$, $D'$ so that  for any grid fine enough ${\cal T}_h$ it holds
 \begin{equation}\label{eq:01:theo:A02}
  \|u-u_h\|_{H^{1+\varepsilon}(D')}\le C\big[ h^{-\varepsilon}_D \|u-u_h\|_{H^{0}(D)}+ h_{D}^{r-\varepsilon} \|u\|_{H^{r+1}(D)} +
  h^{1-\varepsilon}_D \|u-u_h\|_{H^{1}(D)}\big]
 \end{equation}
 where $h_{D}$ is the maximum of the diameters of  the \changeV{elements} contained in $D$. 
\end{theorem} 
 


\begin{proof}
Let  $\Pi_h$ be a Scott-Zhang type projection satisfying \eqref{eq:Ph}, and  hence  the commutator property stated in Lemma \ref{lemma:A2}.   Take a smooth cut-off function  $\varpi$  satisfying 
\begin{equation}\label{eq:B00}
\begin{aligned}
\varpi|_{D'}&\equiv 1,\quad \mathop{\rm supp }\varpi\subset D,\\
\Pi_h(\varpi v_h)|_{D'}&=v_h|_{D' },\quad\forall v_h\in \mathbb{P}_{h,d},\quad 
\Pi_h(\varpi v)|_D\in H_0^1(D),\quad \forall v\in H^1(D). 
\end{aligned}
\end{equation}
We consider the $H^1(D)$ inner product
\[
\beta_{D}(u,v):= (\nabla u,\nabla v)_{D} +(u,v)_{D}, 
\]
and set
\[
e_h:=u-u_h,\quad \widetilde{e}_h:=\Pi_h(\varpi u)-\Pi_h(\varpi u_h)\in \mathbb{P}_{h,d},\quad \widetilde{e}^h:=\varpi u-\Pi_h(\varpi u).
\]
Using \eqref{eq:B00}, 
\begin{equation}\label{eq:B01}
 \varpi  e_h|_{D'} = (\widetilde{e}_h+\widetilde{e}^h)|_{D'}.
\end{equation}
Then
\begin{eqnarray}
 \|e_h\|_{H^{1+\varepsilon}(D')}&\le& 
 \|\varpi e_h\|_{H^{1+\varepsilon}(D)}\le  \|\widetilde{e}_h\|_{H^{1+\varepsilon}(D)}+
  \|\widetilde{e}^h\|_{H^{1+\varepsilon}(D)},
  \label{eq:I1I2}
 \end{eqnarray}
For  bounding the first term we start applying the inverse inequality 
\begin{equation}\label{eq:I1}
\|\widetilde{e}_h\|_{H^{1+\varepsilon}(D)}\le C_{\varepsilon} h_{D}^{-\varepsilon}  \|\widetilde{e}_h\|_{H^{1}(D)}
 = C_{\varepsilon} h_{D}^{-\varepsilon}\beta_{D}(\widetilde{e}_h,v_h),\quad \text{with }
 v_h =\frac{1}{\|\widetilde{e}_h\|_{H^1(D)}} \widetilde{e}_h\in \mathbb{P}_{h,d}\cap H_0^1(D), 
\end{equation} 
and continue applying   the decomposition
\begin{eqnarray*}
\widetilde{e}_h& =&  (\varpi u- \varpi u_h) + (\Pi_h-{\rm I})(\varpi u-\varpi \Pi_h u)+
(\Pi_h-{\rm I})(\varpi \Pi_h u-\varpi u_h). 
\end{eqnarray*}
Hence, 
\begin{eqnarray}
\|\widetilde{e}_h\|_{H^{1+\varepsilon}(D)} &\le& C_{\varepsilon}h_{D}^{-\varepsilon}\Big(\beta_{D}(\varpi u- \varpi u_h,v_h) +
 \beta_{D}((\Pi_h-{\rm I})(\varpi u-\varpi \Pi_h u),v_h)\nonumber \\
 && +
 \beta_{D}((\Pi_h-{\rm I})(\varpi \Pi_h u-\varpi u_h),v_h)\Big) \nonumber\\
  &\le& C_{\varepsilon}h_{D}^{-\varepsilon}\Big(\beta_{D}(\varpi e_h,v_h) 
  + \|(\Pi_h-{\rm I})(\varpi u- \varpi \Pi_h u)\|_{H^{1}(D)}\nonumber\\
  && + \|(\Pi_h-{\rm I})(\varpi \Pi_h u-\varpi u_h)\|_{H^{1}(D)}\Big) \nonumber\\ 
  &\le& C_{\varepsilon}h_{D}^{-\varepsilon} \big(\beta_{D}(\varpi e_h,v_h)+
  \|\varpi u- \varpi \Pi_h u\|_{H^{1}(D)}\nonumber+
h_{D}\| (\Pi_h u-u) +(u- u_h)\|_{H^1(D)}\big)\nonumber\\
  &\le& C_{\varepsilon}h_{D}^{-\varepsilon} \big( \beta_{D}(\varpi e_h,v_h)+
\|  \Pi_h u-u\|_{H^{1}(D)}+
h_{D}\| e_h\|_{H^1(D)} \big),  \label{eq:a11}
\end{eqnarray}
where to bound the last term, we used the the commutator property, cf.  Lemma \ref{lemma:A2} with $w_h=\Pi_h u-u_h$.  Next, 
we consider the first  term in~\eqref{eq:a11} and obtain a bounding estimate. 
Observe that
\begin{eqnarray*}
  \beta_{D}(\varpi {e}_h,v_h)&=& \int_{D} e_h\nabla \varpi \cdot\nabla v_h-
   \int_{D} v_h\nabla e_h \cdot\nabla \varpi+\beta_{D}(e_h,\varpi v_h)\\
   &=&2\int_{D} e_h\nabla \varpi \cdot\nabla v_h+
   \int_{D} v_h  e_h \Delta\varpi\\
   &&+\beta_{D}(e_h,\varpi v_h-\Pi_h(\varpi v_h))
   + \beta_{D}(e_h,\Pi_h(\varpi v_h)).
 \end{eqnarray*}
 Thus, and after applying again Lemma \ref{lemma:A2} to the third term above, 
\begin{equation}
 \label{eq:a12}
\beta_{D}(\varpi {e}_h,v_h)\le C\|e_h\|_{L^2(D)}+Ch_{D} \|e_h\|_{H^1(D)} +\beta_{D}(e_h,\Pi_h(\varpi v_h)). 
\end{equation}
Finally, using the orthogonality relation for the Galerkin solution,  
\begin{eqnarray}
\beta_{D}(e_h,\Pi_h(\varpi v_h))&=&
 \underbrace{\int_{\Omega}\nabla e_h\cdot\nabla \Pi_h(\varpi v_h)
 - \int_{\Omega} n\: e_h \Pi_h(\varpi v_h)}_{=0}+\int_{D} (1+n) e_h \Pi_h(\varpi v_h)\nonumber\\
 &\le&  C \|e_h\|_{L^2(D)}\|\varpi v_h\|_{L^2(\Omega)}
 \le C \|e_h\|_{L^2(D)}.\label{eq:a13} 
\end{eqnarray}

Plugging \eqref{eq:a13} in  \eqref{eq:a12} and next \eqref{eq:a12} in  \eqref{eq:a11} yield
\begin{equation}\label{eq:I1b}
\|\widetilde{e}_h\|_{H^{1+\varepsilon}(D')}\le C h_{D}^{-\varepsilon}  \big(\|e_h\|_{L^2(D)} +  \|  \Pi_h u-u\|_{H^{1}(D)}+
  h_{D}\| e_h\|_{ L^2(D)}  \big) 
 \end{equation}
Therefore, \eqref{eq:I1b}  with \eqref{eq:I1I2}  prove 
\begin{equation}\label{eq:a19}
\begin{aligned}
\|e_h\|_{H^{1+\varepsilon}(D')}\ \le &\ C_{\varepsilon}\big( h_{D}^{-\varepsilon}\|e_h\|_{L^2(D)}+
h_D^{1-\varepsilon}\|e_h\|_{H^1(D)}\\
\  &\ +  h_{D}^{-\varepsilon} \|  \Pi_h u-u\|_{H^{1}(D)}
+\| \Pi_h(\omega u)-\omega u\|_{H^{1+\varepsilon}(D)}\big),
\end{aligned}
\end{equation}
for any $\varepsilon\in [0,1/2)$. The result follows readily from the convergence estimates for $\Pi_h$.

\end{proof}

A convergence estimate in $L^2(\Omega)$ is needed to obtain a convergence result of the finite element solution in $D$. This is done next using a variant of the classical Aubin-Nische argument 

\begin{lemma}[Aubin-Nitsche trick]\label{lemma:Nitsche-Trick}
 Then for $\delta\in (1/2,1]$ it holds 
 \[
  \|u-u_h\|_{L^2(\Omega)}\le C \Big[ h^\delta\|u-u_h\|_{H^1(\Omega)}+  \|g_\Sigma-g_\Sigma^h\|_{L^2(\Sigma)}\Big]
 \]
 where $h$ denotes the maximum of the diameters of the elements of the mesh ${\cal T}_h$ of $\Omega$. 
\end{lemma}
\begin{proof} Let us denote  as before $e_h:=u-u_h$
and take $w$ the (unique) solution of 
\[
 \left|
 \begin{array}{l}
 w\in H_0^1(\Omega)\\
 \Delta w+n\: w = -e_h
 \end{array}
 \right.\quad \text{or in  variational form},\quad
 \left|
 \begin{array}{l}
 w\in H_0^1(\Omega)\\
 b_{\Omega,n}(w,v)=(e_h,v)_\Omega,\quad \forall v\in H_0^1(\Omega).
 \end{array}
 \right.
\]
It is known that there exists $\delta\in (1/2,1]$ such that 
 \[
\|w\|_{H^{1+\delta}(\Omega)}\le C_{\delta} \|e_h\|_{L^2(\Omega)} 
 \]
with $C_\delta$ independent of $e_h$.

Then
\begin{eqnarray*}
 \|e_h\|_{L^2(\Omega)}^2&=&(e_h,e_h)_\Omega=-(\Delta w+n\:w,e_h)_{\Omega}\\
 &=&b_{\Omega,n}(w,e_h)+\int_\Sigma \partial_n w\, e_h=b_{\Omega,n}(w-\Pi_hw,e_h)+\int_\Sigma \partial_n w\, e_h
\end{eqnarray*}
where we have made use of the Galerkin orthogonality of the discrete solution and the fact that $\Pi_h w\in \mathbb{P}_{h,d}\cap H_0^1(\Omega)$ for $w\in H_0^1(\Omega)$. Notice that from the trace theorem
\[ 
\|\partial_n w\|_{L^2(\Sigma)}\le
\|\gamma_\Sigma \nabla w \|_{L^2(\Sigma)}\le C \| w\|_{H^{1+\delta}(\Omega)}.
\]
Therefore, 
\begin{eqnarray*}
 \|e_h\|_{L^2(\Omega)}^2&\le& C\big[\|w-\Pi_h w\|_{H^1(\Omega)}\|e_h\|_{H^1(\Omega)}+
 \| w\|_{H^{1+\delta}(\Omega)}\|e_h\|_{L^2 (\Sigma)}\big]\\
 &\le&C'\big[h^\delta \|e_h\|_{H^1(\Omega)}+
 \|g_\Sigma-g_\Sigma^h\|_{L^2(\Sigma)}\big]\|w\|_{H^{1+\delta}(\Omega)}\\
& \le& C''\big[h^\delta \|e_h\|_{H^1(\Omega)}+
 \|g_\Sigma-g_\Sigma^h\|_{L^2(\Sigma)}\big]\|e_h\|_{L^2(\Omega)}.
\end{eqnarray*} 
The proof is now completed. 
\end{proof}

The parameter $\delta $ in Lemma~\ref{lemma:Nitsche-Trick} is nothing but the extra regularity grade, from a Sobolev point of view, of the homogeneous Dirichlet problem
\[
 \Delta w+ n\,w = g,\quad  {\gamma_{\Sigma }w} =0. 
\]
It is a very well established result, see for instance \cite{Gr:2011}, that $\delta\in (1/2,1]$ and it attains to be $1$  for convex polygons in $\mathbb{R}^2$.

\begin{corollary}\label{cor:a.5}
Under the same assumptions as those taken in Theorem \ref{theo:A02} and Lemma \ref{lemma:Nitsche-Trick}, we have that for any $\varepsilon\in [0,1/2)$
\[ 
   \|u-u_h\|_{H^{1+\varepsilon}(D')}\ \le\   C \big( (h^{\delta} h^{-\varepsilon}_D +h_D ^{1-\varepsilon} ) \|u-u_h\|_{H^{1}(\Omega)}
   +
    h^{-\varepsilon}_D \|g_\Sigma-g_\Sigma^h\|_{L^2(\Sigma)}
     +  h_{D}^{r-\varepsilon} \|u\|_{H^{r+1}(D)}\big) . 
\]
In particular, if $h^{\delta} h^{-\varepsilon}_D\to 0$ and $g_\Sigma^h$ is taken satisfying
\[
h^{-\varepsilon}_D \|g_\Sigma-g_\Sigma^h\|_{L^2(\Sigma)}\le C(h)\|g_\Sigma-g_\Sigma^h\|_{H^{1/2}(\Sigma)}
\]
with $C(h)\to 0$ as $h\to 0$, there exists $\eta(h)$ with $\eta(h)\to 0$ as $h\to 0$  such that
\begin{equation}\label{eq:01}
    \|u-u_h\|_{H^{1+\varepsilon}(D')}\le \eta(h)\big( \inf_{v_h\in \mathbb{P}_{h,d}}\|u-v_h\|_{H^{1}(\Omega)}
   +\|g_\Sigma-g_\Sigma^h\|_{H^{1/2}(\Sigma)}\big).
\end{equation}
\end{corollary}

\end{document}